\documentclass[11pt]{article}
\usepackage{cite}
\usepackage{hyperref}
\usepackage{appendix}
\usepackage{thmtools}  
\usepackage{amssymb,amsmath,amsfonts,mathrsfs,amsthm,epsfig,latexsym,color}
\usepackage{enumitem,geometry}
\usepackage{fourier}
\usepackage[all]{xy}

\makeatletter
\newcommand{\subsectionruninhead}{\@startsection{subsection}{2}{0mm}
	{-\baselineskip}{-0mm}{\bf\large}}
\newcommand{\subsubsectionruninhead}{\@startsection{subsubsection}{3}{0mm}
	{-\baselineskip}{-0mm}{\bf\normalsize}}
\makeatother

\newtheorem*{theorem*}{Theorem}
\newtheorem*{proof*}{Proof}
\newtheorem*{proposition*}{Proposition}
\newtheorem*{notation*}{Notation}
\newtheorem*{corollary*}{Corollary}
\newtheorem*{claim*}{Claim}
\newtheorem*{remark*}{Remark}

\newtheorem*{theorem1}{Theorem 1.1'}
\newtheorem*{theorem2}{Theorem 1.3'}
\newtheorem*{theorem3}{Theorem 2.1'}
\newtheorem{theorem}{Theorem}[section]
\newtheorem{proposition}{Proposition}[section]
\newtheorem{corollary}[theorem]{Corollary}
\newtheorem{lemma}[proposition]{Lemma}
\newtheorem{claim}[proposition]{Claim}

\theoremstyle{definition}
\newtheorem{definition}[proposition]{Definition}
\theoremstyle{remark}
\newtheorem{remark}[proposition]{Remark}

\numberwithin{equation}{section}

\def\NN{\mathbb{N}}
\def\RR{\mathbb{R}}

\def\TT{\mathbb{T}}
\def\ZZ{\mathbb{Z}}
\def\tildeL{\tilde{\mathcal{L}}}
\def\tildeF{\tilde{\mathcal{F}}}
\def\e{{\varepsilon}}
\def\va{{\varphi}}
\def\mathcalf{\mathcal{F}_f}
\def\mathcalg{\mathcal{F}_g}
\def\tildeF{\tilde{\mathcal{F}}_f}
\def\tildeG{\tilde{\mathcal{F}}_g}

\begin{document}
	\title{Topological and smooth classification of Anosov maps on torus}
	\author{Ruihao Gu \quad and \quad Yi Shi\footnote{Y. Shi was partially supported by National Key R\&D Program of China (2021YFA1001900) and NSFC (12071007, 11831001, 12090015).}}
	\date{\today}
	\maketitle
	
	\begin{abstract}
		In this paper, we give a complete topological and smooth classification of non-invertible Anosov maps on torus. We show that two non-invertible Anosov maps on torus are topologically conjugate if and only if their corresponding periodic points have the same Lyapunov exponents on the stable bundles. As a corollary, if two $C^r$ non-invertible Anosov maps on torus are topologically conjugate, then the conjugacy is $C^r$-smooth along the stable foliation. Moreover, we show that the smooth conjugacy class of a non-invertible Anosov map on torus is completely determined by the Jacobians of return maps at periodic points.
	\end{abstract}

	\section{Introduction}
     Let $M$ be a $C^{\infty}$-smooth $d$-dimensional closed Riemannian manifold and $f,g:M\to M$ be two maps. We say that  $f$ is \textit{topologically conjugate }to $g$ if there exists a homeomorphism $h:M\to M$ homotopic to identity Id$_M$ such that $h\circ f=g\circ h$. The topological conjugacy builds a dynamical equivalence of maps since two conjugate maps have exactly the same dynamical behavior up to a homeomorphism. However, it is almost impossible to classify  all maps without any restriction. For local topologically conjugacy, we say a $C^r$-map $f:M\to M$ is $C^r$-structurally stable if every $g:M\to M$ being $C^r$-close to $f$ is topologically conjugate to $f$. The most classical structurally stable systems are expanding maps and Anosov diffeomorphisms, which inspires us to give a full classification of them up to topological conjugacy.
     
     We say a local diffeomorphism $f:M\to M$ is an \textit{expanding map}  if the mini-norm $m(D_xf)>1$ for every $x\in M$. Shub \cite{Shub1969} proved that the universal cover of $M$ is Euclidean space $\RR^d$, and $f$ is topologically conjugate to an affine expanding endomorphism of an infra-nilmanifold if $\pi_1(M)$ is nilpotent. Later Franks \cite{Franks1970} showed that $\pi_1(M)$ has polynomial growth if $f:M\to M$ is an expanding map. Finally, Gromov \cite{Gro81} proved that a finitely generated group is virtually nilpotent, which implies that if $f:M\to M$ is expanding, then $M$ is an infra-nilmanifold and $f$ topologically conjugates to an affine expanding endomorphism. This gives a complete topological classification of expanding maps. Recently,  Gogolev and Hertz \cite{GH2022expanding} give a smooth classification of expanding maps via Jacobians of return maps at periodic points, see Remark \ref{1 rmk GH} for details.

     Recall that a diffeomorphism $f:M\to M$ is \textit{Anosov} if there exists a continuous $Df$-invariant splitting $TM=E^s \oplus E^u$ such that $Df$ is uniformly contracting in $E^s$ and uniformly expanding in $E^u$. The global classification of Anosov diffeomorphisms is widely open. If $f:M\to M$ is a codimension-one Anosov diffeomorphism, i.e., dim$E^s=1$ or dim$E^u=1$, then $M$ must be the $d$-torus $\TT^d$\cite{Franks1970,Newhouse1970}. If we restrict us on $d$-torus $\TT^d$ (or nilmanifold),  then every Anosov diffeomorphism $f:\TT^d\to \TT^d$ is topologically conjugate \cite{Franks1969, Manning1974} to its linear part $f_*:\pi_1(\TT^d)\to \pi_1(\TT^d)$. In particular,  this implies every Anosov diffeomorphism on a closed surface must topologically conjugate to a hyperbolic automorphism $A\in{\rm GL}_2(\ZZ)$ on torus $\TT^2$.

     In 1974,  Ma\~n\'e and Pugh \cite{ManePugh1975} introduced the definition of  Anosov maps which unifies the concepts of Anosov diffeomorphism and expanding map.
     \begin{definition}[\cite{ManePugh1975}]\label{1 def Anosov map manepugh}
     	Let $f:M\to M$ be a $C^1$ local diffeomorphism. We say $f$ is an \textit{Anosov map}, if there exists a continuous $Df$-invariant subbundle $E_f^s\subset TM$ such that $Df$ is uniformly contracting on $E_f^s$ and  uniformly expanding on the quotient bundle $TM/E_f^s$.
     \end{definition}
 
 \begin{remark}
 	Since the expanding direction of a point depends on its negative orbit, it is possible for  a non-invertible non-expanding Anosov map  to have different expanding directions at a point. See  \cite{Przytycki1976} for  examples with infinity expanding directions at certain points. 
 \end{remark}

     The simplest non-invertible Anosov map with non-trivial stable bundle is
     $$
     A_k=\begin{bmatrix}
     k&1\\
     1&1
     \end{bmatrix} :\mathbb{T}^2\to\mathbb{T}^2, \qquad \forall k\in\NN_{\geq 3}.
     $$
     Surprisingly, the linear Anosov map $A_k\ (k\geq3)$ can be $C^1$-approximated by Anosov maps which are not topologically conjugate to $A_k$. In fact, Ma\~n\'e, Pugh \cite{ManePugh1975} and Przytycki \cite{Przytycki1976} proved  the following theorem.
     
     \begin{theorem*}[\cite{ManePugh1975, Przytycki1976}]
     	A non-invertible Anosov map is $C^1$-structurally stable if and only if  it is an expanding map.
     \end{theorem*} 
      
     This theorem makes the topological classification of non-invertible Anosov maps with non-trivial stable bundles difficult (even in a small neighborhood), comparing with expanding maps and Anosov diffeomorphisms.
     
     \vspace{2mm}
     
     We introduce some symbols and notations.
     For $r\geq 1$ or $r=+\infty$, we denote by $\mathcal{A}^r(M)$ the set of $C^r$-smooth Anosov maps on $M$.  Denote by $\mathcal{N}^r(M)$ the subset of  $\mathcal{A}^r(M)$ which consists of  non-invertible maps. In particular, we use $k+\alpha$  in the superscript  such as  $\mathcal{A}^{k+\alpha}(M)$  to emphasize $k\in\NN$ and $0<\alpha<1$.  We mark the dimension of the stable bundle as the subscript, namely,  $\mathcal{A}^r_s(M)$ and $\mathcal{N}^{r+\alpha}_s(M)$ consist maps with $s$-dimensional stable bundles, for $0\leq s<d$.  We denote by ${\rm Home}_0(M)$ the set of homeomorphisms homotopic to identity on $M$.

     \subsection{Topological classification of Anosov maps}
    
    In this paper, we give a complete topological classification of non-invertible Anosov maps on torus.
    
   	\begin{theorem}\label{1 thm main thm}
    	Let $f,g \in\mathcal{N}^r(\TT^2)\ (r>1)$ be homotopic. Then $f$ is topologically conjugate to $g$, if and only if,  $f$ and $g$ have  same Lyapunov exponents on  stable bundles of  corresponding periodic points. 
    \end{theorem}
    
    \begin{remark}
    Here we explain how to match periodic points of two homotopic Anosov maps.
     If two Anosov maps are conjugate, it is natural to match their periodic points by the conjugacy. When there is no  \textit{a priori} conjugacy, we use the \textit{(stable) leaf conjugacy}. In fact, let $f,g :\TT^2\to \TT^2$ be two homotopic Anosov maps, then $f$ is leaf conjugate to $g$ \cite{HallHamerlindl2021} (also see Proposition \ref{2.2 prop leaf conjugate}), i.e., there exists $h\in{\rm Home}_0(\TT^2)$  such that 
     \begin{align}
     	h\circ f \big(\mathcalf^s(x)\big)= g \circ h \big(\mathcalf^s(x)\big),\quad \forall x\in\TT^2,\label{eq.1.leaf conjugacy}
     \end{align}
     where $\mathcalf^s(x)$  is the stable leaf of $f$ at  $x$(see Subsection \ref{subsec 2.1} for definition). 
     Since $f$ is uniformly contracting along every leaf $\mathcalf^s(x)$, every periodic stable leaf of $f$ admits a unique periodic point. Hence a leaf conjugacy $h$ induces a bijection between the sets of periodic points of  $f$ and $g$, denoted by 
     \begin{align}
     	h^P:{\rm Per}(f)\to {\rm Per}(g),\label{eq.1.hp}
     \end{align}
     and this gives us the way to match periodic points of $f$ and $g$. 
     \end{remark}

     For every periodic point $p$ of $f$, denote by $\lambda^s_f(p)$ the Lyapunov exponent on the stable bundle (\textit{stable Lyapunov exponent}) of $p$ for $f$. We say that $f$ and $g$ admit the same \textit{stable periodic data} with respect to a leaf conjugacy $h$ between $f$ and $g$, if the stable Lyapunov exponents of corresponding periodic points which induced by $h$ coincide, i. e.
     \begin{align}
     	\lambda^s_f(p)=\lambda^s_g\big( h^P(p) \big), \quad \forall p\in{\rm Per}(f), \label{eq.1.pds}
     \end{align} 
  where  $h^P$ is given by \eqref{eq.1.hp}. Hence Theorem \ref{1 thm main thm} means that the same stable periodic data implies that the leaf conjugacy can be deformed to a topological conjugacy.
  
   Theorem \ref{1 thm main thm} has the following interesting corollary, which exhibits the rigid phenomenon that "weak equivalence" (topological conjugacy) implies "strong equivalence"(smooth conjugacy). 
  %It also gives us a difference  in smooth dynamics between non-invertible Anosov  maps  and  Anosov diffeomorphisms.
  
  \begin{corollary}\label{1 cor h smooth along stable}
  	Let $f,g \in \mathcal{N}^r(\TT^2)\ (r>1)$. If $f$ is topologically conjugate to $g$, then the conjugacy is $C^r$-smooth along  the stable foliation.
  \end{corollary}
  
  \begin{remark}
  	Both Theorem \ref{1 thm main thm} and Corollary \ref{1 cor h smooth along stable} automatically hold for expanding maps since they have trivial stable bundles. But they do not hold for Anosov diffeomorphisms.
  \end{remark}
  
  \begin{remark}
  	For higher-dimensional torus we obtain same results for the case that $f$ is irreducible with one-dimensional stable bundle,  see Subsection \ref{subsec 2.2}. 
  \end{remark}
  
  \begin{remark}
  	Recently, Gogolev and Hertz have researched  the similar phenomenon of Corollary \ref{1 cor h smooth along stable}, usually called "bootstrap", for contact Anosov flows and codimension-one Anosov flows. They showed that under some additional assumption such as bunching conditions or exluding the flows being constant roof suspensions over Anosov diffeomorphisms, the topological conjugacy between two Anosov flows  must be smooth \cite{GH2021Anosovflow1,GH2021Anosovflow2}.
  \end{remark}
 
 \vspace{2mm}
  
  There are other ways to match the stable periodic data of $f$ and $g$.  Although there is no  \textit{a priori} conjugacy between two homotopic Anosov maps on $\TT^2$, there exist conjugacies on the level of  universal cover $\RR^2$, inverse limit spaces \cite{AokiHiraide1994} and also on the (stable) leaf spaces.  Moreover, we build connection among these three types of conjugacy and the leaf conjugacy in  \eqref{eq.1.leaf conjugacy} . Hence we can translate \eqref{eq.1.pds} into the corresponding stable periodic data via the above three different  types of conjugacy. Here we give a quick view of this as follow and one can get more details in Subsection \ref{subsec 2.3}.

  Let $f,g\in\mathcal{A}^1(\TT^2)$ be homotopic.  Denote by $F$ and $G$, any two liftings of $f$ and $g$ respectively by the natural projection $\pi:\RR^2\to\TT^2$.  Then there exists a conjugacy $H$  bounded from Id$_{\RR^2}$ between $F$ and $G$ (\cite{AokiHiraide1994}, also see Proposition \ref{2.1 prop lifting conjugate}).
   \begin{itemize}
   	\item For $\tau=f\ {\rm or}\ g$, let  $\TT^2_{\tau}$ be the inverse limit space and  $\sigma_{\tau}:\TT^2_{\tau}\to\TT^2_{\tau}$ be the (left) shift homeomorphism (see Subsection \ref{subsec 2.1} for definitions). 
    	\begin{itemize}
    	\item $H$  induces a conjugacy $\bar{h}$ between $\sigma_f$ and $\sigma_g$.
    \end{itemize}
   	\item For $\tau=f\ {\rm or}\ g$, denote the stable leaf space $S_{\tau}:=\TT^2/\mathcal{F}^s_{\tau}$, then $\tau$ induces the map  $\tau_s:S_{\tau}\to S_{\tau}$. 
   		\begin{itemize}
   		\item A leaf conjugacy homotopic to Id$_{\TT^2}$ in the sense of \eqref{eq.1.leaf conjugacy} induces a conjugacy $s$ between $f_s$ and $g_s$.
   		\item $H$ induces a conjugacy $s$ between $f_s$ and $g_s$.
   		\item  A conjugacy $s$ between $f_s$ and $g_s$ induces a conjugacy $H$ between some liftings $F$ and $G$. 
   	\end{itemize}
   \end{itemize} 
   As a result, we have the following  commutative diagram (Figure 1). In which,  $\pi_0:(\TT^2)^{\ZZ}\to \TT^2$ is  projection $\pi_0\big((x_i)\big)=x_0$ and $\bar{\pi}$  projects  orbits on $\RR^2$ to  orbits on $\TT^2$. Note that the leaf conjugacy cannot \textit{a priori} be commutative with $f$ and $g$.
    \[ \xymatrix@=7.4ex{
    \RR^2\ar@(dl,ul)^{F}	\ar[rrrrr]^{H} \ar[dr]^{\pi} 	\ar[ddd]_{\bar{\pi}} &&&&& \RR^2\ar@(dr,ur)_{G}  \ar[dl]_{\pi} 	\ar[ddd]^{\bar{\pi}}\\
   	&\TT^2\ar@(dl,ul)^{f} \ar[dr] \ar@{~)}[rrr]^{\rm leaf\ conjugacy} &&&\TT^2 \ar@(dr,ur)_{g} \ar[dl]  \\
   	&&S_f\ar@(ul,dl)_{f_s} \ar[r]^{s} &S_g\ar@(ur,dr)^{g_s} \\
    \TT^2_f \ar@(dl,ul)^{\sigma_f}  \ar[uur]_{\pi_0}  \ar[rrrrr]_{\exists\  \bar{h}}   &&&&& \TT^2_g \ar@(dr,ur)_{\sigma_g}  \ar[uul]^{\pi_0} 
   } 
   \]
 $$\text{Figure 1. The relation among four types of conjugacy.}$$
 
   \vspace{2mm}
   Now we explain Theorem \ref{1 thm main thm} in the sense of mapping spaces with the help of Figure 1. Recall that if $f\in\mathcal{A}^1(\TT^d)\ (d\geq 2)$, then its linearization $A=f_*:\pi_1(\TT^d) \to \pi_1(\TT^d)$ also induces an Anosov map on $\TT^d$ \cite[Theorem 8.1.1]{AokiHiraide1994}.

   Let $A\in GL_2(\RR)\cap M_2(\ZZ)$ which induces an non-invertible Anosov map on $\TT^2$. Denote by $\mathcal{A}^r(A)$ $(r>1)$ the set of  $C^r$-smooth Anosov maps homotopic to $A$.  Denote
   \begin{align}\label{A-maps}
    \tilde{\mathcal{A}}^r(A)=
        \big\{
        (f,H)~\big|&~f\in\mathcal{A}^r(A),~H:\RR^2\to\RR^2~
   \text{is a homeomorphism s.t.} \notag\\
       &\|H-{\rm Id}_{\RR^2}\|_{C^0}<\infty,~A\circ H=H\circ F
       ~\text{where}~F:\RR^2\to\RR^2~\text{is a lift of}~f \big\}.
   \end{align}
   We say two tuples $(f_1,H_1), (f_2,H_2)\in \tilde{\mathcal{A}}^r(A)$ are \textit{topologically equivalent on torus}, which denote by $(f_1,H_1)\sim_H (f_2,H_2)$, if  $H_2^{-1}\circ H_1: \RR^2\to\RR^2$ is commutative with deck transformations:
  $$
  H_2^{-1}\circ H_1(x+n)=H_2^{-1}\circ H_1(x)+n, 
  \quad \forall x\in\RR^2, \forall n\in\ZZ^2.
  $$ 
  This implies $H_2^{-1}\circ H_1$ can descend to $\TT^2$, i.e., there exists $h\in{\rm Home}_0(\TT^2)$ such that 
  $$
  \pi\circ \big(H_2^{-1}\circ H_1\big)=h\circ\pi,
  \qquad \text{and} \qquad
  h \circ f_1=f_2\circ h.
  $$ 
  Note that $h$ is \textit{a priori} only H\"older continuous. 
  
  We say $(f_1,H_1), (f_2,H_2)\in \tilde{\mathcal{A}}^r(A)$ are \textit{smooth equivalent along stable foliation on torus}, which denote by $(f_1,H_1)\sim_s (f_2,H_2)$, if we further assume the conjugacy $h\in{\rm Home}_0(\TT^2)$ is $C^r$-smooth along stable foliation of $f_1$.
  
  We denote by 
  $$
  \mathcal{T}^r_H(A)=\tilde{\mathcal{A}}^r(A)/_{\sim_H}
  \qquad \text{and} \qquad
  \mathcal{T}^r_s(A)=\tilde{\mathcal{A}}^r(A)/_{\sim_s}
  $$ 
  spaces of two equivalence classes of such tuples, which are Teichm\"uller spaces of Anosov maps by H\"older conjugacies, and by H\"older conjugacies which are $C^r$-smooth along stable foliations, respectively. 
  
   Denote by $\mathcal{F}^{\rm H}(\TT^2_A)$ the set of H\"older functions on $\TT^2_A$. Define an equivalence relation  $"\sim_{\sigma_A}"$ in $ \mathcal{F}^{\rm H}(\TT^2_A)$ as 
  $$\tilde{\phi}_1 \sim_{\sigma_A}\tilde{ \phi}_2 \iff \exists\;  {\rm H\ddot{o}lder} \;\tilde{u}:\TT^2_A\to \RR {\rm \;\; such\;\; that\;\;} \tilde{\phi}_1-\tilde{\phi}_2=\tilde{u}\circ \sigma_A-\tilde{u}.$$
  This means $\tilde{\phi}_1$ and $\tilde{ \phi}_2$ are cohomologous via $\sigma_A$ on $\TT^2_A$. 
  
  \vspace{2mm}
  
  We can restate Theorem \ref{1 thm main thm} and Corollary \ref{1 cor h smooth along stable} as following.   As mentioned before, for any $(f,H)\in  \tilde{\mathcal{A}}^r(A)$, $H$ can induce a conjugacy $\bar{h}:\TT^2_f\to \TT^2_A$ between $\sigma_f$ and $\sigma_A$ (also see  Proposition \ref{2.2 prop: the relationship of 3 conjugacies inverse limit space}).
  \begin{theorem1}\label{thm 1.1'}
  	Let $A\in GL_2(\RR)\cap M_2(\ZZ)$ which induces a non-invertible Anosov map on $\TT^2$, and 
  	$(f_i,H_i)\in \tilde{\mathcal{A}}^{r}(A)\ (r>1)$ for $i=1,2$. Then  
  	$$ 
  	(f_1,H_1)\sim_H (f_2,H_2) ~\iff~  
  	{\rm log}\|Df_1|_{E^s_{f_1}}\|\circ \bar{h}_1^{-1} \sim_{\sigma_A} {\rm log}\|Df_2|_{E^s_{f_2}}\|\circ \bar{h}_2^{-1}
  	~\iff~ (f_1,H_1)\sim_s (f_2,H_2),
  	$$
  	where $\bar{h}_i$ is the conjugacy induced by $H_i$  on $\TT^2_A$. 
  	In particular, there  is   a natural injection
   \begin{align}
  	\mathcal{T}^r_H(A)~=~
  	\mathcal{T}^r_s(A)~\hookrightarrow~ 
  	\mathcal{F}^{\rm H}(\TT^2_A)\big/_{\sim_{\sigma_A}}
  	\qquad \text{and} \qquad
  	(f,H)~\mapsto~{\rm log}\|Df|_{E^s_f}\|\circ \bar{h}^{-1}.
  	\label{eq.1. space injection}
  \end{align}
   Here $\bar{h}:\TT^2_f\to \TT^2_A$ is the conjugacy induced by $H$.
  \end{theorem1}

 \begin{remark}
  The injection \eqref{eq.1. space injection} also holds for expanding maps, if one defines {\rm log}$\|Df|_{E^s_f}\|$ as $-\infty$. For a precise explication, we will give more details of Theorem \ref{thm 1.1'}' later in  Subsection \ref{subsec 2.2}.
 \end{remark}

   \subsection{Smooth classification of Anosov maps}

   Based on the topological classification stated in Theorem \ref{1 thm main thm}, we give a smooth classification  of non-invertible Anosov maps on torus with non-trivial stable bundles via Jacobians of return maps at periodic points. 
   For $r>1$, let  $r_*=\Big\{ \begin{array}{lr}
   r-\e, \; \ r\in\NN  \\ \;\;\; r,\ \ \ \ \ \ \  r\notin  \NN \ {\rm or}\ r=+\infty
   \end{array}$, where $\e>0$ can be arbitrarily small.  
   
    \begin{theorem}\label{1 thm jacobian rigidty}
    	Assume $f,g \in\mathcal{N}^r_1(\TT^2)\ (r>1)$ and $f$ is topologically conjugate to $g$ by $h\in{\rm Home}_0(\TT^2)$. If 
    	$$
    	{\rm Jac}\big(f^n\big)(p)={\rm Jac}\big(g^n\big)(h(p)),
    	\qquad \forall p=f^n(p),
    	$$
    	then $h$ is a $C^{r_*}$-smooth diffeomorphism.
    \end{theorem}

 		\begin{remark}
 		From Corollary \ref{1 cor h smooth along stable},  the topological conjugacy has already been smooth along  stable foliation. We just need to prove regularity of  conjugacy restricted on the unstable direction by  method  originated in \cite{Llave1992,GG2008}, see also \cite{SY2019}.
 	\end{remark}

 	\begin{remark}\label{1 rmk GH}
 	Gogolev and Hertz  discoverd surprisingly that for two conjugate irreducible $C^{r}$-smooth $(r\geq 2)$   expanding maps $f$ and $g$ on $d$-torus $(d\geq 2)$, if they have the same periodic Jacobian data and the Jacobian determinant of $f$ is not cohomology to a constant, then $f$ is $C^{r-1}$-smoothly conjugate to $g$ (see \cite[Corollary 2.6]{GH2022expanding}). We would like to point out a little difference in Theorem \ref{1 thm jacobian rigidty} in which  the Jacobian determinant of the map  being cohomology to a constant would be allowed.
 	\end{remark}

 	\begin{remark}
 	 We mention that \cite{Micena2022} gives a smooth classification of topologically conjugate Anosov maps which have unstable subbundles on $\TT^2$, via both stable and unstable Lyapunov exponents at  periodic points. Our work also deals with the case having no unstable subbundles which brings us an obstacle.
    \end{remark}

\vspace{2mm}
   We also give a view of Theorem \ref{1 thm jacobian rigidty} in the sense of mapping spaces.
   	For any $f\in\mathcal{A}_1^r(\TT^2)\ (r>1)$, let
   	$$ 
   	\tilde{\mathcal{A}}^r_{\TT^2}(f)=
   	\left\{(g,h)~\big|~ g\in\mathcal{A}_1^r(\TT^2),~h\in{\rm Home}_0(\TT^2),~ 
   	\text{s. t.}~h\circ f= g\circ h \right\}
   	$$
   	We say that two tuples $(g_1,h_1), (g_2,h_2)\in \tilde{\mathcal{A}}^r_{\TT^2}(f)$ are $C^{r_*}$-\textit{equivalent}, denoted by $(g_1,h_1)\sim_{r_*}(g_2,h_2)$, if the homeomorphism 
   	$h_1\circ h_2^{-1}\in{\rm Home}_0(\TT^2)$ is $C^{r_*}$-smooth. 
   	
   	The Teichm${\rm \ddot{u}}$ller space $\mathcal{T}^{r_*}_{\TT^2}(f)$ is the space of "$\sim_{r_*}$"-equivalence classes of such tuples 
   	\begin{align}
   		\mathcal{T}^{r_*}_{\TT^2}(f)= \tilde{\mathcal{A}}^r_{\TT^2}(f)/_{\sim_{r_*}}. \label{eq. 1. teichmuller}
   	\end{align}
   Denote $\mathcal{F}^{\rm H}(\TT^2)$ the set of {\rm H${\rm \ddot{o}}$lder} functions on $\TT^2$. Define the equivalence relation "$\sim_f$" in  $\mathcal{F}^{\rm H}(\TT^2)$  as 
   \begin{align}
   	 \phi_1 \sim_f \phi_2 \iff \exists\;  {\rm H\ddot{o}lder} \;u:\TT^2\to \RR {\rm \;\; such\;\; that\;\;} \phi_1-\phi_2=u\circ f- u. \label{eq. 1. coboundary}
   \end{align}
   Since Livschitz Theorem holds for Anosov maps on torus (see Subsection \ref{subsec 2.4}), Theorem \ref{1 thm jacobian rigidty} immediately implies  the following theorem.  
   
   \begin{theorem2}\label{thm 1.2'}
   	 Let $f\in\mathcal{N}_1^r(\TT^2)\ (r>1)$ and $(g_i,h_i)\in \tilde{\mathcal{A}}^r_{\TT^2}(f)$ for $i=1,2$.  Then 
   	$$
   	(g_1,h_1)\sim_{r_*}(g_2,h_2)\iff {\rm log}\big({\rm Jac}(g_1)\circ h_1\big) \sim_f {\rm log}\big({\rm Jac}(g_2)\circ h_2\big).
   	$$ 
   	In particular, there is a natural injection
   	$$
   	\mathcal{T}^{r_*}_{\TT^2}(f)~\hookrightarrow~
   	\mathcal{F}^{\rm H}(\TT^2)\big/_{\sim_f}
   	\qquad \text{and} \qquad
   	(g,h)~\mapsto~{\rm log}\big({\rm Jac}(g)\circ h\big).
   	$$
   \end{theorem2}

   Finally, we can give a global picture of Anosov maps on $\TT^2$ up to smooth conjugacy in a fixed homotopy class. Let $A\in GL_2(\RR)\cap M_2(\ZZ)$ which induces a linear Anosov map in $\mathcal{N}_1^r(\TT^2)$. For instance, we can take
   $$
   A=\begin{bmatrix}
   	k&1\\
   	1&1
   \end{bmatrix}, \qquad \forall k\in\NN_{\geq 3}.
   $$
   
   Let $\tilde{\mathcal{A}}^r(A)$ be the set of tuples $(A,H)$, where $f\in\mathcal{A}^r(A)$ is an Anosov map homotopic to $A$ and $H$ is the conjugacy in $\RR^2$, see \ref{A-maps}. We say two tuples $(f_1,H_1), (f_2,H_2)\in \tilde{\mathcal{A}}^r(A)$ are \textit{$C^{r_*}$-smooth equivalent on torus}, which denote by $(f_1,H_1)\sim_{r_*} (f_2,H_2)$ if $H_2^{-1}\circ H_1$ is $\ZZ^2$-periodic and $C^{r_*}$-smooth. This implies there exists a $C^{r_*}$-diffeomorphism $h:\TT^2\to\TT^2$ such that
   $$
   \pi\circ \big(H_2^{-1}\circ H_1\big)=h\circ\pi,
   \qquad \text{and} \qquad
   h \circ f_1=f_2\circ h.
   $$ 
   The Teichm\"uller space of Anosov maps homotopic to $A$ is the equivalence class
   $$
   \mathcal{T}^{r_*}(A)=\tilde{\mathcal{A}}^r(A)/_{\sim_{r_*}}.
   $$
   
   Combining Theorem 1.1' and Theorem 1.3', we have the following corollary.
   
   \begin{corollary}
   	   Let $A\in GL_2(\RR)\cap M_2(\ZZ)$ which induces a linear Anosov map in $\mathcal{N}_1^r(\TT^2)$. Then
   	   $$
   	   \mathcal{T}^{r_*}(A)~=~
   	   \mathcal{T}^r_H(A)\ltimes_f\mathcal{T}^{r_*}_{\TT^2}(f)~:=~
   	   \bigcup_{\langle(f,H)\rangle\in\mathcal{T}^r_H(A)}\mathcal{T}^{r_*}_{\TT^2}(f)
   	   $$
   	   Here $\langle(f,H)\rangle$ is the equivalence class of $(f,H)$ in $\mathcal{T}^r_H(A)$ by H\"older conjugacy. In particular, there exists a natural injection
   	   $$
   	   \mathcal{T}^{r_*}(A)~\hookrightarrow~ 
   	   \mathcal{F}^{\rm H}(\TT^2_A)\big/_{\sim_{\sigma_A}}
   	   \ltimes_f\mathcal{F}^{\rm H}(\TT^2)\big/_{\sim_f}.
   	   $$
   \end{corollary}

  \begin{remark}\label{1 rmk teichmuller of diffeo}
   	The Teichm${\rm \ddot{u}}$ller space of an Anosov diffeomorphism on $\TT^2$ already has a complete characterization in \cite{Cawley1993,Llave1987,1MM1987,2MM1987,Llave1992}.   Let $f:\TT^2\to \TT^2$ be a $C^r$-smooth Anosov diffeomorphism.  Denote the Teichm${\rm \ddot{u}}$ller space of $f$ consisting $C^r$-equivalence classes  by $\mathcal{T}_{\TT^2}^{r}(f)$ in the same sense of \eqref{eq. 1. teichmuller}.
   	Define the equivalence relation "$\sim_f$"   in  $\mathcal{F}^{\rm H}(\TT^2)$  as  same as \eqref{eq. 1. coboundary}.
   	Define the SRB map on $\mathcal{T}_{\TT^2}^{r}(f)$  as 
   	\begin{align*}
   		{\rm SRB}:\ \  \mathcal{T}_{\TT^2}^{r}(f) &\longrightarrow \mathcal{F}^{\rm H}(\TT^2)\big/_{\sim_f}\times \mathcal{F}^{\rm H}(\TT^2)\big/_{\sim_f},\qquad \\
   		\langle(g,h)\rangle&\longmapsto \Big(\langle-{\rm log}\big( {\rm Jac} (g|_{E^u_g})\big)\circ h\rangle\ ,\ \langle{\rm log}\big( {\rm Jac}  (g|_{E^s_g})\big)\circ h\rangle\Big),
   	\end{align*}
    where  $\langle(g,h)\rangle$ and $\langle\phi\rangle$ are equivalence class in $\mathcal{T}_{\TT^2}^{r}(f)$ and $\mathcal{F}^{\rm H}(\TT^2)\big/_{\sim_f}$ respectively. This map is well defined by Livschitz Theorem. It has been proved  in \cite{Llave1987,1MM1987,2MM1987,Llave1992} that the SRB map is an injection. In the case of $r=1+{\rm H}$ that is $r=1+\alpha$ for some $0<\alpha<1$,  \cite{Cawley1993} showed that the SRB map is simultaneously a surjection if one considers the equivalence relation "$\sim_f$" in  $\mathcal{F}^{\rm H}(\TT^2)$ as "almost coboundary" which allows a constant error on coboundary \eqref{eq. 1. coboundary}, see more details in \cite{Cawley1993}.  A natural question is that in the sense of "almost coboundary", the injection in Theorem \ref{thm 1.2'}' is also surjective or not?
    \end{remark}

		\vskip 0.5 \baselineskip
	
		\noindent {\bf Organization of this paper:}
	In section \ref{sec 2}, we recall some general properties of Anosov maps,   give some useful properties on the assumptions of  Theorem \ref{1 thm main thm} and sketch our proof of  Theorem \ref{1 thm main thm}. 
	In section \ref{sec 3}, we prove the "sufficient" part of  Theorem \ref{1 thm main thm} which states that  the same stable periodic data implies	the existence of conjugacy.
	In section \ref{sec accessible}, we give a dichotomy to the toral Anosov maps in a sense of accessibility and this is helpful for the "necessary" part  of Theorem \ref{1 thm main thm}.
	In section \ref{sec 4}, we prove the "necessary" part  of Theorem \ref{1 thm main thm} and also show Corollary \ref{1 cor h smooth along stable}.
	In section \ref{sec 5}, we prove Theorem \ref{1 thm jacobian rigidty}.
	
	\vskip 0.5 \baselineskip

	\section{Topological and leaf conjugacies of Anosov maps}\label{sec 2}

	\subsection{Anosov maps}\label{subsec 2.1}
     There are two classical ways to study an Anosov map. One is using the inverse limit space and the other one is observing on the universal cover space. In this subsection, we recall these two methods and some  useful propositions for this paper.

	We first introduce the dynamics on the inverse limit space. Let $(M,d)$ be a compact metric space and $M^{\mathbb{Z}}:=\{(x_i) \;|\;x_i\in M,  \forall i\in \mathbb{Z} \}$ be the product topological space.  It is clear that $M^{\mathbb{Z}}$ is compact and metrizable by  metric
	$$
	\bar{d}((x_i),(y_i))= \sum_{-\infty}^{+\infty} \frac{d(x_i,y_i)}{2^{|i|}}.
	$$
	For any $\bar{x}=(x_i)\in M^{\ZZ}$, let $(\bar{x})_0:=x_0$ and $\pi_0:M^{\ZZ}\to M$ be the projection $\pi_0\big((x_i)\big)=x_0$. Let $\sigma: M^{\mathbb{Z}}\to M^{\mathbb{Z}}$ be the \textit{(left) shift} homeomorphism by
	$\big(\sigma (x_i)\big)_j=x_{j+1}$, for all $j\in \mathbb{Z}$.
	For a continuous map $f:M\to M$, define the \textit{inverse limit space} of $f$ as
	\begin{align}
		M_f:=\big\{(x_i) \;|\; x_i\in M \;\; {\rm and} \;\; f(x_i)=x_{i+1},  \forall i\in \mathbb{Z} \big\}.\label{eq. 2.1 inverse limit space}
	\end{align}
	Denote the restriction of $\sigma$ on $M_f$ by $\sigma_f$. It is clear that the inverse limit space	$(M_f,\bar{d})$ is a $\sigma_f$-invariant and  compact metric space.

	\begin{definition}[\cite{Przytycki1976}]\label{2.1 defprz}
		A $C^1$ local diffeomorphism $f:M\to M$ is called Anosov map, if there exist constants $C>0$
		and $0<\mu<1$ such that for every $\tilde{x}=(x_i) \in M_f$,
		there exists a splitting
		\begin{align*}
			T_{x_i}M=E_f^s(x_i,\tilde{x})\oplus E_f^u(x_i,\tilde{x}),\quad  \forall i\in \mathbb{Z},
		\end{align*}
		which is $Df$-invariant
		\begin{align*}
			D_{x_i}f\left(E_f^{s}(x_i,\tilde{x})\right)
			=E_f^s(x_{i+1},\tilde{x})
			\qquad{\rm and}\qquad D_{x_i}f\left(E_f^{u}(x_i,\tilde{x})\right)
			=E_f^u(x_{i+1},\tilde{x}),
			\qquad\forall i\in\mathbb{Z},
		\end{align*}
		and for all $n>0$, one has that
		\begin{align*}
			&\|D_{x_i}f^n (v)\|\le C\mu^n\|v\|, \qquad \quad \forall v\in E_f^s(x_i,\tilde{x}), \;\;\forall i\in\mathbb{Z},\\
			&\|D_{x_i}f^n (v)\|\ge C^{-1}\mu^{-n}\|v\|, \;\;\quad \forall v\in E_f^u(x_i,\tilde{x}),\;\;\forall i\in\mathbb{Z}.
		\end{align*}
	\end{definition}
	
	By the equivalent Definition \ref{1 def Anosov map manepugh} or the view of cone-field (see \cite{AGGS2022}), one has that the stable bundle $E^s_f$ is independent of the choice of orbit for a typical point. Meanwhile, the unstable direction relies on the negative orbits. We say $f\in\mathcal{A}^1(M)$ is \textit{special}, if the unstable directions on each point are independent of the negative orbits, i.e., there exists a $Df$-invariant continuous splitting $TM=E^s_f\oplus E^u_f$ such that $Df$ is uniformly constracting in $E_f^s$ and uniformly expanding in $E^u_f$. Note that Anosov diffeomorphisms and expanding maps are always special.

	Since every Anosov map on $\TT^d$ is  transitive (\cite{AokiHiraide1994}, also see \cite[Section 4]{AGGS2022}), it follows  from \cite{Przytycki1976, MicenaTahzibi2016} that  non-special Anosov maps are $C^1$-dense in non-invertible Anosov maps with  non-trivial stable bundles and each of them admits  a  residual set in which each point  has infinitely many unstable directions. Especially, let $f\in\mathcal{A}^1(\TT^d)$ be non-special, then $f$ is not conjugate to its linearization $f_*:\pi_1(\TT^d) \to \pi_1(\TT^d)$.  In fact, one has the following  proposition.
	
	\begin{proposition}[\cite{MoosaviTajbakhsh2019}]\label{2.1 prop special and conjugate}
		Let $f\in\mathcal{A}^1(\TT^d)$, then $f$ is conjugate to its linearization $f_*:\pi_1(\TT^d) \to \pi_1(\TT^d)$,
		if and only if, $f$ is special.
	\end{proposition}
	
	Despite all this,  there exists a conjugacy on the level of inverse limit spaces for toral Anosov maps. For convenience, we state it as follow.
	\begin{proposition}[\cite{AokiHiraide1994}]\label{2.1 prop: conjugacy on the inverse limit spaces}
		Let $f\in\mathcal{A}^1(\TT^d)$ and $A=f_*$. Then there exists a homeomorphism $\tilde{h}:\TT^d_f \to \TT^d_A$ such that $\tilde{h} \circ \sigma_f = \sigma_A \circ \tilde{h}$. 
	\end{proposition}

	For $f\in \mathcal{N}^1(\TT^d)$ and $x\in\TT^d$, define the preimage set $P(x,f)$ of $x$ by
	\begin{align}
		P(x,f)=\bigcup_{k\geq 0} P_k(x,f), \quad {\rm where} \quad P_k(x,f)=\{y\in\TT^d\;|\; f^k(y)=f^k(x)   \}.\label{eq. 1.1. preimage set}
	\end{align}
	We also have the stable foliation just as the case of Anosov diffeomorphism and each leaf of it is independent of the choice of orbits too. 
	Note that in this sense we cannot define the stable leaf of a point $x$ by collecting points whose positive orbits closed to one of $x$, since we need exclude its $k$-preimage set $P_k(x,f)$  for any $k\in\NN$.  So we should redefine the \textit{stable leaf} in strong sense (\cite[Section 5.1]{AokiHiraide1994}) by linking the \textit{local stable leaves} with small size $R>0$ 
	$$\mathcal{F}_{f,R}^s(x)=\big\{ y\in\ M \;|\; d\big(f^i(x),f^i(y) \big)<R, \; \forall i\geq 0 \big\}.$$

	We also have the (local) unstable leaf, but it relies on the choice of negative orbits. Let $\tilde{x}=(x_i)\in M_f$ and $R>0$, we define the \textit{local unstable leaf} of $x_0$ for $\tilde{x}$ with size $R$ by 
	$$\mathcal{F}_{f,R}^u(x_0,\tilde{x})=\big\{ y\in M\;\big|\; \exists (y_i)\in M_f \;\;{\rm with}\;\;  y_0=y \;\;{\rm such\; that}\;\; d(x_{-i},y_{-i})<R, \; \forall i\leq0      \big\}.$$
	It is clear that the stable leaf $\mathcalf^s(x)$ and the local unstable leaf $\mathcal{F}_{f,R}^u(x_0,\tilde{x})$ are both submanifolds and tangent to $E^s_f$ and $E^u_f(x_0,\tilde{x})$ respectively ( see \cite{Przytycki1976, QXZ2009}). Moreover, as the case of Anosov diffeomorphisms, the local stable leaves of an Anosov map form a family of embeddings varing continuously with $x\in M$ in $C^1$-topology. In particular, one has a similar property  for  local unstable leaves as follow.

	\begin{proposition}[\cite{Przytycki1976}]\label{2.1 prop: unstable leaf continuity}
		Let $\tilde{x}^k=(x_i^k)_{i\in\ZZ}\in M_f (k\in\NN)$ be a sequence of orbits. If $\tilde{x}^k\to \tilde{y}=(y_i)\in M_f$ as $k\to+\infty$, then for any $R>0$, $$\mathcal{F}_{f,R}^u(x_0^k,\tilde{x}^k)\longrightarrow \mathcal{F}_{f,R}^u(y_0,\tilde{y}),$$ in $C^1$-topology as $k\to+\infty$. In particular, $E^u_f(y_0,\tilde{y})$ is continuous with respect to $(y_0,\tilde{y})\in M\times M_f$.
	\end{proposition}
	
	We also have the \textit{Local Product Structure} for Anosov maps (see \cite[Section 5.2]{AokiHiraide1994} and \cite[Section IV.2]{QXZ2009}) as follow.
	
	\begin{proposition}[\cite{QXZ2009}]\label{2.1 prop: local product sturcture}
		There exist $\delta>0$ and $0<\e<\delta$ such that for any $x_0\in M$ and any $\tilde{y}=(y_i)\in M_f$, if $d(x_0,y_0)<\e$, then $\mathcal{F}_{f}^u(y_0,\tilde{y},\delta)$ transversely intersects with $\mathcal{F}_{f}^s(x_0,\delta)$ at a unique point $z(x_0,\tilde{y})$.  In particular, $z(x_0,\tilde{y})$ is continuous with respect to $(x_0,\tilde{y})\in M \times M_f$. Moreover, there exists a unique orbit $\tilde{z}=(z_i)\in M_f$ such that $z_0=z(x_0,\tilde{y})$ and $d(z_{-i},y_{-i})<\delta$ for all $i\geq 0$.
	\end{proposition}

	Although  the inverse limit space has compactness which the universal cover  lacks, observing the dynamics on the universal cover  has some added convenience. In fact,  Ma$\rm \tilde{n} \acute{e}$ and Pugh \cite{ManePugh1975} proved that the lifting on universal cover of an Anosov map on a closed Riemannian manifold is an Anosov diffeomorphism. Moreover, when we restrict on $\TT^d$, the liftings on $\RR^d$ of an Anosov map  and its linearization are conjugate  \cite[Proposition 8.2.1 and Proposition 8.4.2]{AokiHiraide1994}. For convenience, we restate them as the following proposition.
	
	In the rest part of this subsection, we always assume that $f\in\mathcal{A}^1(\TT^d)$   and $A:=f_*:\pi_1(\TT^d)\to\pi_1(\TT^d)$ is its linearization. Let $F:\mathbb{R}^d \to \mathbb{R}^d$ be a lift of $f$ by the  natural projection $\pi:\mathbb{R}^d\to \mathbb{T}^d$.  For short, we still denote  by $A:\RR^d\to \RR^d$ a lift of $A$ by $\pi$, if there is no confusion.  Recall that if $f\in\mathcal{A}^1(\TT^d)$, then its linearization $A\in \mathcal{A}^1(\TT^d)$ \cite[Theorem 8.1.1]{AokiHiraide1994}.
	
	\begin{proposition}[\cite{AokiHiraide1994}]\label{2.1 prop lifting conjugate}
	There is a unique bijection $H:\mathbb{R}^d\to \mathbb{R}^d$ such that
		\begin{enumerate}
			\item $H\circ F=A\circ H$.
			\item $H$ and $H^{-1}$ are both uniformly continuous.
			\item There exists $C>0$ such that $\|H-{\rm Id}_{\RR^d}\|_{C^0}<C$ and $\|H^{-1}-{\rm Id}_{\RR^d}\|_{C^0}<C$.
		\end{enumerate}
	\end{proposition}
	
	For any $n\in\ZZ^d$, let $T_n:\RR^d\to \RR^d$ be the deck transformation (or $\ZZ^d$-action) $T_n(x)=x+n$. By Proposition \ref{2.1 prop special and conjugate}, one has that  $f$ is special  if and only if  $H$ is \textit{commutative with deck transformation}, namely, $$H(x+n)=H(x)+n,\quad \forall x\in\RR^d \;\;\;{\rm and}\;\;\; \forall n\in\ZZ^d.$$

	Recall that  the  \textit{stable manifolds} of the lifting $F$, denoted by $\tildeF^s(x)$ has the following topological characterization
	$$ \tildeF^s(x):=\big\{y\in \RR^d \; \big|\;  d\big(F^k(y),F^k(x)\big)\to 0 \;\;{\rm as}\;\; k\to +\infty \big\},$$
	for all $x\in \RR^d$, 
	and the  \textit{unstable manifolds} $\tildeF^u(x)$ by iterating backward.   Note that the two foliations $\tildeF^s$ and $\tildeF^u$ admit the \textit{Global Product Structure}, namely any two leaves $\tildeF^s(x)$ and $\tildeF^u(y)$ transversely intersect at a unique point in $\RR^d$.

	For short, we denote the stable/unstable bundles and foliations of the linearization $A$ on $\TT^d$ by $L^{s/u}$, $\mathcal{L}^{s/u}$ and on $\RR^d$ by $\tilde{L}^{s/u}$, $\tildeL^{s/u}$ respectively. 	In the rest part of this paper, the \textit{local (un)stable} manifolds with size $\delta$ of $F$ denoted by  $\tildeF^{s/u}(x,\delta)$ means
	$$\tildeF^{s/u}(x,\delta):=\big\{y\in \tildeF^{s/u}(x) \; \big|\;  d_{\tildeF^{s/u}}(x,y)\leq \delta \big\}, $$
	for all $x\in \RR^d$,  where $d_{\tildeF^{s/u}}(\cdot,\cdot)$ is induced by the metric on $\RR^d$.  And we also denote by $\mathcalf^s(x_0,\delta)$ and $\mathcalf^u(x_0,\tilde{x},\delta)$, the local stable and unstable leaves of $f$ for $(x_i)=\tilde{x}\in \TT^d_f$ with size $\delta$ measured by metrics along leaves.

	It is clear that $H$ maps the stable/unstable leaf of $F$ to one of $A$ from the topological character of  stable/unstable  leaf, namely,
	\begin{align}
		H\big(\tildeF^{s/u}(x)\big)=\tildeL^{s/u}\big(H(x)\big),\quad \forall x\in\RR^d. \label{eq. 2.1.1}
	\end{align} 
	We denote the map in the stable leaf space of $F$ by $F_s:\RR^d\big/ \tildeF^{s} \to \RR^d\big/ \tildeF^{s}$ with $F_s\big( \tildeF^s(x)\big)=\tildeF^s\big(F(x)\big)$.
	It follows from \eqref{eq. 2.1.1} that $H$ induces a conjugacy between the stable leaf spaces,	
	\begin{align}
		H_{s}:\RR^d\big/ \tildeF^{s} \to :\RR^d\big/ \tildeL^{s},\label{eq.2.1.leaf space conjugacy}
	\end{align}
	with $H_s\big(\tildeF^s(x) \big)=\tildeL^s\big(H(x)\big)$ such that $H_s\circ F_s= A_s\circ H_s$. 
	
	Although in general $H$ cannot be commutative with deck transformation, the next proposition says that $H_s$  is always commutative with deck transformation.

	\begin{proposition}[\cite{AGGS2022}]\label{2.1 prop H is Zd restricted on stable}
		Assume that $H$ is given by Proposition \ref{2.1 prop lifting conjugate}. Then for any $x\in\RR^d$ and $n\in\ZZ^d$,
		$$H(x+n)-n\in \tildeL^s\big(H(x)\big) \quad {\rm and } \quad H^{-1}(x+n)-n\in \tildeF^s\big(H^{-1}(x)\big).$$
	\end{proposition}
	
	By proposition \ref{2.1 prop H is Zd restricted on stable}, one can get the following proposition which means that $H$ is "asymptotically" commutative with the deck transformation $T_n$, for $n\in A^m\ZZ^d$ with $m\to+\infty$.
	
	\begin{proposition}[\cite{AGGS2022}]\label{2.1 prop nmH}
		Assume that $H$ is given by Proposition \ref{2.1 prop lifting conjugate}.  There exist $C>0$ and  $\{\varepsilon_m\}$ with $\varepsilon_m\to 0$ as $m\to +\infty$, such that for every $x\in\RR^d$ and every $n_m\in A^m\ZZ^d$, one has
		$$|H(x+n_m)-H(x)-n_m|<C\cdot \| A|_{L^s}\|^m,$$
		and
		$$|H^{-1}(x+n_m)-H^{-1}(x)-n_m|<\varepsilon_m.$$
	\end{proposition}     
	
	\begin{remark}
		Obviously, let $n_m=|{\rm det}(A)|^m n$ for any $n\in\ZZ^d$ and $m\in\NN$, one has $A^{-i}n_m\in\ZZ^d$ for all $1\leq i\leq m$. This observation would be used in  Section \ref{sec 3}.
	\end{remark}
	
	To end this subsection, we state two properties about  stable and unstable foliations of Anosov maps whose proofs are quite similar to the case of Anosov diffeomorophisms. 
	
\begin{proposition}\label{2.1 prop C1 foliation}
	The stable and unstable foliations $\tildeF^s$ and $\tildeF^u$ are both  absolutely continuous.
	Moreover, if $\tildeF^{s/u}$ is codimension-one, then it is $C^1$-smooth.
	\end{proposition}
	The proof of this proposition is a local  analysis  just like one of \cite[Theorem 7.1]{Pesinbook2004}. We mention that one may use the compactness about the dynamics of $F$ which is from one of $\sigma_f$ by projection.

	\begin{proposition}\label{2.1 prop quasi-isometric}
	If $\tildeF^{s/u}$ is one-dimensional foliation, then it is quasi-isometric, i.e., 
	there exist constants $a,b>0$ such that $d_{\tildeF^{s/u}}(x,y)<a\cdot d(x,y)+b$ for any $x\in\RR^d$ and $y\in \tildeF^{s/u}(x)$, i.e. there exists $C>1$ such that $d_{\tildeF^{s/u}}(x,y)<C \cdot d(x,y)$ for any $x\in\RR^d$ and $y\in \tildeF^{s/u}(x)$.
	\end{proposition}
	
	\begin{proof}
	We refer to \cite{BBI2009} for proof of the diffeomorphism case and one can find an adaptation for case of $\tildeF^s$ in  \cite{AGGS2022}. However, since the unstable foliation $\tildeF^u$ may not be commutative with deck transformations,  we will show this case.
	We first  claim that for any $C_1>0$ there exists $R>0$ such that for any $y\in\RR^d$ and $z\in\tildeF^u(y)$, if $d(y,z)<C_1$, then $d_{\tildeF^u}(y,z)<R$. For the coherence of proofs of both this proposition and later one, we leave the proof of the claim to Claim \ref{4  claim: uniform control for unstable distance}.
	
   Let $H$ be the conjugacy between $F$ and $A$ given by   Proposition \ref{2.1 prop lifting conjugate}. In particular, $|H-{\rm Id}_{\RR^d}|<C_0$.
	Since $H$ maps $\tildeF^{s/u}$ to $\tildeL^{s/u}$, one has that $\tildeF^s(x)$ is  contained in the $2C_0$-neighborhood of $ \tildeL^s(x)$, namely, $\tildeF^{s/u}(x)\subset B_{2C_0}\big(  \tildeL^{s/u}(x)\big)$ for any $x\in\RR^d$.
	Fix $n\in\ZZ^d$ such that $d(x+n,x)> 5C_0$ for all $x\in\RR^d$. Then one hsa that $d_H\big(  \tildeF^s(x), \tildeF^s(x+n)  \big)> C_0$ for all $x\in\RR^d$ where $d_H(\cdot,\cdot)$ is the Hausdorff distance.
	It follows from  the Global Product Structure and the claim above that there exists $L_0>0$ such that $\tildeF^u(x,L_0)$ intersects $\tildeF^s(x+n)$ at point $y$ such that $d_H\big(\tildeL^s(x), \tildeL^s(y) \big)>C_0$.
	By cutting the curve of unstable leaf in $L_0$-length, this implies that $\tildeF^u$ is quasi-isometric, i.e., 
	there exist constants $a,b>0$ such that $d_{\tildeF^{s/u}}(x,y)<a\cdot d(x,y)+b$ for any $x\in\RR^d$ and $y\in \tildeF^{s/u}(x)$.
	
	Since $\tildeF^u$ is tangent to $E^u_F$ which is uniformly continuous  by Propsoition \ref{2.1 prop: unstable leaf continuity}, there exists $\e>0$ such that if  $d(x,y)<\e$ for any $x\in\RR^d$ and $y\in\tildeF^u(x)$, one has that $d_{\tildeF^u}(x,y)<2d(x,y)$. It follows that there exists $C>1$ such that $d_{\tildeF^{s/u}}(x,y)<C \cdot d(x,y)$ for any $x\in\RR^d$ and $y\in \tildeF^{s/u}(x)$.
	\end{proof}

\subsection{Conjugacies of Anosov maps}\label{subsec 2.3}
In this subsection, we always assume that $f,g\in \mathcal{A}^1(\TT^d)$ and $f$ is homotopic to $g$. Without loss of generality, we can suppose $f_*=g_*:\pi_1(\TT^d)\to \pi_1(\TT^d)$ and $f_*$ induces $A:\TT^d \to \TT^d$. We will build connection among the conjugacies on the universal cover given by Proposition  \ref{2.1 prop lifting conjugate}, on the inverse limit spaces given by Proposition \ref{2.1 prop: conjugacy on the inverse limit spaces} and on  the stable leaf spaces given by \eqref{eq.2.1.leaf space conjugacy} and also the leaf conjugacy (see \eqref{eq.1.leaf conjugacy}). As a corollary, we can "translate" the stable periodic data condition given by the leaf conjugacy into one by conjugacies on leaf space and inverse limit space. Here we mention that the last two way to match the periodic data on stable bundle can be extended to the case of dim$E^s_f>1$. 

First of all, we introduce the conjugacy on  (stable) leaf space. Here the \textit{(stable) leaf space} for $f$ and $g$ are the quotient spaces $\TT^d\big/ \mathcalf^s$ and $\TT^d\big/ \mathcalg^s$ respectively. For any $x\in\TT^d$, denote the equivalence class of $x$ in $\TT^d\big/{\mathcalf^s}$ and $\TT^d\big/{\mathcalg^s}$ by $[x]_f$ and $[x]_g$ respectively. Then $f$ and $g$ induce maps on their spaces of stable leaves:
$$f_s:\TT^d\big/{\mathcalf^s}\to\TT^d\big/{\mathcalf^s}\quad {\rm and} \quad g_s: \TT^d\big/{\mathcalg^s}\to  \TT^d\big/{\mathcalg^s} \ ,$$
with $f_s([x]_f)=[f(x)]_f$ and  $g_s([x]_g)=[g(x)]_g$ for any $x\in\TT^d$. 
A homeomorphism $$s:\TT^d\big/{\mathcalf^s}\to \TT^d\big/{\mathcalg^s} \  ,$$ is called a \textit{conjugacy on (stable) leaf spaces} between $f_s$ and $g_s$, if $s\circ f_s ([x]_f)= g_s\circ s([x]_f)$ for any $x\in\TT^d$.
Using  the conjugacy on leaf spaces is a natural way to connect the periodic points  of $f$ with ones of $g$ and match their periodic data.	Indeed,	 let $s$ be a conjugacy between $f_s$ and $g_s$. Since $s$ maps periodic leaf of $f$ to one of $g$ and any periodic stable leaf  admits a unique periodic point, $s$ induces a map between the sets of periodic points, denoted by 
\begin{align}
	s^P:{\rm Per}(f)\to {\rm Per}(g).\label{eq.2.2.sp}
\end{align}

A natural way to obtain a conjugacy on leaf spaces is projecting from universal cover. For convenience, we collect all conjugacies given by Proposition \ref{2.1 prop lifting conjugate} between all liftings of $f$ and $g$ by the natural projection $\pi:\RR^d\to\TT^d$ and  denote this set by $\mathcal{H}(f,g)$. 
Let $H\in\mathcal{H}(f,g)$ be the conjugacy 
between liftings $F$ and $G$ with $H\circ F=G\circ H$. Let	$H_s:\RR^d\big/ \tildeF^s \to \RR^d\big/ \tildeG^s$ be the conjugacy between $F_s$ and $G_s$ induced by $H$  in the sense of \eqref{eq.2.1.leaf space conjugacy} such that $H_s\circ F_s= G_s\circ H_s$. Note that  $\tildeF^s(x+n)=\tildeF^s(x)+n$ for all $x\in\RR^d$ and all $n\in\ZZ^d$, moreover we have $\pi\big(\tildeF^s(x)\big)=\mathcalf^s\big(\pi(x)\big)$ .
Then by Proposition \ref{2.1 prop H is Zd restricted on stable}, we can project $H_s$ to $d$-torus by  $\pi$ and induce a conjugacy between $f_s$ and $g_s$. This gives us a map from $\mathcal{H}(f,g)$ to $\mathcal{S}^*(f,g):=\big\{ s: \TT^d\big/{\mathcalf^s}\to \TT^d\big/{\mathcalg^s}\;  \;{\rm homeomorphism}\  \big|\; s\circ f_s=g_s\circ s  \big\}$ the space of conjugacies on leaf spaces . Namely, the map
\begin{align}
	\mathcal{S}_{f,g}:\mathcal{H}(f,g) \longrightarrow  \mathcal{S}^*(f,g) \label{eq. 2.3 map on conjugacy spaces}
\end{align} 
such that $\mathcal{S}_{f,g}(H)([x]_f):=\pi \circ H_s\big( \pi^{-1} [x]_f\big)$ is well defined. We will give a restriction called "bounded from Id" to the conjugacies on leaf space. We will see that this condition is similar to being homotopic to Id$_{\TT^d}$ if one observes it on the universal cover space $\RR^d$.

  For  precise statement, we need mark  liftings by fixed points. Indeed, any lifting $F$ of $f$  by  $\pi$ has exactly one fixed point (see \cite{Franks1969} and \cite[Theorem 8.1.3]{AokiHiraide1994}). We denote by $(F,p_0)$ the lifting of $f$ by $\pi$ with $F(p_0)=p_0$.	It follows from the Global Product Structure that the quotient space $\RR^d\big/{\tildeF^s}$ is homeomorphic to $\tildeF^u(p_0)$. Since the lifting $(F,p_0)$ induces a homeomorphism $F_s:\RR^d\big/{\tildeF^s}\to \RR^d\big/{\tildeF^s}$, in order not to abuse symbols,  we can also see $F_s$ as
$$F_s:\tildeF^u(p_0)\to \tildeF^u(p_0).$$
Let $x\in\RR^d$. Denote the equivalence class of $x$ in $\RR^d\big/{\tildeF^s}$  by $[x]_F$ and denote $x^u=\tildeF^s(x)\cap\tildeF^u(p_0)$. 	Note that $F_s(x^u)=F(x^u)=(F(x))^u$, this means that $F_s$ is in fact the restriction of $F$ on $\tildeF^u(p_0)$. Define a metric $d_F(\cdot,\cdot)$ on  $\RR^d\big/{\tildeF^s}$ by $$d_F([x]_F,[y]_F):=d(x^u,y^u),$$
where $d(\cdot,\cdot)$ is the Euclidean metric on $\RR^d$.

Let $s\in \mathcal{S}^*(f,g)$. Then $s$ maps the fixed point  of $f_s$ to one of  $g_s$. Denote by $p$ the unique fixed point of $f$ in $[p]_f$  and by $q$ for $g$ and $[q]_g$. Now we fix $p\in{\rm Fix}(f)$ and $q\in{\rm Fix}(g)$ with $s([p]_f)=[q]_g$. Then we fix $p_0\in \pi^{-1}(p)$ and $q_0\in\pi^{-1}(q)$. 	It is clear that $s$ can be lifted to the universal cover spaces $\RR^d\big/{\tildeF^s}$ and $\RR^d\big/{\tildeG^s}$ given by $(F,p_0)$ and $(G,q_0)$,
$$\tilde{s}_{p_0,q_0}:\RR^d\big/{\tildeF^s}\to \RR^d\big/{\tildeG^s}.$$
For short, we denote $\tilde{s}_{p_0,q_0}$ by $\tilde{s}$, if there is no confusion. Again by the Global Product Structure, we can see $\tilde{s}$ as $\tilde{s}:\tildeF^u(p_0) \to \tildeG^u(q_0)$. 

We say that $\tilde{s}$ is \textit{bounded from} Id, if there exists $C>0$ such that for any $x\in\RR^d$, $$d_G\big([x]_G, \tilde{s}([x]_F)\big)<C.$$
One can see in the following proposition that this is equivalent to that there exists $C>0$ such that 
$$d\big(x, \tilde{s}(x)\big)<C, \ \ \forall x\in \tildeF^u(p_0).$$
We mention in advance that this restriction for $s$ is also equivalent to its corresponding leaf conjugacy $h$ homotopic to Id$_{\TT^d}$ (see Remark \ref{2.2 rmk: redefine p.d.s.}).

Define the set $\mathcal{S}(f,g)$  by
$$\mathcal{S}(f,g):=\Big\{ s: \TT^d\big/{\mathcalf^s}\to \TT^d\big/{\mathcalg^s}\;  \;{\rm homeomorphism} \ \Big|\; s\circ f_s=g_s\circ s \;\; {\rm and}\;\; \exists\;\tilde{s} {\rm\ is \ bounded \ from \ Id}    \Big\}.$$
The following proposition says that the map $\mathcal{S}_{f,g}$ is  actually a surjection from $\mathcal{H}(f,g)$ to $\mathcal{S}(f,g)$.

\begin{proposition}\label{2.2 prop: the conjugacies on leaf spaces}
	Let $f,g\in \mathcal{A}^1(\TT^d)$ be homotopic and $\mathcal{S}_{f,g}:\mathcal{H}(f,g) \to  \mathcal{S}^*(f,g)$ be the map given by \eqref{eq. 2.3 map on conjugacy spaces}. Then one has the following two items. 
	\begin{enumerate}
		\item For any $H\in\mathcal{H}(f,g)$, one has that  $\mathcal{S}_{f,g}(H)\in\mathcal{S}(f,g)$ and for any $x\in\RR^d$, 
		\begin{align}
			\mathcal{S}_{f,g}(H)\circ\pi\big( [x]_F \big)=\pi\big([H(x)]_G\big), \label{eq. 2.8.1}
		\end{align}
	  where $(F,p_0)$ and $(G,q_0)$ are  liftings of $f$ and $g$ by  $\pi$ and conjugate via $H$. 
	
		\item For any  $s\in\mathcal{S}(f,g)$,  there exists $H\in\mathcal{H}(f,g)$ such that  $s=\mathcal{S}_{f,g}(H)$.
	\end{enumerate}
See the following commutative diagram ${\rm (Figure 2)}$.
\end{proposition}

     \[ \xymatrix@=5ex{
     	\RR^d \ar[rrrrr]^{F} \ar[ddddd]_{H} \ar[dr]^{\pi}  &&&&& \RR^d\ar[dl]_{\pi} \ar[ddddd]^{H}\\
     	&\TT^d\ar[dr]\ar[rrr]^f  &&&\TT^d \ar[dl]  \\
     	&&\TT^d/\mathcalf^s\ar[r]^{f_s} \ar[d]_{\mathcal{S}_{f,g}(H)}&\TT^d/\mathcalf^s\ar[d]^{\mathcal{S}_{f,g}(H)}\\
        &&\TT^d/\mathcalg^s\ar[r]_{g_s} &\TT^d/\mathcalg^s\\
        &\TT^d\ar[ur] \ar[rrr]_g &&&\TT^d \ar[ul] \\
        \RR^d \ar[rrrrr]_{G} \ar[ur]_{\pi}  &&&&& \RR^d\ar[ul]^{\pi}
     } 
     \]
      $$\text{Figure 2. The map}\  \mathcal{S}_{f,g}:\mathcal{H}(f,g) \to  \mathcal{S}^*(f,g) \text{\ is  surjective.}$$

\begin{proof}
	It is clear that $\mathcal{S}_{f,g}(H)$ is a homeomorphism satisfying \eqref{eq. 2.8.1}.  For getting the first item, it suffices to prove that $H_s$ is bounded from Id. 	
	
	By Proposition \ref{2.1 prop lifting conjugate}, there exists  $C_1>0$ such that $\|H-{\rm Id}_{\RR^d}\|_{C^0}<C_1$. Moreover  we claim that	there exists $C_0>0$ such that 
	\begin{align}
		\tilde{\mathcal{F}}^{s/u}_{\sigma}(x)\subset B_{C_0}\big(\tildeL^{s/u}(x)\big) ,\quad \forall x\in\RR^d \;\; {\rm and }\;\; \sigma=f,g \label{eq. 2.10.1}
	\end{align}
	where $B_{C_0}(\mathcal{L})$ is the $C_0$ neighborhood of the set $\mathcal{L}$.

	We need just prove the claim  for $\tildeF^s$. Indeed, let $H_F$ be the conjugacy between $F$ and $A$ given by  Proposition \ref{2.1 prop lifting conjugate} and take $C_2$ such that $\|H_F-{\rm Id}_{\RR^d}\|_{C^0}<C_2$. Then for any $y\in\tildeF^s(x)$ one has that $y\in B_{C_2}\big( \tildeL^s\big( H_F(x) \big)\big)$. It follows that $y\in B_{2C_2}\big( \tildeL^s( x)\big)$ and we can take $C_0=2C_2$.

	By  \eqref{eq. 2.10.1}, one has $\tildeG^u(q_0)\subset B_{C_0}\big(\tildeL^u(q_0)\big)$ and $\tildeG^s\big( H(x)\big)\subset B_{C_1}\big(  \tildeF^s(x) \big) \subset  B_{C_0+C_1}\big(\tildeL^s(x)\big) $ for any $x\in\RR^d$. Then we have that 
	$$x^u:=\  \tildeG^u(q_0) \cap\tildeG^s(x)\ \in\   B_{C_0}\big(\tildeL^u(q_0)\big) \cap B_{C_0+C_1}\big(\tildeL^s(x)\big),$$
	and 
	$$H(x)^u:=\  \tildeG^u(q_0)\cap \tildeG^s\big( H(x)\big)\  \in\  B_{C_0}\big(\tildeL^u(q_0)\big) \cap B_{C_0+C_1}\big(\tildeL^s(x)\big).$$
	It follows that $d_G\big([x]_G, H^s([x]_F)\big)= d\big(x^u, H(x)^u\big)$ is uniformly bounded for all $x\in\RR^d$ and hence $H_s$ is bounded from Id.

	Now we prove the second item.	Let $s\in\mathcal{S}(f,g)$.  By the definition of $\mathcal{S}(f,g)$, there exist liftings $(F,p_0)$ and $(G,q_0)$  of $f$ and $g$  with $s^P\big(\pi(p_0)\big)=\pi(q_0)\in{\rm Fix}(g)$ such that the lifting  $\tilde{s}$ of $s$ corresponding $(F,p_0)$ and $(G,q_0)$ is bounded from Id. For short, let $p=\pi(p_0)$ and $q=\pi(q_0)$.
	
	Let $H$ be the conjugacy between $F$ and $G$ given by Proposition \ref{2.1 prop lifting conjugate}. By the first item, we already have that $\mathcal{S}_{f,g}(H)\in\mathcal{S}(f,g)$. Hence  it suffices to prove that the conjugacy  which is bounded from Id  between $F_s:\RR^d\big/\tildeF^s\to \RR^d\big/\tildeF^s$ and $G_s:\RR^d\big/\tildeG^s\to \RR^d\big/\tildeG^s$ is unique and hence $\mathcal{S}_{f,g}(H)=s$.
	
	Indeed, assume that there exist two different conjugacies $h_i:\RR^d\big/\tildeF^s\to \RR^d\big/\tildeG^s\  (i=1,2)$ and both of them are bounded from Id. Then there exists $x\in\RR^d$ such that $h_1([x]_F)\neq h_2([x]_F)$ and this implies that these two leaves intersect $\tildeG^u(q_0)$ at two distinct points $y_1$ and $y_2$. Note that  for any $n\in\ZZ$, one has 
	$$G^ny_i=h_i\big( [F^nx]_F  \big)\cap \tildeG^u(q_0), \quad i=1,2.$$
	Let $H_G$ be the conjugacy between $(G,q_0)$ and $(A,0)$. Then  $H_G(y_1)\in \tildeL^u\big( H_G(y_2) \big)$ and  $$d\big( H_G(G^ny_1),H_G(G^ny_2)  \big)= d\big( A^n\circ H_G(y_1),  A^n\circ H_G(y_2)   \big) \to +\infty, $$
	as $n\to + \infty$. It follows that $d(G^ny_1,G^ny_2 )\to +\infty$ since $d_{C^0}(H_G,{\rm Id}_{\RR^d})$ is bounded. 
	Note that the assumption that $h_1$ and $h_2$ are both bounded from Id implies that both $G^ny_1$ and $G^ny_2$  have bounded distance with $[F^nx]_G\cap \tildeG^u(q_0)$, for any $n\in\NN$. It is a contradiction.
\end{proof}

Recall a (stable) \textit{leaf conjugacy} $h\in {\rm Home}_0(\TT^d)$ between $f$ and $g$ maps stable leaves of $f$ to ones of $g$ such that  $h\circ f \big(\mathcalf^s(x)\big)= g \circ h \big(\mathcalf^s(x)\big)$ for every $x\in\TT^d$. Let $f,g\in\mathcal{A}^1_1(\TT^d)$ be homotopic, the next proposition actually says that there exists leaf conjugacy between $f$ and $g$.

\begin{proposition}\label{2.2 prop leaf conjugate}
	Let $f\in\mathcal{A}^1_1(\TT^d)$ and $A$ be its linearization. 
	Then $f$ is leaf conjugate to $A$.
\end{proposition}

\begin{proof}
	Let $H:\mathbb{R}^d\to \mathbb{R}^d$ 
	be the unique conjugacy given by Proposition \ref{2.1 prop lifting conjugate} such that $H\circ F=A\circ H$.
 Then  by averaging $H$ over all $\mathbb{Z}^d$-actions, we can get a leaf conjugacy on $\mathbb{R}^d$ which can descend to $\mathbb{T}^d$. This method can be found in \cite[Section 6]{Hammerlindl2013} where one may get more details. We sketch the proof here.
	\begin{enumerate}
		\item Let $H_n(x) = H(x-n)+n$ for all $n\in \mathbb{Z}^d$ and $x\in \mathbb{R}^d$. Then by Proposition \ref{2.1 prop H is Zd restricted on stable}, for any  $k>0$, 
		\begin{align*}
			\overline{H}_k:=\sum_{n\in C_k}\frac{1}{(2k+1)^d}H_n,
		\end{align*}
		where $C_k=\{(k_1,...,k_d)\in \mathbb{Z}^d | {\rm max} \{ |k_1|,...,|k_d|\}\le k \}$ is well defined on each stable leaf.
		\item $\{\overline{H}_k\}_{k\in \mathbb{N}}$ is uniformly equicontinuous and has uniformly distance with Id$_{\RR^d}$, note that these 
		two properties are genetic from $\{ H_n \}_{n\in \mathbb{Z}^d}$.  
		By Arzel$\grave{\rm a}$-Ascoli theorem, we obtain an accumulation point of $\{\overline{H}_k\}_{k\in \mathbb{N}}$, denoted by $\tilde{H}$.
		\item $\tilde{H}$ is a leaf conjugacy between $F$ and $A$ which has uniformly distance with Id$_{\RR^d}$ and satisfies $\tilde{H}(x+n)=\tilde{H}(x)+n$ for all $n\in \mathbb{Z}^d$. Moreover, $\tilde{H}$ is injetive. We note that the injection relies on ${\rm dim}E^s=1$. Descend $\tilde{H}$ to  $h:\TT^d\to\TT^d$, then $h\in {\rm Home}_0(\TT^d)$ is a leaf conjugacy. 
	\end{enumerate}
\end{proof}

\begin{remark}
 Let $f,g\in\mathcal{A}^1_1(\TT^d)$ be homotopic. The last proposition actually shows that any conjugacy $H\in \mathcal{H}(f,g)$  can induce a leaf conjugacy $h\in {\rm Home}_0(\TT^d)$. It is not hard to see that for any  leaf conjugacy $h\in {\rm Home}_0(\TT^d)$ between $f$ and $g$, there exist liftings $F$ and $G$ such that the conjugacy $H\in \mathcal{H}(f,g)$ between  $F$ and $G$ induces  $h$. Indeed, since $h\in {\rm Home}_0(\TT^d)$ naturally induces a conjugacy on leaf space whose lifting is bounded from Id, by Proposition \ref{2.2 prop: the conjugacies on leaf spaces}, one can find $F, G$ and $H$.
\end{remark}

\begin{remark}\label{2.2 rmk: redefine p.d.s.}
	Through the  conjugacy on the stable leaf spaces, we redefine that $f,g\in\mathcal{A}^1_1(\TT^d)$ admit the  \textit{same stable periodic data} by matching periodic points via some $s\in\mathcal{S}(f,g)$,
   namely,
	\begin{align}
		\lambda^s_f(p)=\lambda^s_g\big( s^P(p) \big), \quad \forall p\in{\rm Per}(f),\label{eq. 2.2. rmk 1}
	\end{align} where $s^P$ is given by \eqref{eq.2.2.sp}. And this is equivalent to the definition given by leaf conjugacy (see \eqref{eq.1.pds}).
	Indeed, on the one side,  let $h\in {\rm Home}_0(\TT^d)$ be a leaf conjugacy  between $f$ and $g$ with \eqref{eq.1.pds} and $H\in \mathcal{H}(f,g)$ be a conjugacy on $\RR^d$ coresponding to  $h$. Then $H$  induces  $s=\mathcal{S}_{f,g}(H)\in \mathcal{S}(f,g)$   by Proposition \ref{2.2 prop: the conjugacies on leaf spaces} and it follows that $s$ satisfies \eqref{eq. 2.2. rmk 1}.
	On the other side, let $s\in\mathcal{S}(f,g)$ satisfy \eqref{eq. 2.2. rmk 1}. Then $s$ induces a conjugacy $H\in\mathcal{H}(f,g)$ by the second item of Proposition \ref{2.2 prop: the conjugacies on leaf spaces} and $H$ induces  a leaf conjugacy $h\in {\rm Home}_0(\TT^d)$ given by  Proposition \ref{2.2 prop leaf conjugate}. It is clear that $h$ satisfies \eqref{eq.1.pds}.
\end{remark}

The following proposition  says that every conjugacy $H\in \mathcal{H}(f,g)$  on $\RR^d$ can induce a conjugacy on inverse limit spaces. This will be helpful to "translate" the condition of matching stable periodic data on the leaf space into that on the inverse limit space (see Corollary \ref{2.2 cor: p.d.s on leaf space to inverse limit space }).

\begin{proposition}\label{2.2 prop: the relationship of 3 conjugacies inverse limit space}
	Let $f,g\in \mathcal{A}^1(\TT^d)$ be homotopic and $F,G$ be any two liftings of $f,g$ by  $\pi$. Assume that $H:\RR^d\to \RR^d$ is  the conjugacy given by Proposition \ref{2.1 prop lifting conjugate} between $F$ and $G$  with $H\circ F=G\circ H$. Then there exists  a  conjugacy $\bar{h}:\TT^d_f\to \TT^d_g$ between $\sigma_f:\TT^d_f\to\TT^d_f$ and $\sigma_g:\TT^d_g\to\TT^d_g$ such that $\bar{h}\circ\sigma_f=\sigma_g\circ \bar{h}$ and 
	\begin{align}
		\bar{h}\circ\pi\big( {\rm Orb}_F(x)  \big)  =  \pi\big( {\rm Orb}_G(H(x))  \big),\ \ \forall x\in\RR^d.\label{eq. 2.8.2}
	\end{align}	 
See the following commutative diagram {\rm (Figure 3)}, where $\pi_0:\TT^d_{f/g}\to \TT^d$ is the projection $\pi_0\big((x_i)\big)=x_0$.
\end{proposition}

	\[ \xymatrix@=5ex{
	\TT^d_f \ar[rrrrr]^{\sigma_f} \ar@{-->}[ddddd]_{\bar{h}} \ar[drr]^{\pi_0}  &&&&& \TT^d_f\ar[dll]_{\pi_0} \ar@{-->}[ddddd]^{\bar{h}}\\
	&&\TT^d\ar[r]^f  &\TT^d \\
	&\RR^d \ar[ur]^{\pi}\ar[d]_{H}\ar[rrr]^F&&&\RR^d\ar[ul]_{\pi}\ar[d]^{H}\\
	&\RR^d \ar[dr]_{\pi}\ar[rrr]_G&&&\RR^d\ar[dl]^{\pi}\\
	&&\TT^d\ar[r]_g  &\TT^d   \\
	\TT^d_f \ar[rrrrr]_{\sigma_g} \ar[urr]_{\pi_0}  &&&&& \TT^d_f\ar[ull]^{\pi_0}
} 
\]
$$\text{Figure 3. The conjugacy}\  H \text{\ induces a conjugacy}\ \bar{h}\  \text{on the inverse limit space.}$$

\begin{proof}
	Let $\RR^d_F$ be the orbit space of $F:\RR^d\to \RR^d$ and $\sigma_F:\RR^d_F\to\RR^d_F$ be the (left) shift homeomorphism just like \eqref{eq. 2.1 inverse limit space} and define $\sigma_G:\RR^d_G\to\RR^d_G$ by the same way. Then $H$ induces a conjugacy $\overline{H}: \RR^d_F\to \RR^d_G$ between shift maps $\sigma_F$  and $\sigma_G$ such that $\overline{H}\circ \sigma_F= \sigma_G\circ \overline{H}$ and $\overline{H}\big({\rm Orb}_F(x)\big)= {\rm Orb}_G\big(H(x)\big)$ for all $x\in\RR^d$. It is clear that 
	\begin{align}
		\TT^d_f=\overline{\bigcup_{x\in\RR^d}   \pi \big( {\rm Orb}_F(x)  \big)}. \label{eq. 2.2. R orbit to T}
	\end{align}
	Indeed, for any $\tilde{x}=(x_i)\in \TT^d_f$ and $k\leq 0$, let $y_k$ be a lifting of $x_k$. Then the projection of $F$-orbit of $F^k(y_k)$ is approaching to $\tilde{x}$.  If one need, we refer to \cite[Proposition 2.5]{MicenaTahzibi2016} for more details of \eqref{eq. 2.2. R orbit to T}. 
	
	We are going to prove that $\overline{H}$ can descend to $\TT^d_f$ by the following lemma.

	\begin{lemma}\label{2.2 lem: orbit conjugacy well defined}
		For any orbit $\bar{x}\in\TT^d_f$ and any two sequences $\{y_k\}$ and $\{z_k\}$ with  $\pi\big( {\rm Orb}_F(y_k) \big) \to \bar{x}$ and  $\pi\big( {\rm Orb}_F(z_k) \big) \to \bar{x}$ as $k\to +\infty$,  one has
		$$\bar{d} \Big( \pi\big( {\rm Orb}_G(H(y_k)) \big) \;,  \; \pi \big( {\rm Orb}_G(H(z_k))\big)    \Big)\to 0. $$ 
	\end{lemma}
	\begin{proof}[Proof of Lemma \ref{2.2 lem: orbit conjugacy well defined}]
		It is clear that for any $\e>0$, there exists $\alpha>0$ such that for any $x,y\in\RR^d$ with $d(x,y)<\alpha$, one has 
		$$\bar{d} \Big( \pi\big( {\rm Orb}_G(x) \big) \;, \;  \pi \big( {\rm Orb}_G(y)\big)    \Big) <\frac{\e}{2}. $$ 
		And for any $\e>0$, there exists $I_0\in\NN$ such that for all $x\in\RR^d$, all $i\geq I_0$ and all $n\in A^i\ZZ^d$, one has
		$$\bar{d} \Big( \pi\big( {\rm Orb}_G(x) \big)\; , \;  \pi \big( {\rm Orb}_G(x+n)\big)   \Big) <\frac{\e}{2}. $$ 
		Indeed, just note that $\pi\big(G^j(x+n)\big)=\pi\big(G^jx\big)$, for any $n\in A^i\ZZ^d$ and $-i\leq j <+\infty$.
		
		By Proposition \ref{2.1 prop nmH}, for given $\alpha>0$, there exists $I_1\in\NN$ such that for any $x\in\RR^d$, $i\geq I_1$ and  $n\in A^i\ZZ^d$, one has
		$$d\big( H(x+n)\;,\; H(x)+n \big)<\frac{\alpha}{2}.$$
		And by the uniform continuity of $H$, there exists $\delta>0$ such that $d(x,y)<\delta$, one has 
		$d\big(H(x),H(y)\big)<\frac{\alpha}{2}$.
		
		Now fix $\e>0$, then we can fix $\alpha>0$, $\delta>0$ and $i_0\geq{\rm max}\{I_0,I_1\}$.   The conditions that both $\pi\big( {\rm Orb}_F(y_k) \big) $ and  $\pi\big( {\rm Orb}_F(z_k) \big) $ converge to $ \bar{x}$ means that there exists $N_0\in\NN$ such that for every $k\geq N_0$, there exists $n_k\in\ZZ^d$ satisfying 
		$$d\left(F^j(F^{-i_0}z_k+n_k )\;,\; F^j\circ F^{-i_0}(y_k)\right)<\delta,$$
		for any $0\leq j\leq 2i_0$.   Especially  one has $d\left(z_k+ A^{i_0}n_k \; ,\; y_k\right)<\delta$ and hence
		$$d\big(H(z_k+ A^{i_0}n_k)\; , \; H(y_k)\big)<\frac{\alpha}{2}.$$
		Since $i_0\geq I_1$, one has $d\big(H(z_k+ A^{i_0}n_k) ,  H(z_k)+ A^{i_0}n_k\big) < \frac{\alpha}{2}$. It follows that
		$$d\big( H(z_k)+ A^{i_0}n_k \; , \; H(y_k)  \big)<\alpha.$$
		Then we get that 
		$$\bar{d} \Big( \pi\big({\rm Orb }_G(H(z_k)+A^{i_0}n_k)\big)\;,\;  \pi\big( {\rm Orb}_G(H(y_k))  \big)        \Big)<\frac{\e}{2}.$$
		On the other hand, since $i_0\geq I_0$, 
		$$\bar{d} \Big( \pi\big({\rm Orb }_G(H(z_k)+A^{i_0}n_k)\big)\;,\;  \pi\big( {\rm Orb}_G(H(z_k))  \big)        \Big)<\frac{\e}{2}. $$
		As a result, for any $\e>0$,  there exists $N_0\in\NN$ such that if $k\geq N_0$, one has 
		$$\bar{d} \Big( \pi\big({\rm Orb }_G(H(y_k))\big)\;,\;  \pi\big( {\rm Orb}_G(H(z_k))  \big) \Big)<\e. $$
	\end{proof}
	
	\begin{remark}
		Lemma \ref{2.2 lem: orbit conjugacy well defined} also implies that when $\pi\left( {\rm Orb}_F(y_k) \right) \to \bar{x}$, the sequence $\Big\{  \pi\left( {\rm Orb}_G(Hy_k)\right)  \Big\}$ is convergent in $\TT^d_g$. Indeed using the same notations in the proof of Lemma \ref{2.2 lem: orbit conjugacy well defined}, for every $\e>0$, every $i_0\geq {\rm max}\{I_0,I_1\}$ and every Cauchy sequence $\Big\{\pi\left( {\rm Orb}_F(y_k) \right)\Big\}\subset \TT^d_f$, there exists $N>0$ such that  for any $k,m>N$, there exists $n\in\ZZ^d$ with $$d\big(F^j\circ F^{-i_0}(y_k)\;,\; F^j(F^{-i_0}y_m+n) \big)<\delta,$$ for all $0\leq j\leq 2i_0$. 
		Then one has
		$$\bar{d}\big( \pi\circ {\rm Orb}_G(H(y_m))\;,\; \pi\circ {\rm Orb}_G(H(y_k))   \big)< \e.$$
		So $\big\{   \pi\circ {\rm Orb}_G(H(y_k))  \big\}$ is a Cauchy sequence  and hence convergent in the compact metric space $(\TT^d_g,\bar{d})$.
	\end{remark}
	\vspace{2mm}

	We continue to prove Proposition \ref{2.2 prop: the relationship of 3 conjugacies inverse limit space}. For any $\bar{x}\in\TT^d_f$, there exists $\{y_k\}\subset\RR^d$ such that $\pi\left( {\rm Orb}_F(y_k) \right)$ converges to $ \bar{x}$. By Lemma \ref{2.2 lem: orbit conjugacy well defined}, we can define $\bar{h}: \TT^d_f\to \TT^d_g$ as 
	\begin{align}
		\bar{h}(\bar{x})= \lim_{k\to\infty}\pi\left( {\rm Orb }_G(H(y_k))\right).\label{eq. 2.8.3}
	\end{align}
	By the  proof of Lemma \ref{2.2 lem: orbit conjugacy well defined},  we also get that $\bar{h}$ is a continuous. 
	
	Similarly, for any $\bar{x}\in \TT^d_g$ and $\{z_k\}\subset \RR^d$ with $\pi\big( {\rm Orb}_G(z_k) \big)$ converging to $ \bar{x}$.  Let $\bar{h}^{-1}$ be defined by
	$$\bar{h}^{-1}(\bar{x})= \lim_{n\to\infty}\pi\left( {\rm Orb }_F(H^{-1}(z_k))\right).$$
	It is well defined and also continuous. Note that  $\bar{h}\circ\bar{h}^{-1}={\rm Id}_{\TT^d_g}$ and $\bar{h}^{-1}\circ\bar{h}={\rm Id}_{\TT^d_f}$ by the definition. So we have that $\bar{h}$ is a homeomorphism. Moreover,
	\begin{align*}
		\sigma_g\circ \bar{h} (\bar{x})
		&=  \sigma_g \Big(\lim_{k\to\infty} \pi \circ {\rm Orb}_G\big(H(y_k)\big)\Big),\\
		&= \lim_{k\to\infty} \left( g \circ \pi \circ G^i\big(H(y_k)\big)  \right)_{i\in\ZZ},\\
		&= \lim_{k\to\infty} \left( \pi \circ G \circ G^i\big(H(y_k)\big)  \right)_{i\in\ZZ},\\
		&= \lim_{k\to\infty} \left( \pi \circ G^i\big(H(F(y_k))\big)  \right)_{i\in\ZZ},\\
		&= \lim_{k\to\infty} \pi  \circ {\rm Orb}_G\big(H(F(y_k))\big).
	\end{align*}
	Since $\pi\left( {\rm Orb}_F(y_k) \right)\to  \bar{x}$,  we have $\pi\left( {\rm Orb}_F(F(y_k)) \right)\to \sigma_f (\bar{x})$. Then by \eqref{eq. 2.8.3}, we get $$\sigma_g\circ \bar{h} (\bar{x})=\lim_{k\to\infty} \pi  \circ {\rm Orb}_G\big(H(F(y_k))\big)=\bar{h}\circ \sigma_f(\bar{x}).$$
\end{proof}

Let $\bar{h}:\TT^d_f\to\TT^d_g$ be a conjugacy between $\sigma_f:\TT^d_f\to\TT^d_f$ and $\sigma_g:\TT^d_g\to\TT^d_g$. For any $p\in{\rm Per}(f)$, we denote the periodic $f$-orbits of $p$ by $\bar{p}\in\TT^d_f$. It is clear that $\bar{h}(\bar{p})$ is a perodic point of $\sigma_g:\TT^d_g\to\TT^d_g$ with the same period of $\bar{p}$. We further assume $f\in\mathcal{A}^1_1(\TT^d)$ and define the stable Lyapunov exponent for $\bar{p}\in{\rm Per}(\sigma_f)$ by 
$$\lambda^s(f,\bar{p}):=\lambda^s_f\left(\pi_0(\bar{p})\right) ,$$
where $\lambda^s_f(x)$ is the  stable Lyapunov exponent of $x$ for $f$.

\begin{corollary}\label{2.2 cor: p.d.s on leaf space to inverse limit space }
	Let $f,g\in\mathcal{A}^1_1(\TT^d)$ be homotopic. If  $f$ and $g$ admit the same periodic data on the stable bundles, then there exist  liftings $F$ and $G$ of $f$ and $g$ by  $\pi$ such that the conjugacy $\bar{h}:\TT^d_f \to \TT^d_g$ between $\sigma_f$ and $\sigma_g$ given by Proposition \ref{2.2 prop: the relationship of 3 conjugacies inverse limit space} satisfies
	$$\lambda^s(f,\bar{p})=\lambda^s\big(g,\bar{h}(\bar{p})\big), \quad \forall \bar{p}\in {\rm Per}(\sigma_f).$$
\end{corollary}

\begin{proof}
	By the assumption and Remark \ref{2.2 rmk: redefine p.d.s.}, there exists  $s\in\mathcal{S}(f,g)$ such that $\lambda^s_f(p)=\lambda^s_g\big( s^P(p) \big)$ for all $p\in{\rm Per}(f)$. Let $\bar{h}$ be given by Proposition \ref{2.2 prop: the relationship of 3 conjugacies inverse limit space} in which $H$ is given by the second item of  Proposition \ref{2.2 prop: the conjugacies on leaf spaces}.
	Take $\bar{p}\in{\rm Per}(\sigma_f)$, we will show that $\bar{h}(\bar{p})$ and $s^P(p)$ express the same periodic orbit and in fact $\big(\bar{h}(\bar{p})\big)_0=s^P(p)$. It follows that $\lambda^s(f,\bar{p})= \lambda^s_f(p)=\lambda^s_g\big(s^P(p)\big)=\lambda^s\big(g,\bar{h}(\bar{p})\big)$.
	
	Since $\bar{h}(\bar{p})\in{\rm Per}(\sigma_g)$ and every periodic stable leaf admits exactly one periodic point, it suffices to prove that $$\mathcalg^s\big(\left(\bar{h}(\bar{p})\right)_0\big)=\mathcalg^s\big( s^P(p)\big).$$
	Indeed, let $p_n\in\RR^d$ with $\pi(p_n)=p$ such that $\pi\circ{\rm Orb}_F(p_n)\to \bar{p}$ as $n\to +\infty$. Then by \eqref{eq. 2.8.2} or \eqref{eq. 2.8.3}, one has $$\bar{h}(\bar{p})= \lim_{n\to +\infty} \pi\circ {\rm Orb}_G\big(H(p_n)\big).$$
	By the second item of Proposition \ref{2.2 prop: the conjugacies on leaf spaces}  and  \eqref{eq. 2.8.1}, one has 
	$$\big[\pi \big(H(p_n)\big)\big]_g=\pi\big[ H(p_n) \big]_G=s\circ\pi \big([p_n]_F\big)= s\big([p]_f\big)=[s^P(p)]_g.$$
	Therefore, 
	$$\big[\left(\bar{h}(\bar{p})\right)_0\big]_g=\lim_{n\to+\infty}\big[\pi(H(p_n))\big]_g=\big[s^P(p)\big]_g.$$
\end{proof}

	\subsection{Livschitz Theorem  on the inverse limit space }\label{subsec 2.4}
In this subsection, we found a Livschitz Theorem for  shift  homeomorphisms on the inverse limit spaces of Anosov maps. This will be helpful to construct an affine sturcture in Section \ref{sec 3}. Although we already have the Livschitz Theorem for Anosov maps (see Theorem \ref{2.3 thm Livschit for Anosov map}),  we can only use the Livschitz Theorem for the inverse limit space on the assumption of Corollary \ref{2.2 cor: p.d.s on leaf space to inverse limit space }. Recall that $M$ is a $C^{\infty}$-smooth closed Riemannian manifold. 

\begin{theorem}[Livschitz Theorem for Anosov map \cite{katokbook}]\label{2.3 thm Livschit for Anosov map}
	Let $f\in\mathcal{A}^{1+\alpha}(M)$ be transitive  and $\psi_1, \psi_2 :M\to \RR$  be two H$\ddot{o}$lder functions  with 
	$$\sum_{i=1}^{\pi({p})}\psi_1\big(f^i({p})\big)= \sum_{i=1}^{\pi({p})}\psi_2\big(f^i({p})\big), \quad \forall p\in {\rm Per}(f),$$
	where $\pi({p})$ is the period of $p$. Then  there exists a H$\ddot{o}$lder function $v:M\to \RR$ such that $$\psi_1(x)=\psi_2(x)+v\circ f(x)-v(x),\quad \forall x\in M.$$
\end{theorem}
As stating in \cite[Section 4]{AGGS2022},  the proof  is quite similar to the Livschitz Theorem for transitive Anosov diffeomorphisms whose  proof can be found in \cite[Corollary 6.4.17 and Theorem 19.2.1]{katokbook}. Note that we still have the Local Product Structure (see Proposition \ref{2.1 prop: local product sturcture}) which  is useful to obtain the exponential Anosov closing lemma (see \cite[Corollary 6.4.17 ]{katokbook}) and hence the Livschitz Theorem for Anosov maps.

\begin{theorem}[Livschitz Theorem on inverse limit space]\label{2.3 thm: Livshitz for inverse limit space}
	Let $f\in\mathcal{A}^{1+\alpha}(M)$ be transitive and $\sigma_f:M_f\to M_f$ be the shift homeomorphism on the inverse limit space.  Assume that $\va_1, \va_2 :M_f\to \RR$  are two H$\ddot{o}$lder functions  with 
	$$\sum_{i=1}^{\pi(\bar{p})}\va_1\big(\sigma_f^i(\bar{p})\big)= \sum_{i=1}^{\pi(\bar{p})}\va_2\big(\sigma_f^i(\bar{p})\big), \quad \forall \bar{p}\in {\rm Per}(\sigma_f),$$
	where $\pi(\bar{p})$ is the period of $\bar{p}$. Then $\va_1$ is H$\ddot{o}$lder-cohomology to $\va_2$, i.e.,  there exists a H$\ddot{o}$lder function $u:M_f\to \RR$ such that $$\va_1(\bar{x})=\va_2(\bar{x})+u\circ \sigma_f(\bar{x})-u(\bar{x}),\quad \forall \bar{x}\in M_f.$$
\end{theorem}

Since we already have Theorem \ref{2.3 thm Livschit for Anosov map}, we need just reduce  Theorem \ref{2.3 thm: Livshitz for inverse limit space} to the above case. And this can be done by the following lemma. Briefly, it  says that a function on the inverse limit space can be  cohomology  to a function on the base manifold by "cutting" the backward orbit.  We mention that this method originated from the case of the shift homeomorphism of finite type  in \cite[Section 1]{Bowen1975}. But our case is not discrete and we need more delicate calculation with the help of the Local Product Structure (see Proposition \ref{2.1 prop: local product sturcture}).  Note that this lemma only need $f$ be $C^{1}$-smooth.

\begin{lemma}\label{2.3 lem: Cohomology to forward orbits function}
	Let $f\in\mathcal{A}^1(M)$ and	$\va:M_f\to \RR$ be a  H$\ddot{o}$lder function. Then there exists a H$\ddot{o}$lder function $\psi:M_f\to \RR$ H$\ddot{o}$lder-cohomology   to $\va$ such that for all $\bar{x},\bar{y}\in M_f$, if  $x_i=y_i$ for all $i\geq 0$, then $\psi(\bar{x})=\psi(\bar{y})$. 
	Especially, $\psi$ can descend to $M$.
\end{lemma}

\begin{proof}
	For any $x\in M$, we fix an orbit $\bar{a}(x)=(a_i(x))\in M_f$ with $a_0(x)=x$.	Let $r: M_f\to M_f$ be defined as $r(\bar{y})=\bar{y}^*$ such that $(\bar{y}^*)_i=y_i$ for $i\geq 0$ and   $(\bar{y}^*)_i=a_{i}(y_0)$ for $i\leq 0$. Let 
	$$u(\bar{y})=\sum_{j=0}^{+\infty}\left( \va \circ \sigma_f^j(\bar{y})-   \va \circ \sigma_f^j\circ r(\bar{y})  \right).$$
	
	Firstly, we claim that $u: M\to \RR$ is well defined and bounded. Indeed, let $\va$ be H$\ddot{\rm o}$lder with constant $C_{\va}>0$ and exponent $\alpha\in(0,1)$. Then  for any $j>0$, since $\big(\sigma_f^j(\bar{y})\big)_i= \big(\sigma_f^j\circ r(\bar{y}) \big)_i$ for $i\geq -j$, we have
	\begin{align*}
		\left|  \va \circ \sigma_f^j(\bar{y}) -   \va \circ \sigma_f^j\circ r(\bar{y})    \right| 
		\leq C_{\va}\cdot \bar{d}^{\alpha}\left(  \sigma_f^j(\bar{y}) \sigma_f^j\circ r(\bar{y}) \right)
		\leq C_{\va} \cdot \left(  \sum_{i=j+1}^{+\infty} \frac{1}{2^i}   \right) ^{\alpha}
		\leq C_{\va} \cdot 2^{-j\alpha}.
	\end{align*}
	It follows that $u(\bar{y})\leq C_{\va} \cdot \sum_{j=0}^{+\infty}2^{-j\alpha}<\infty$.
	
	Then we prove that $r:M_f \to M_f$ is H$\ddot{\rm o}$lder continuous.  Let $\delta_0$ satisfy the size of Local Product Structure (see  Proposition \ref{2.1 prop: local product sturcture}).  Then there exists $C_1=C_1(\delta)>0$ such that for any $\e<\delta_0$ and any two points $y_0,z_0$ with $d(y_0,z_0)<\e$, one has that $\mathcalf^s(y_0,C_1\e)\cap \mathcalf^u(z_0,\bar{z}, C_1\e)$ is a single point set, denote this point by $w_0=w_0(y_0,\bar{z})$ and the orbit $\bar{w}=(w_i)$ given by Proposition \ref{2.1 prop: local product sturcture}.   Let $\lambda :={\rm sup}_{\bar{x}\in\TT^d_f}\|Df^{-1}|_{E_f^s(\bar{x})}\| $. Then for $0\leq i\leq i_{\e}:= [{\rm log}_{\lambda}\frac{\delta_0}{2C_1\e}]$, one has
	\begin{align*}
		d\left( r(\bar{y})_{-i}, r(\bar{z})_{-i}  \right)
		\leq d_{\mathcalf^s}\big(r(\bar{y})_{-i}, w_{-i} \big) +d_{\mathcalf^u} \big(r(\bar{z})_{-i}, w_{-i} \big)
		\leq 2 C_1\e\lambda^i<\delta_0.
	\end{align*}
	So when $\bar{d}(\bar{y},\bar{z})=\e\ll\delta_0$, we have
	\begin{align*}
		\bar{d}\big(r(\bar{y}), r(\bar{z})\big)
		&=\sum_{i=i_{\e}+1}^{+\infty}\frac{d(r(\bar{y})_{-i},    r(\bar{z})_{-i})}{2^i}
		+\sum_{i=1}^{i_{\e}}\frac{d(   r(\bar{y})_{-i},    r(\bar{z})_{-i})}{2^i}
		+\sum_{i=0}^{+\infty}\frac{d(y_i,    z_i)}{2^i},\\
		&\leq C\cdot 2^{-i_{\e}}  +  2C_1 \e \cdot \sum_{i=1}^{i_{\e}}\frac{\lambda^i}{2^i} + \bar{d}(\bar{y},\bar{z}),\\
		&\leq C \cdot \left(  \frac{2C_1\e}{\delta_0}  \right)^{\frac{ln2}{ln \lambda}  } + 2C\e \left(  \frac{2C_1\e}{\delta_0}  \right)^{-\frac{ln\frac{\lambda}{2}}{ln\lambda}} +\e,
	\end{align*}
	for some $C>1$. We can assume that $\lambda>2$, otherwise the sum $\sum_{i=1}^{i_{\e}}\frac{\lambda^i}{2^i}$ is controlled by a constant.  Then there exist  $C_{\delta_0}$ and  $0<\beta<1$ such that 
	\begin{align}
		\bar{d}\big(r(\bar{y}), r(\bar{z})\big)	\leq C_{\delta_0}\cdot \e^{\beta}= C_{\delta_0}\cdot \bar{d}(\bar{y},\bar{z})^{\beta}. \label{eq. 2.4.1 livsic} 
	\end{align}
	
	Now we show that $u$ is  H$\ddot{\rm o}$lder continuous.  Note that for all $\bar{y}=(y_i), \bar{z}=(z_i) \in  M_f$, one has
	\begin{align}
		\bar{d}\left(\sigma_f(\bar{y}),\sigma_f(\bar{z}) \right) \leq 2\bar{d}(\bar{y},\bar{z}).\label{eq. 2.4.2 livsic}
	\end{align}   
	Indeed, it follows from the next two equations,
	\begin{align*}
		\sum_{i=1}^{+\infty} \frac{ d\left( (\sigma_f\bar{y})_{-i},  (\sigma_f\bar{z})_{-i}       \right)}{2^i}
		= \frac{1}{2}\cdot\sum_{i=1}^{+\infty} \frac{ d\left( y_{1-i}, z_{1-i}       \right)}{2^{i-1}}
		=\frac{1}{2}\cdot\sum_{i=0}^{+\infty} \frac{ d\left( y_{-i}, z_{-i}       \right)}{2^{i}},
	\end{align*}   
	and
	\begin{align*}
		\sum_{i=0}^{+\infty} \frac{ d\left( (\sigma_f\bar{y})_i,  (\sigma_f\bar{z})_i       \right)}{2^i}
		= 2\cdot\sum_{i=0}^{+\infty} \frac{ d\left( y_{i+1}, z_{i+1}       \right)}{2^{i+1}}
		= 2\cdot\sum_{i=1}^{+\infty} \frac{ d\left( y_{i}, z_{i}       \right)}{2^{i}}.
	\end{align*}

	Let $\e:=\bar{d}(\bar{y},\bar{z})$ and $\e_0=\e^{1/2}$. Then for $0\leq j\leq j_0:=  [{\rm log}_2\frac{\e_0}{\e}]$ by \eqref{eq. 2.4.2 livsic}, 
	\begin{align}
		\left|   \va \circ \sigma_f^j(\bar{y}) -   \va \circ \sigma_f^j(\bar{z}) \right|  \leq C_{\va} \cdot (2^j \e)^{\alpha} \leq C_{\va} \cdot \e_0^{\alpha}. \label{eq. 2.4.3 livsic}
	\end{align}
	Therefore by the choice of $j_0$ and equations \eqref{eq. 2.4.1 livsic} and \eqref{eq. 2.4.3 livsic}, we have
	\begin{align*}
		|u(\bar{y})- u(\bar{z})| 
		&\leq \sum_{j=0}^{j_0}\left|  \va \circ \sigma_f^j(\bar{y}) -   \va \circ \sigma_f^j(\bar{z})   \right|
		+ \sum_{j=0}^{j_0}\left|  \va \circ \sigma_f^j\circ r (\bar{y}) -   \va \circ \sigma_f^j\circ r (\bar{z})   \right|\\
		&+ \sum_{j=j_0+1}^{\infty}\left|    \va \circ \sigma_f^j(\bar{y})-   \va \circ \sigma_f^j\circ r(\bar{y})     \right|
		+  \sum_{j=j_0+1}^{\infty}\left|    \va \circ \sigma_f^j(\bar{z})-   \va \circ \sigma_f^j\circ r(\bar{z})     \right|,\\
		&\leq  \sum_{j=0}^{j_0} C_{\va} \cdot (2^j \e)^{\alpha} +   \sum_{j=0}^{j_0} C_{\va} C_{\delta_0} \cdot (2^j \e)^{\alpha\beta}
		+ 2\sum_{j=j_0+1}^{\infty} C_{\va} \cdot 2^{-j\alpha},\\
		&\leq C\cdot 2^{j_0\alpha}\e^{\alpha}   +   C\cdot 2^{j_0\alpha\beta}\e^{\alpha\beta}  +    C\cdot 2^{-j_0\alpha},\\
		&\leq C\cdot \e^{\alpha/2}   +   C\cdot \e^{\alpha\beta/2}   +     C\cdot \e^{\alpha/2}  
		\leq C\cdot \bar{d}(\bar{y},\bar{z})^{\alpha_0},
	\end{align*}
	where $0<\alpha_0<1$ and both $\alpha_0$ and $C$ only depend on $\va$ and $f$. Hence $u$ is H$\ddot{\rm o}$lder continuous.
	
	Finally, let
	\begin{align*}
		\psi(\bar{y})&= \va(\bar{y})+u\circ \sigma_f(\bar{y})- u(\bar{y}),\\
		&= \va(r(\bar{y}))+ \sum_{j=0}^{+\infty}\left(   \va(\sigma_f^{j+1}\circ r(\bar{y}))-     \va(\sigma_f^j\circ r(\sigma_f\bar{y}))  \right).
	\end{align*}
	It follows that  $\psi(\bar{y})$ only depends on $\{y_i \}_{i\geq 0}$ and is H$\ddot{\rm o}$lder continuous.
\end{proof}

\begin{proof}[Proof of Theorem \ref{2.3 thm: Livshitz for inverse limit space} ]
	By Lemma \ref{2.3 lem: Cohomology to forward orbits function}, there exist two H$\ddot{\rm o}$lder functions $u_1$ and $u_2$ defined on $ M_f$ such that $\psi_1=\va_1+u_1\circ\sigma_f-u_1$ and $\psi_2=\va_2+u_2\circ\sigma_f-u_2$ are both independent of  negative orbits. This means that $\psi_1$ and $\psi_2$ are actually functions defined on  $ M$. Then by  the Livschitz Theorem for Anosov maps, there exists a  H$\ddot{\rm o}$lder function $v: M\to\RR$ such that $\psi_1=\psi_2+v\circ f-v$. Denote the lifting of $v$ on the inverse limit space $M_f$ by $v_0$. It is clear that $v_0: M_f\to \RR$ is  H$\ddot{\rm o}$lder continuous under the metric $\bar{d}(\cdot,\cdot)$. Now, let $u=v_0+u_2-u_1$, we have that $$\va_1=\va_2 +u\circ \sigma_f - u,$$
	where $u: M_f\to \RR$ is H$\ddot{\rm o}$lder continuous.
\end{proof}

\subsection{Topological classification for higher-dimensional Anosov maps}\label{subsec 2.2}
As mentioned in the Introduction, Theorem \ref{1 thm main thm} and Corollary \ref{1 cor h smooth along stable} also hold for $d$-torus $(d\geq 2)$  on the assumption that $f\in\mathcal{N}^r_1(\TT^d)$ $(r>1)$ is irreducible. In this subsection, we state these two results and give the frame to prove them and also Theorem \ref{1 thm main thm}. 

We say that an Anosov map $f\in\mathcal{A}^1(\TT^d)$ is \textit{irreducible}, if its linearization $A=f_*: \pi_1(\TT^d)\to \pi_1(\TT^d)$ has irreducible characteristic polynomial over $\mathbb{Q}$ as a matrix $A\in GL_d(\RR)\cap M_d(\ZZ)$.

It is clear that the Theorem \ref{1 thm main thm} is a direct corollary of the following one.

\begin{theorem}\label{2 thm main theorem}
	Let $f,g\in \mathcal{N}^{r}_1(\TT^d)\ (r>1)$ be homotopic and  irreducible. Then $f$ is conjugate to $g$ if and only if  $f$ and $g$ admit the same stable periodic data. 
\end{theorem}

Here we mention that in this $d$-torus case, we also have leaf conjugacy (see Proposition \ref{2.2 prop leaf conjugate}) to match the stable periodic data of $f$ and $g$ in the sense of \eqref{eq.1.pds}. The case that $g$ is special has been studied in \cite{AGGS2022} as the following theorem.

\begin{theorem}[\cite{AGGS2022}]\label{1 theorem AGGS}
	Let $f\in  \mathcal{N}^{r}_1(\TT^d)\ (r>1)$ be irreducible. 
	Then $f$ is conjugate to its linearization $A$, if and only if,  $f$  admits the same stable periodic data as $A$.    Moreover, both conditions implies that the conjugacy is smooth along each leaf of the stable foliation.
\end{theorem}

The sufficient part in \cite{AGGS2022} relies on the existence of affine metric on each leaf of the stable foliation $\mathcalf^s$  which requires the periodic stable Lyapunov exponents of $f$ actually to be constant. Instead of the above affine metric, we will build the relation between affine structures of $\tildeF^s$ and $\tildeG^s$,  via  Livschitz Theorem on the inverse limit space (see Theorem \ref{2.3 thm: Livshitz for inverse limit space}). For this, we need match the stable periodic data by the conjugacy on inverse limit spaces  by Corollary \ref{2.2 cor: p.d.s on leaf space to inverse limit space }.  Hence it suffices to prove the following theorem for obtaining the sufficient part of  Theorem \ref{2 thm main theorem}. We mention that the proof in this paper for the sufficient part of  Theorem \ref{2 thm main theorem}  is also effective to the case of \cite{AGGS2022}. 

\begin{theorem}\label{2 thm s-rigidity implies conjugacy}
	Let $f,g\in\mathcal{A}^{r}_1(\TT^d)\ (r>1)$ be homotopic and irreducible. Let $F,G$ be liftings of $f,g$ respectively by  $\pi:\RR^d\to \TT^d$. 	Let  $H:\RR^d \to \RR^d$ be the conjugacy given by Proposition \ref{2.1 prop lifting conjugate} such that $H\circ F= G\circ H$. Denote by $\bar{h}:\TT^d_f\to \TT^d_g$ the conjugacy induced by $H$ given by Proposition \ref{2.2 prop: the relationship of 3 conjugacies inverse limit space}. If
	$$\lambda^s(f,\bar{p})=\lambda^s\big(g,\bar{h}(\bar{p})\big), \quad \forall \bar{p}\in {\rm Per}(\sigma_f),$$  then $H$ is commutative with deck transformation. This implies that $f$ is topologically conjugate to $g$ on $\TT^d$.
\end{theorem}

We prove this theorem in Section \ref{sec 3}. 
\vspace{3mm}

For  the necessary part of Theorem \ref{2 thm main theorem}, especially for the non-special case, we first give a dichotomy for $f\in\mathcal{A}^1_1(\TT^d)$ which says that a non-special Anosov map possesses a kind of accessibility.  

Przytycki constructed in  \cite{Przytycki1976} a class of Anosov maps on $\TT^d$  such that each of them admits a certain point on which the set of unstable directions  contains a curve homeomorphic to interval in the (dim$M-$dim$E^s$)-Grassman space. By this, it is reasonable to consider the \textit{$u$-accessible class} for $f\in\mathcal{A}^1(M)$ as the following definition.

\begin{definition}\label{2 def u-accessible class}
	For any $f\in\mathcal{A}^1(M)$ and any $x\in M$, the \textit{$u$-accessible class} of $x$ for $f$ is defined as 
	\begin{align*}
		Acc(x,f):=\big\{ y\in\TT^d \; | \;  \exists x=y^0, y^1,...,y^k=y \;\;&{\rm and} \;\; \tilde{y}^i\in M_f \;\; {\rm with}\;\; (\tilde{y}^i)_0=y^i \\
		&{\rm such \;\; that} \;\; y^{i+1}\in\mathcalf^u (y^i,\tilde{y}^i), \forall 0\leq i\leq k-1 \big\}.
	\end{align*}
	We say $f$ is \textit{$u$-accessible}, if $Acc(x,f)=M$, for any $x\in M$. And this means that every pair of points on $M$ can be connected by path of unstable manifolds with respect to  different negative orbits.
\end{definition}

We have the following dichotomy which will be proved in Section \ref{sec accessible}.
\begin{proposition}\label{1 prop dichotomy of non-invertible Anosov map}
	Let $f\in \mathcal{A}^1_1(\TT^d)$. Then  either $f$ is special, or $f$ is $u$-accessible. 
\end{proposition}

\begin{remark}
	We give another view of $u$-accessibility in the sense of that one can see the preimage set \eqref{eq. 1.1. preimage set} of a point as its strongest stable leaf.
	Let  $f_1:\TT^3\to\TT^3$ be  an Anosov diffeomorphism with partially hyperbolic splitting $T\TT^3=E_{f_1}^{ss}\oplus E_{f_1}^{ws}\oplus E_{f_1}^u$. If $E_{f_1}^{ss}$ and $E_{f_1}^u$ are not jointly integrable, then $f_1$ is $su$-accessible (see \cite{HammerlindlUres2014, GRZ2017}), i.e., any two points on $\TT^3$ can be connected by finite many strong stable manifolds tangent to $E_{f_1}^{ss}$ and unstable manifolds.  Hence, one can see  the joint integrability of $E_{f_1}^{ss}\oplus E_{f_1}^u$ as the special property of $f$ and see  $su$-accessibility of $f_1$ as $u$-accessiblity of $f$.  There are also  famous works about the $su$-accessible classes of  partially hyperbolic diffeomorphisms e.g.  \cite{Hertz2005, HHU2008}. Meanwhile some reseachers have studied the accessibility of non-invertible partially hyperbolic local diffeomorphisms e.g.  \cite{He2017}.
\end{remark}

\vspace{1mm}

Let $f,g\in \mathcal{N}^{r}_1(\TT^d)\ (r>1)$ be homotopic and irreducible. Since $f$ is special if and only if $f$ is conjugate to its linearization $A$ by Proposition \ref{2.1 prop special and conjugate},  we have that $f$ and $g$ are both special or $u$-accessible simultaneously from Proposition \ref{1 prop dichotomy of non-invertible Anosov map}. 
If $f$ and $g$ are special, then by Theorem \ref{1 theorem AGGS} one has $$\lambda^s_f(p)=\lambda^s(A)=\lambda^s_g(h(p)),\quad  \forall p\in{\rm Per}(f),$$ where $h\in {\rm Home}_0(\TT^d)$  is a conjugacy between $f$ and $g$. The necessary part of the Theorem \ref{2 thm main theorem} can be deduced by the following theorem.

\begin{theorem}\label{2 thm conjugacy implies s-rigidity}
	Let $f,g\in\mathcal{N}^r_1(\TT^d)$ $(r>1)$ be conjugate via a homeomorphism $h:\TT^d\to \TT^d$. If $f$ is $u$-accessible.
	Then $\lambda_f^s(p)=\lambda_g^s\big(h(p)\big)$,  for every $p\in{\rm Per}(f)$. 
\end{theorem}

We mention that in the $u$-accessible case, there is no need to assume that $f$ is irreducible for the necessary part of Theorem \ref{2 thm main theorem}. However, in the special case, $f$ must be a priori irreducible ( see \cite{AGGS2022}). 

\vspace{2mm}
We also have a version of  Corollary \ref{1 cor h smooth along stable} for higher-dimensional  case.   Note that Theorem \ref{1 theorem AGGS} state only $C^{1+\alpha}$-regularity of conjugacy along stable leaves. Hence we give the corollary for both  special case and $u$-accessible case.

\begin{corollary}\label{2 cor conjugacy implies smooth}
	Let $f,g\in\mathcal{N}^r_1(\TT^d)$ $(r>1)$ be irreducible and conjugate via a homeomorphism $h:\TT^d\to \TT^d$.
	Then $h$ is $C^r$-smooth along the stable foliation.
\end{corollary}

We will give the proofs of Theorem \ref{2 thm conjugacy implies s-rigidity}  and Corollary \ref{2 cor conjugacy implies smooth} in Section \ref{sec 4}.
\vspace{3mm}

  To end this subsection, we explain Theorem \ref{2 thm main theorem} in the sense of mapping spaces. For a linear Anosov map $A\in\mathcal{N}^1_1(\TT^d)\ (d\geq 2)$, let $f\in\mathcal{A}^r(A)$ and $(f,H)\in\tilde{\mathcal{A}}^r(A)$ be a tuple defined by \ref{A-maps}, where $H\in\mathcal{H}(f,A)$ with $A\circ H=H\circ F$ for some lifting $F$ of $f$. 
   Define two tuples $(f_1,H_1), (f_2,H_2)$ are  equivalent denoted by $(f_1,H_1)\sim_H (f_2,H_2)$, if  $H_2^{-1}\circ H_1: \RR^d\to\RR^d$ can descend to $h\in{\rm Home}_0(\TT^d)$.   Denote by $\mathcal{T}^r_H(A)$ the space of equivalence classes of such tuples.
  Denote by $\mathcal{F}^{\rm H}(\TT^d_A)$ the set of {\rm H${\rm \ddot{o}}$lder} functions on $\TT^2_A$. Define $\tilde{\phi}_1 \sim_{\sigma_A}\tilde{ \phi}_2 $ in $\mathcal{F}^{\rm H}(\TT^d_A)$ as 
  $$\tilde{\phi}_1 \sim_{\sigma_A}\tilde{ \phi}_2 \iff \exists\;  {\rm H\ddot{o}lder} \;\tilde{u}:\TT^d_A\to \RR {\rm \;\; such\;\; that\;\;} \tilde{\phi}_1-\tilde{\phi}_2=\tilde{u}\circ \sigma_A-\tilde{u}.$$ 
  Theorem \ref{thm 1.1'}'  is a direct corollary of the following one. 
   \begin{theorem3}\label{thm 2.1'}
   	Let $A\in GL_d(\RR)\cap M_d(\ZZ)$ induce a non-invertible irreducible  Anosov map with one-dimensional stable bundle. Then for every $r>1$, there is  a natural injection
   	\begin{align*}
   		\mathcal{T}^r_{H}(A)\hookrightarrow \mathcal{F}^{\rm H}(\TT^d_A)\big/_{\sim_{\sigma_A}}\qquad  {\rm and}\qquad  (f,H)\mapsto {\rm log}\|Df|_{E_f^s}\|\circ \bar{h}^{-1},
   	\end{align*}
   where $\bar{h}$ is induced by $H$ via Proposition \ref{2.2 prop: the relationship of 3 conjugacies inverse limit space}.
   	 \end{theorem3}

   	\begin{proof}
   Let  $f\in\mathcal{A}^r(A)$ and $F,A$ be  liftings  of $f,A$ by $\pi$ respectively. Let  $(f,H)$ be the tuple such that $H$ is the unique conjugacy between $F$ and $A$ given by Proposition \ref{2.1 prop lifting conjugate} such that $A\circ H= H\circ F$.  Let 
  	$$i:(f,H)\mapsto {\rm log}\|Df|_{E_f^s}\|\circ \bar{h}^{-1},$$ 
  	where $\bar{h}$ is induced by $H$ via Proposition \ref{2.2 prop: the relationship of 3 conjugacies inverse limit space}.
  	
  	On the one hand, $i$ naturally induces a well defined map $\tilde{i}:\mathcal{T}^r_{H}(A)\to \mathcal{F}^{\rm H}(\TT^d_A)\big/_{\sim_{\sigma_A}}$. Indeed, let $(f_1,H_1)\sim (f_2,H_2)$ and by definition the map $H=H_2^{-1}\circ H_1$ induces a homeomorphism $h\in{\rm Home}_0(\TT^d)$ such that $h \circ f_1=f_2\circ h$. Then by the Theorem \ref{2 thm conjugacy implies s-rigidity} and the Livschitz Theorem for Anosov maps (see Theorem \ref{2.3 thm Livschit for Anosov map}), one has that there exists a H${\rm \ddot{o}}$lder function $u:\TT^d\to\RR$ such that  
  	$$
  	{\rm log}\|Df_2|_{E_{f_2}^s}\|\circ h={\rm log}\|Df_1|_{E_{f_1}^s}\|+u\circ f_1-u.
  	$$ 
  	Define $\bar{u}:\TT^d_{f_1}\to\RR$ as $\bar{u}(\bar{x}):=u\big( (\bar{x})_0\big)$. Then we have 
  	$${\rm log}\|Df_2|_{E_{f_2}^s}\|\circ \bar{h}={\rm log}\|Df_1|_{E_{f_1}^s}\|+\bar{u}\circ \sigma_{f_1}-\bar{u}.$$ 
  	Let $\bar{h}_i\ (i=1,2)$ be conjugacy between $\sigma_{f_i}$ and $\sigma_A$ induced by $H_i$ by Proposition \ref{2.2 prop: the relationship of 3 conjugacies inverse limit space}. Composing $\bar{h}_1^{-1}$ to the last equation, then we get
  	$${\rm log}\|Df_2|_{E_{f_2}^s}\|\circ \bar{h}_2^{-1}={\rm log}\|Df_1|_{E_{f_1}^s}\|\circ\bar{h}_1^{-1}+(\bar{u}\circ\bar{h}_1^{-1})\circ \sigma_A-(\bar{u}\circ\bar{h}_1^{-1}).$$ It is clear that $\bar{u}\circ\bar{h}_1^{-1 }$ is a H${\rm \ddot{o}}$lder function defined on $\TT^d_A$. 
  	
  	On the other hand, $\bar{i}$ is injective. In fact, let  $(f_1,H_1)$ and $(f_2,H_2)$ satisfy $${\rm log}\|Df_1|_{E_{f_1}^s}\|\circ \bar{h}_1^{-1}\sim_{\sigma_A} {\rm log}\|Df_2|_{E_{f_2}^s}\|\circ \bar{h}_2^{-1}.$$  Denote $\bar{h}=\bar{h}_2^{-1} \circ \bar{h}_1$, then one has that 
  	\begin{align}
  			{\rm log}\|Df_1|_{E_{f_1}^s}\| \sim_{\sigma_{f_1}} {\rm log}\|Df_2|_{E_{f_2}^s}\|\circ \bar{h}.\label{eq. 2.3. condition}
  	\end{align}
  	Since both $\bar{h}_1$ and $\bar{h}_2$ satisfy Proposition \ref{2.2 prop: the relationship of 3 conjugacies inverse limit space}, the conjugacy $\bar{h}$ can be given by $H_2^{-1}\circ H_1$ by Proposition \ref{2.2 prop: the relationship of 3 conjugacies inverse limit space}.  
  	Hence \eqref{eq. 2.3. condition} actually gives us the condition stated in Theorem \ref{2 thm s-rigidity implies conjugacy}.  It follows that $H_2^{-1}\circ H_1$ can descend to $\TT^d$.
 	\end{proof}

 \section{The existence of topological conjugacy on $\TT^d$}\label{sec 3}
	
	In this section, we prove Theorem \ref{2 thm s-rigidity implies conjugacy}. Assume that $f,g\in\mathcal{A}^{1+\alpha}_1(\TT^d)$ are homotopic and irreducible. Let $F,G$ be liftings of $f,g$ respectively by the natural projection $\pi:\RR^d\to \TT^d$. 	Let  $H:\RR^d \to \RR^d$ be the conjugacy such that $H\circ F= G\circ H$ given by Proposition \ref{2.1 prop lifting conjugate} which induces a conjugacy $\bar{h}:\TT^d_f\to \TT^d_g$ between $\sigma_f$ and $\sigma_g$ by Proposition \ref{2.2 prop: the relationship of 3 conjugacies inverse limit space}. Suppose 
	$$\lambda^s(f,\bar{p})=\lambda^s\big(g,\bar{h}(\bar{p})\big), \quad \forall \bar{p}\in {\rm Per}(\sigma_f).$$ 
 Theorem \ref{2 thm s-rigidity implies conjugacy} says that for obtaining the conjugacy between $f$ and $g$, we  prove that $H$ is commutative with deck transformations. 
    
  Firstly, we endow  affine structures on each leaf of stable foliations of $F$ and $G$ on $\RR^d$.  Denote the stable/unstable bundle of $F/G$ by $E^{s/u}_{F/G}$ respectively. 
  For any $x\in\RR^d$ and $y\in\tildeF^s(x)$, define the density function
   $$\rho_F(x,y):=\prod_{i=0}^{+\infty} \frac{\| DF|_{E^s_F(F^i(y))} \|}{\| DF|_{E^s_F(F^i(x))} \|}.$$
   Note that $D\pi\big(E^s_F(x)\big)=E^s_f\big(\pi(x)\big)$ for all $x\in\RR^d$. Hence $\rho_F(\cdot,\cdot)$ can descend to $\TT^d$  denoted by $\rho_f(\cdot,\cdot)$, sicne $\| DF|_{E^s_F(x)} \|=\| Df|_{E^s_f(\pi (x))} \|=\| DF|_{E^s_F(x+n)} \|$ for any $x\in\RR^d$ and $n\in\ZZ^d$.
   \begin{proposition}\label{3 prop: density function}
   	The function $\rho_F(\cdot,\cdot)$  have the following properties:
   	\begin{enumerate}
   		\item For any $x\in\RR^d$, $\rho_F(x,\cdot)$ is well defined and uniformly continuous.
   		\item  For any $y,z\in \tildeF^s(x)$,  $\rho_F(x,y) \rho_F(y,z)= \rho_F(x,z)$.
   		\item For any $y\in\tildeF^s(x)$, $\rho_F(F(x),F(y))= \frac{ \| DF|_{E^s_F(y)} \|}{\| DF|_{E^s_F(x)} \|} \rho_F(x,y) $.
   		\item The function $\rho_F(\cdot,\cdot)$ is the only uniformly continuous function satisfying $\rho_F(x,x)=1$ and the third item.
   		\item For any $K>0$, there exists $C>1$ such that, if $d_{\tildeF^s}(x,y)<K$, then $C^{-1}<\rho_F(x,y)<C$. 
   	\end{enumerate}
   \end{proposition}
   
   By lifting $\rho_f(\cdot,\cdot)$ on $\TT^d$, one  can get these properties of $\rho_F(\cdot,\cdot)$ by ones of $\rho_f(\cdot,\cdot)$ which can be found in \cite[Subsection 4.2]{GG2008}. However we would like to explain the forth item, the uniqueness of $\rho_F(\cdot,\cdot)$. Indeed if there are points $x\in\RR^d, y\in\tildeF^s(x)$ and function $\bar{\rho}(x,y)\neq \rho_F(x,y)$ satisfing the third item, then one has that 
   $$\lim_{n\to +\infty}\frac{\bar{\rho}(F^n(x),F^n(y))}{\rho_F(F^n(x),F^n(y))}=\lim_{n\to +\infty} \frac{\bar{\rho}(x,y)}{\rho_F(x,y)}\neq 1.$$
   On the other hand, since $\rho_F(x,x)=1=\bar{\rho}(x,x)$ for any $x\in\RR^d$, we have that as $n\to +\infty$, $d(F^n(x),F^n(y))\to 0$ and hence $\bar{\rho}(F^n(x),F^n(y))\to 1$, $\rho_F(F^n(x),F^n(y))\to 1$. By the uniform continuity, one has $$\lim_{n\to +\infty}\frac{\bar{\rho}(F^n(x),F^n(y))}{\rho_F(F^n(x),F^n(y))}=1,$$ which is a contradiction.

   Let $d^s_F(\cdot,\cdot)$ be a function defined on $\tildeF^s$ as 
   $$d^s_F(x,y):=\int_{x}^{y} \rho_F(z,x) d{\rm Leb}_{\tildeF^s}(z),$$
   where $y\in\tildeF^s(x)$. Note that $d^s_F(\cdot,\cdot)$ can also descend to $\TT^d$  denoted by $d^s_f(\cdot,\cdot)$.

    \begin{proposition} \label{3 prop: affine metric}
   	The function $d^s_F(\cdot,\cdot)$ satisfies the following properties:
   	\begin{enumerate}
   		\item For any $y\in\tildeF^s(x)$,  $d^s_F(x,y)=d_{\tildeF^s}(x,y)+o\left(  d_{\tildeF^s}(x,y)\right)$. That is for given $x\in\RR^d$ $$\lim_{y\to x}\ \frac{d^s_F(x,y)}{d_{\tildeF^s}(x,y)}=1,$$ and the convergence is uniform with respect to $x$.
   		\item  For any $y\in\tildeF^s(x)$,  $d^s_F(Fx,Fy)= \| DF|_{E^s_F(x)} \| \cdot d^s_F(x,y)$.
   		\item For any $K>0$, there exists $C>1$, if $d_{\tildeF^s}(x,y)<K$, 
   		then $C^{-1}\cdot d_{\tildeF^s}(x,y)\leq d^s_F(x,y )\leq C\cdot d_{\tildeF^s}(x,y)$.
   	\end{enumerate}
   \end{proposition}
   
    Proposition \ref{3 prop: affine metric} is a direct corollary of Proposition \ref{3 prop: density function}. One  can also get these properties of $d^s_F(\cdot,\cdot)$  from ones of $d^s_f(\cdot,\cdot)$ on $\TT^d$ which are given in \cite[Subsection 4.2]{GG2008}. It is similar to define $\rho_G(\cdot,\cdot),\rho_g(\cdot,\cdot)$ and $d^s_G(\cdot,\cdot), d^s_g(\cdot,\cdot)$. They also have properties stated in Proposition \ref{3 prop: density function} and Proposition \ref{3 prop: affine metric}.
  
   \vspace{2mm}
   Define an order $"\prec"$ on each leaf of  $\tildeF^s$ and $\tildeG^s$, we can assume that all of $F,G$ and $H$ preserve this order. And  $"x \preceq y"$ means that "$x\prec y$" or "$x=y$". 
   
   \begin{remark}\label{3 rmk not metric}
The function $d^s_F(\cdot,\cdot)$ is not additive or dual. So it cannot be a metric on the stable leaf. This is different from the situation of p.d.s being constant in \cite{AGGS2022}. Indeed, we have 
\begin{enumerate}
	\item For any $x\in\RR^d$ and $y\in\tildeF^s(x)$,
	\begin{align*}
		d^s_F(y,x)=\int_{y}^{x} \rho_F(z,y) d{\rm Leb}_{\tildeF^s}(z)= \rho_F(x,y)\cdot \int_{y}^{x} \rho_F(z,x) d{\rm Leb}_{\tildeF^s}(z)=\rho_F(x,y)d^s_F(x,y).
	\end{align*}
\item For any $x\in\RR^d$ and $y,w\in\tildeF^s(x)$ with $x\prec y\prec w$,
\begin{align*}
	d^s_F(x,w)&=\int_{x}^{w} \rho_F(z,x) d{\rm Leb}_{\tildeF^s}(z)=\int_{x}^{y} \rho_F(z,x) d{\rm Leb}_{\tildeF^s}(z)+\int_{y}^{w} \rho_F(z,x) d{\rm Leb}_{\tildeF^s}(z),\\
	&=d^s_F(x,y)+ \rho_F(y,x)\cdot \int_{y}^{w} \rho_F(z,y) d{\rm Leb}_{\tildeF^s}(z)= d^s(x,y)+ \rho_F(y,x)\cdot d^s_F(y,w).
\end{align*}
\end{enumerate}
   \end{remark}  
   
  To prove that $H$ is commutative with deck transformations on the assumption that $f$ and $g$ admit the same stable periodic data, we first show that 
  $H$ is bi-differentiable  along each leaf of the stable foliation $\tildeF^s$.
  The preparation step is that $H$ is Lipschitz continuous along stable leaves with the help of the affine structure and Livschitz Theorem. The following lemma is an adaption of Livschtiz Theorem for the universal cover $\RR^d$.
  
     \begin{lemma}\label{3 lem: s rigidity induce to RRd}
    	If $\lambda^s(f,\bar{p})=\lambda^s(g,\bar{h}(\bar{p}))$, for all $\bar{p}\in{\rm Per}(\sigma_f)$, then there exist  H\"older continuous $u:\TT^d_f\to\RR$ and $U:\RR^d\to \RR$ such that,  
    	$${\rm log}\| Df|_{E^s_f(\bar{x})}\| = {\rm log} \| Dg|_{E^s_g(\bar{h}(\bar{x}))}\|+u\circ \sigma_f(\bar{x})- u(\bar{x}),\quad \forall \bar{x}\in \TT^d_f,$$
    	and
    	$$ {\rm log}\| DF|_{E^s_F(x)}\|= {\rm log}\| DG|_{E^s_G(H(x))}\|+U\circ F(x)- U(x),\quad \forall x\in\RR^d.$$
    	 Moreover, $U$ is bounded and uniformly continuous on $\RR^d$.
    \end{lemma}
    \begin{proof}
    	Note that for any $x\in\RR^d$ and any $(x_i)=\bar{x}\in M_f$ with $x_0=\pi(x)$, one has $$D\pi\circ DF|_{E^s_F(x)}= Df|_{E^s_f(\pi(x))}= Df|_{E^s_f(\bar{x})}.$$
    	By Proposition \ref{2.2 prop: the relationship of 3 conjugacies inverse limit space},  
    	\begin{align*}
    		D\pi\circ DG|_{E^s_G(H(x))}
    		&= Dg|_{E^s_g(\pi\circ H(x))}= Dg|_{E^s_g(\pi\circ {\rm Orb}_G\circ H(x))},\\
    		&=  Dg|_{E^s_g(\bar{h}\circ\pi \circ {\rm Orb}_F(x))}= Dg|_{E^s_g(\bar{h}(\bar{x}))}.
    	\end{align*}
    	
    	Since Anosov maps on $d$-torus  are always transitive (\cite{AokiHiraide1994}, also see \cite[Section 4]{AGGS2022}), then by the stable periodic data condition and the Livschitz Theorem on the inverse limit space (Theorem \ref{2.3 thm: Livshitz for inverse limit space}), there exists 
    	a  H$\ddot{\rm o}$lder continuous $u:\TT^d_f\to\RR$ such that 
    	$${\rm log}\| Df|_{E^s_f(\bar{x})}\| = {\rm log} \| Dg|_{E^s_g(\bar{h}(\bar{x}))}\|+u\circ \sigma_f(\bar{x})- u(\bar{x}).$$
    	Since $\bar{h}$ is given by Proposition \ref{2.2 prop: the relationship of 3 conjugacies inverse limit space}, it follows that  
    	$$ {\rm log}\| DF|_{E^s_F(x)}\|= {\rm log}\| DG|_{E^s_G(H(x))}\|+U\circ F(x)- U(x) ,$$
    	where the function $U:\RR^d\to \RR$ defined as $U(x)=u\circ \pi\big({\rm Orb}_F(x)\big)$ is bounded and  H$\ddot{\rm o}$lder continuous.
    	Note that $\pi\circ{\rm Orb}_F: \big(\RR^d, d(\cdot,\cdot) \big)\to \big(\TT^d_f, \bar{d}(\cdot,\cdot)\big)$ is H$\ddot{\rm o}$lder continuous.
    	Indeed if we choose $\bar{d}$ related with hyperbolic rate such as 
    	$$\bar{d}(\bar{x},\bar{y}):=\sum_{i=-\infty}^{+\infty}\frac{d(x_i,y_i)}{(2\lambda)^{|i|}},$$
    	where $\lambda:={\rm sup}_{\bar{x}\in\TT^d_f}\big\{\|Df^{-1}|_{E^s_f(\bar{x})}\|,   \|Df|_{E^u_f(\bar{x})}\| \big\}$,
    	then  $\pi\circ{\rm Orb}_F$ would be Lipschitz continuous. 
    \end{proof}
    
    Now we will prove the Lipschitz continuity of $H$ restricted on each stable leaf.
    \begin{proposition}\label{3 prop: bi-Lipschitz}
   	There exists $C_0>0$ such that for every $x\in\RR^d$ and $y\in\tildeF^s(x)$, one has $$d_{\tildeG^s}\big(H(x),H(y) \big)< C_0\cdot d_{\tildeF^s}(x,y).$$
   \end{proposition}
   
   \begin{proof}
   By Proposition \ref{2.1 prop lifting conjugate}, there exists $C_1$ such that $\|H-{\rm Id}_{\RR^d}\|_{C^0}<C_1$. Then for any $x,y\in\RR^d$ with $d(x,y)\geq C_1$, one has $$d\big(H(x),H(y)\big)<d(x,y)+2C_1<3d(x,y).$$
   Note that the one-dimensional foliation $\tilde{\mathcal{F}}_{\sigma}^s$ $(\sigma=f,g)$ is quasi-isometric (see Proposition \ref{2.1 prop quasi-isometric}), i.e., there exists $C_2>1$ such that $$d_{\tilde{\mathcal{F}}_{\sigma}^s}(x,y)\leq C_2\cdot d(x,y), \quad \forall x\in\RR^d, \forall y\in\tilde{\mathcal{F}}_{\sigma}^s(x).$$
   Hence let $x\in\RR^d$ and $y\in\tildeF^s(x)$ with $d_{\tildeF^s}(x,y)\geq C_1\cdot C_2$. We have
   \begin{align}
   	 d_{\tildeG^s}\big(H(x),H(y)\big)<C_2\cdot d\big(H(x),H(y)\big)<3C_2\cdot d(x,y)\leq 3C_2 \cdot d_{\tildeF^s}(x,y). \label{eq. 3.5.1}
   \end{align}

  Now for every $x\in\RR^d$ and $y\in\tildeF^s(x)$ with $d_{\tildeF^s}(x,y)<C_1\cdot C_2$, we can choose the smallest $N\in\NN$ such that $d_{\tildeF^s}\big(F^{-N}(x),F^{-N}(y)\big)$ is bigger than $C_1\cdot C_2$. It follows  that 
  \begin{align}
  	C_1\cdot C_2 <d_{\tildeF^s}\big(F^{-N}(x),F^{-N}(y)\big)\leq C_1\cdot C_2\cdot \sup_{z\in\TT^d}\|Df|_{E_f^s(z)}\|^{-1}. \label{eq. 3.3.5.1}
  \end{align}
  Again by $\|H-{\rm Id}_{\RR^d}\|_{C^0}<C_1$ and the quasi-isometric property of $\tildeG^s$, there exists $C_3>0$ such that 
  \begin{align}
  	 d_{\tildeG^s}\big(G^{-N}\circ H(x)\;,\; G^{-N}\circ H(y)\big)\leq C_3.\label{eq. 3.3.5.2} 
  \end{align}
  By the affine structure and the Livschitz Theorem, one has
 	\begin{align}
 	\frac{d_G^s\big(H(x),H(y)\big)}{d^s_F(x,y)}
 	&=\prod_{i=-N}^{-1}\frac{\|DG|_{E^s_G(G^i\circ H(x))}\|}{\|DF|_{E^s_F(F^i(x))}\|} \cdot \frac{d_G^s\big(G^{-N}\circ H(x) \; ,\;   G^{-N}\circ H(y)\big)}{d_F^s\big( F^{-N}(x) \; ,\;   F^{-N}(y)\big)},\\
 	&= \frac{P(F^{-N}(x))}{P(x)}\cdot \frac{d_G^s\big(G^{-N}\circ H(x) \; ,\;  G^{-N}\circ H(y)\big)}{d_F^s\big( F^{-N}(x) \; ,\;   F^{-N}(y)\big)}, \label{eq. 3.5.2}
 \end{align}
 where $P(x)=e^{U(x)}$ and $U(x)$ is given by Lemma \ref{3 lem: s rigidity induce to RRd}. Note that $P(x)$ is uniformly away from $0$ and upper bounded. By \eqref{eq. 3.3.5.1} and \eqref{eq. 3.3.5.2}, it follows from the third item of Proposition \ref{3 prop: affine metric} for both $d_F^s(\cdot,\cdot)$ and $d_G^s(\cdot,\cdot)$ that there exists $C_4>0$ such that
 $$\frac{d_{\tildeG^s}\big(H(x),H(y) \big)}{d_{\tildeF^s}(x,y)} \leq C_4\cdot \frac{d_{\tildeG^s}\big(G^{-N}\circ H(x), G^{-N}\circ H(y) \big)}{d_{\tildeF^s}\big(F^{-N}(x),F^{-N}(y)\big)}.$$
 Finally, by \eqref{eq. 3.5.1}, there exists $C_0>0$ such that
  $d_{\tildeG^s}\big(H(x),H(y) \big) <C_0 \cdot d_{\tildeF^s}(x,y)$.
   \end{proof}
 
 \begin{remark}\label{3 rmk H almost differentiable}
 Similarly,  we have that $H^{-1}$ is also Lipschitz continuous along stable leaves. It follows that there exist  bi-differentiable points which form a full Lebesgue set restricted on a local stable leaf.  
 \end{remark}
 
 \vspace{3mm}
 For proving the total differentiability of $H$ restricted on $\tildeF^s$, we need some compactness which can be gotten on the inverse limit space. Therefore, we will discuss the "differentiability"  on $\TT^d_f$. 
 
 Recall that $\pi_0:(\TT^d)^{\ZZ}\to \TT^d$ is the projection on the $0$th position.   Define $\bar{h}_0:\TT^d_f\to \TT^d$ as  $\bar{h}_0(\tilde{x})=\pi_0\circ \bar{h}(\tilde{x})$.  Let $\tilde{x}=(x_i),\tilde{y}=(y_i)\in\TT^d_f$,  we use $\tilde{y}\in \mathcalf^s(\tilde{x},R)$ to denote $y_i\in\mathcalf^s(x_i)\ (i\in\ZZ)$ and $y_0\in\mathcalf^s(x_0,R)$. It is clear that $\bar{h}_0(\tilde{y})\in \mathcalg^s(\bar{h}_0(\tilde{x}))$. By Lemma \ref{3 lem: s rigidity induce to RRd} and the proof of Proposition \ref{3 prop: bi-Lipschitz}, we have the following corollary. 
 
 \begin{corollary}\label{3 cor h0 almost differentiable}
 	There exists $C_0>1$ such that for every $\tilde{x}\in\TT^d_f$ and $\tilde{y}\in\mathcalf^s(\tilde{x})$, one has 
 	$$C_0^{-1}<\frac{d_{\mathcalg^s}\big(\bar{h}_0(\tilde{x}),\bar{h}_0(\tilde{y}) \big)}{d_{\mathcalf^s}\big(\pi_0(\tilde{x}),\pi_0(\tilde{y})\big)}<C_0.$$
   Hence, for every $\tilde{x}\in\TT^d_f$, there is a full Lebesgue set $\Gamma(\tilde{x})$ on $\mathcalf^s(\pi_0(\tilde{x}), R)$ for some $R>0$ such that for any $ y\in \Gamma(\tilde{x})$ and the unique point $\tilde{y}\in \mathcalf^s(\tilde{x},R)$ with $\pi_0(\tilde{y})=y$, one has that the restriction $\bar{h}_0|_{\mathcalf^s(\pi_0(\tilde{x}))} : \mathcalf^s(\pi_0(\tilde{x})) \to \mathcalg^s(\bar{h}_0(\tilde{x}))$ by $y \mapsto \bar{h}_0(\tilde{y})$ is differentiable at $y$. Denote the derivative by $D\bar{h}_0|_{E^s_f(\tilde{y})}$.
 \end{corollary}
 
  Let $\tilde{P}(\tilde{x})=e^{u(\tilde{x})}$, for any $\tilde{x}\in \TT^d_f$ where $u:\TT^d_f\to \RR$ is given by Lemma  \ref{3 lem: s rigidity induce to RRd}. Note that $\tilde{P}$ is H\"older continuous, bounded and uniformly away from $0$. By a method originated from \cite{GG2008}, we can get the following proposition from Corollary \ref{3 cor h0 almost differentiable}. 
 \begin{proposition}\label{3 prop h0 differentiable}
 	For any $\tilde{x}\in\TT^d_f$, the derivative  $D\bar{h}_0|_{E^s_f(\tilde{x})}$ exists, moreover $$\|D\bar{h}_0|_{E^s_f(\tilde{y})}\|=\frac{\tilde{P}(\tilde{x})}{\tilde{P}(\tilde{y})}\cdot\|D\bar{h}_0|_{E^s_f(\tilde{x})}\|,\quad \forall \tilde{y}\in\TT^d_f.$$
 \end{proposition}
 
  Although the proof of Proposition \ref{3 prop h0 differentiable}  seems like to a version of Anosov diffeomorphisms in \cite{GG2008}, there is something quite different. For the cohesion of reading, we leave the proof to the Appendix. Since $\bar{h}$ is induced by $H$ in the way of Proposition \ref{2.2 prop: the relationship of 3 conjugacies inverse limit space}, we have the following corollary from Proposition \ref{3 prop h0 differentiable}, immediately. 
 
 \begin{corollary}\label{3 cor H is differentiable}
 	For any $x\in\RR^d$, $H$ is differentiable at $x$ along $\tildeF^s(x)$ and 
 	$$\|DH|_{E^s_F(y)}\|=\frac{P(x)}{P(y)}\cdot\|DH|_{E^s_F(x)}\|,\quad \forall y\in\RR^d,$$
 	where $P(x)=e^{U(x)}$ and $U(x)$ is given by Lemma \ref{3 lem: s rigidity induce to RRd}. 
 \end{corollary}

 \vspace{2mm}
 By Corollary \ref{3 cor H is differentiable}, we can obtain the following properties of the derivative of $H$ along $\tildeF^s$.

\begin{proposition}\label{3 prop: property of differentible points}
The conjugacy $H$ restricted on $\tildeF^s$ satisfy the following there properties:
\begin{enumerate}
	\item (Boundedness)	There exists $C>1$ such that for every  $x\in\RR^d$, 
	$$ C^{-1}< \|DH|_{E^s_F(x)}\|<C.$$
	\item (Uniform continuity)	For every $\e>0$, there exists $\delta>0$ such that for any $x,y\in\RR^d$  with $d(x,y)<\delta$, $$\Big|  \|DH|_{E^s_F(x)}\|-\|DH|_{E^s_F(y)}\|  \Big|<\e.$$ 
	\item (Uniform convergence) 	For every $\e>0$, there exists $\delta>0$ such that for any $x\in\RR^d$ and $y\in\tildeF^s(x,\delta)$,  $$\Big|  \frac{d_{\tildeG^s}(H(x),H(y))}{d_{\tildeF^s}(x,y)}-\|DH|_{E^s_F(x)}\|  \Big|< \e.$$
\end{enumerate}
\end{proposition}

  \begin{proof}
  	Fix $x_0\in\RR^d$, by Corollary \ref{3 cor H is differentiable}, we have that 	$$\|DH|_{E^s_F(y)}\|=\frac{P(x_0)}{P(y)}\cdot\|DH|_{E^s_F(x_0)}\|,\quad \forall y\in\RR^d.$$
  	Note that $P(x)=e^{U(x)}$ is uniformly continuous, bounded and uniformly away from $0$. Then one can immediately get the first two items.

  	For the third item. Since  for any $\e>0$, there exists $\delta>0$ such that $\Big|  \|DH|_{E^s_F(x)}\|-\|DH|_{E^s_F(z)}\|  \Big|<\e,$ for any  $x$ and  $z\in\tildeF^s(x,\delta)$ by the second item. Then by
  	$$d_{\tildeG^s}(H(x),H(y))=\int_{x}^{y} \|DH|_{E^s_F(z)}\| {\rm d} {\rm Leb}_{\tildeF^s} (z),$$
    one has that
  	$$\frac{d_{\tildeG^s}(H(x),H(y))}{d_{\tildeF^s}(x,y)}\in \big[ \|DH|_{E^s_F(x)}\|-\e  \;,\;   \|DH|_{E^s_F(x)}\|+\e  \big].$$
  	This completes the proof.
  	  \end{proof}

Now, we prove Theorem \ref{2 thm s-rigidity implies conjugacy}.

\begin{proof}[Proof of Theorem \ref{2 thm s-rigidity implies conjugacy}]
	Recall that $H(x+n)-n\in\tildeG^s\big(  H(x) \big)$ for all $x\in\RR^d$ and $n\in\ZZ^d$ by Proposition \ref{2.1 prop H is Zd restricted on stable}. Assume by contradiction that there exist $x_0\in\RR^d$ and $n\in\ZZ^d$ such that 
	\begin{align}
		d_{\tildeG^s} \big( H(x_0) \;,\; H(x_0+n)-n \big)=\alpha>0. \label{eq. 3. assumption}
	\end{align}
	Without loss of  generality, we assume that $H(x_0)\prec H(x_0+n)-n$.
	
	\begin{lemma}\label{3 lem: control the deviation}
	There exists $C_0>1$ such that for all point $z\in\tildeF^s(x_0)$, if  $d_{\tildeF^s}\left(x_0, z \right)>C_0$, then
	$$d_{\tildeG^s}\big(H(z+n)-n, H(z)  \big)>0,$$
	moreover, one has $ H(z) \prec H(z+n)-n$.
	\end{lemma}
	\begin{proof}[Proof of Lemma \ref{3 lem: control the deviation}]
	Without loss of generality, we assume that $x_0\prec z$. For short, let $$y_0=H^{-1}(H(x_0)+n)\quad {\rm and}\quad w=H^{-1}(H(z)+n).$$
	It is clear that $y_0, w\in\tildeF^s(x_0+n)$ and $y_0\prec x_0+n\prec w$. Indeed, since $x_0\prec z$, we have $x_0+n \prec z+n$ and $H(x_0)+n\prec H(x_0+n)\prec H(z+n)$.  Moreover, since $H$ is bounded from Id$_{\RR^d}$ and $\tilde{\mathcal{F}}^s_{\sigma}$ $(\sigma=f,g)$ is quasi-isomertic, there exists $C_0>0$ such that if $d_{\tildeF^s}(x,z)>C_0$, one has $H(x_0)+n\prec H(x_0+n)\prec H(z)+n$.
	 
	 Now, we are going to prove $w\prec z+n$ and it completes the proof of this lemma. By the third item of Proposition \ref{3 prop: property of differentible points} and Proposition \ref{3 prop: affine metric},  for each $\e>0$ there exists $K\in\NN$ such that if $k>K$, we have 
	\begin{align}
		\Big|\frac{d^s_G\left(    G^k\circ H(x_0),  G^k\circ H(z)   \right)}{d^s_F\left(    F^k(x_0),  F^k(z)   \right)} - \|DH|_{E^s_F(F^k(x_0))}\|   \Big|<\e, \label{5.2.1}
	\end{align}
    and
    \begin{align}
    	\Big|\frac{d^s_G\left(    G^k\circ H(y_0),  G^k\circ H(w)   \right)}{d^s_F\left(    F^k(y_0),  F^k(w)   \right)} - \|DH|_{E^s_F(F^k(y_0))}\|   \Big|<\e. \label{5.2.2}
    \end{align}
   
   \begin{claim}\label{3 claim: nm DH close}
   	For every $\e>0$, there exists $K>0$ such that when $k>K$,
   	\begin{align}
   		\Big|  \|DH|_{E^s_F(F^k(x_0))}\|  -\|DH|_{E^s_F(F^k(y_0))}\|   \Big| <\e.\label{5.2.3}
   	\end{align}
   \end{claim}
   \begin{proof}[Proof of Claim \ref{3 claim: nm DH close}]
   	Since $F^k(y_0)= F^k \big( H^{-1}(H(x_0)+n) \big) =H^{-1} \big( H\circ F^k(x_0) + A^kn\big)$, then by Proposition \ref{2.1 prop nmH}, we have 
   	$$d\Big(F^k(y_0)\;,\;  H^{-1} \big( H\circ F^k(x_0) \big) + A^kn  \Big) \to 0,$$
   	as $k\to +\infty$.
   	It follows that $d\big(F^k(y_0)\;,\;   F^k(x_0)  + A^kn  \big) \to 0,$ as $k\to +\infty$.
   By  the second item of Proposition \ref{3 prop: property of differentible points}, it suffices to show that for any $\e>0$, there exists $K\in\NN$ such that when $k>K$,
   \begin{align}
   	\Big|  \|DH|_{E^s_F(x)}\|  -\|DH|_{E^s_F(x+A^kn)}\|   \Big| <\e,\label{eq. 3.11.1}
   \end{align}
   	for any $x\in\RR^d$ and any $n\in\ZZ^d$. 
   	
   	Indeed, by the uniform convergence, for fixed $\e>0$, we take $\delta>0$ in the third item of Proposition \ref{3 prop: property of differentible points}. For this $\delta$,  there is $K\in\NN$ such that $$\Big|d_{\tildeG^s}\big(H(x),H(x')\big)- d_{\tildeG^s}\big(H(x+A^kn),H(x'+A^kn)\big)\Big|<\frac{\delta}{2}\cdot \e,$$ for any $k>K$ and any $x\in\RR^d, x'\in\tildeF^s(x)$ with $d_{\tildeF^s}(x,x')=\frac{\delta}{2}$. Since 
   	 $$\Big|  \frac{d_{\tildeG^s}(H(x),H(x'))}{d_{\tildeF^s}(x,x')}-\|DH|_{E^s_F(x)}\|  \Big|< \e,$$ and $$ \ \ \Big|  \frac{d_{\tildeG^s}\big(H(x+A^kn),H(x'+A^kn)\big)}{d_{\tildeF^s}(x+A^kn,x'+A^kn)\big)}-\|DH|_{E^s_F(x+A^kn)}\|  \Big|< \e,$$
   	 then one has that 
   \begin{align*}
   	\Big|  \|DH|_{E^s_F(x)}\|  -\|DH|_{E^s_F(x+A^kn)}\|   \Big|&< 2\e+
   	\Big|  \frac{d_{\tildeG^s}\big(H(x+A^kn),H(x'+A^kn)\big)}{d_{\tildeF^s}(x+A^kn,x'+A^kn)\big)}-\frac{d_{\tildeG^s}(H(x),H(x'))}{d_{\tildeF^s}(x,x')}\Big|\\
   	&<2\e+ \big(\frac{\delta}{2}\cdot \e\big) / \frac{\delta}{2}=3\e.
   \end{align*}
It compelets the proof of this claim.
   \end{proof}
   
   It is clear that for any $k\in\NN$, 
   \begin{align*}
   	d^s_F\left(    F^k(x_0),  F^k(z)   \right)= d^s_F\left(    F^k(x_0+n),  F^k(z+n)   \right), \ \  d^s_G\left(    G^k \circ H(x_0),  G^k\circ H(z)   \right)= d^s_G\left(   G^k\circ H(y_0),  G^k \circ H(w)   \right)
   \end{align*}
   Hence, combining \eqref{5.2.1} with \eqref{5.2.2} and \eqref{5.2.3}, we have
    \begin{align}
   \frac{\|DH|_{E^s_F(F^k(x_0))}\| -2\e}{\|DH|_{E^s_F(F^k(x_0))}\| +\e} < \frac{d^s_F\left(    F^k(x_0+n),  F^k(z+n)   \right) }{d^s_F\left(    F^k(y_0),  F^k(w)   \right) } < \frac{\|DH|_{E^s_F(F^k(x_0))}\| +2\e}{\|DH|_{E^s_F(F^k(x_0))}\| -\e}.
   \end{align}
   Since $\|DH|_{E^s_F(F^k(x_0))}\|$ is bounded and uniformly away from $0$ (see the first item of Proposition \ref{3 prop: property of differentible points}), when $\e\to 0$ and $k\to+\infty$, one has 
   \begin{align}
   	 \frac{d^s_F\left(    F^k(x_0+n),  F^k(z+n)   \right) }{d^s_F\left(    F^k(y_0),  F^k(w)   \right) } \longrightarrow 1. \label{5.2.6}
   \end{align}
     On the other hand, since $y_0\prec x_0+n\prec w$ and by Remark \ref{3 rmk not metric},
   \begin{align*}
   	d^s_F\left(    F^k(y_0),  F^k(w)   \right)  
   	&= d^s_F\left(    F^k(y_0),  F^k(x_0+n)   \right) \\
   	&\qquad  +  \rho_F \left( F^k(x_0+n),F^k(y_0)\right) \cdot d^s_F\left(    F^k(x_0+n),  F^k(w)   \right),\\
   	&=  \rho_F \left(F^k(x_0+n), F^k(y_0)\right) \cdot d^s_F\left(    F^k(x_0+n),  F^k(y_0)   \right) \\
   	& \qquad  +  \rho_F\left( F^k(x_0+n),F^k(y_0)\right) \cdot d^s_F\left(    F^k(x_0+n),  F^k(w)   \right).
   \end{align*}
  Then by \eqref{5.2.6} and $d^s_F\left(    F^k(x_0+n),  F^k(x)   \right) =\prod_{i=0}^{k-1}\|DF|_{E_F^s(F^i(x_0+n))}\| \cdot d^s_F(x_0+n,x)$ for all $x\in\tildeF^s(x_0+n)$, we have 
   \begin{align}
  	\frac{d^s_F\left( x_0+n,  z+n   \right) }{ \rho_F \left(F^k(x_0+n), F^k(y_0)\right) \cdot \Big(d^s_F \left(   x_0+n,  y_0  \right) +  d^s_F\left(    x_0+n,  w   \right)\Big)} \longrightarrow 1. \label{5.2.7}
  \end{align}
  Since $\bar{d}\left( \pi\circ {\rm Orb}_F(F^k(x_0)), \pi\circ {\rm Orb}_F(F^k(y_0)) \right) \to 0$ as $k\to +\infty$,  then
  $\rho_F \left(F^k(x_0+n), F^k(y_0)\right) \to 1$  and 
  $$ d^s_F\left( x_0+n,  z+n   \right) = d^s_F \left(   x_0+n,  y_0  \right) +  d^s_F\left(    x_0+n,  w   \right). $$   
   It follows that $y_0\prec x_0+n\prec w\prec z+n$, moreover,
   \begin{align*}
   	d^s_F \left(   x_0+n,  y_0  \right) +  d^s_F\left(    x_0+n,  w   \right)
   	&=d^s_F\left( x_0+n,  z+n   \right) ,\\
   	&= d^s_F\left( x_0+n, w  \right) +\rho_F(w, x_0+n )d^s_F\left( w,  z+n   \right). 
   \end{align*}
   Hence
   $ d^s_F \left(   x_0+n,  y_0  \right) =\rho_F(w, x_0+n )d^s_F\left( w,  z+n   \right)$.
   
   We mention that for the case of $H(x_0)+n\prec H(x_0+n)$ and  $z\prec x_0$, by the same way, one has $$d_F^s(y_0,w)=d^s_F(y_0,x_0+n)+d^s_F(y_0,z+n).$$ It follows that $w\prec z+n\prec y_0\prec x_0+n$ and  $d^s_F(y_0,x_0+n)=\rho_F(z+n,y_0)d^s_F(z+n,w)$.
  
	\end{proof}
   
   Now we continue to prove Theorem \ref{2 thm s-rigidity implies conjugacy}.
   Since $A=f_*$ is irreducible, one has that the foliation $\mathcal{L}^s$ is minimal, i.e., $\mathcal{L}^s(x)$ is dense in $\TT^d$ for every $x\in\TT^d$. Then by the leaf conjugacy between $f$ and $A$ (see Proposition \ref{2.2 prop leaf conjugate}), $\mathcalf^s$ is also minimal. Combining the minimal foliation with Lemma \ref{3 lem: control the deviation}, we can get a stronger lemma as follow.

  \begin{lemma}\label{3 lem: keepdeviation}
  For any $y\in \RR^d$, $H(y)\preceq H(y+n)-n$.
  \end{lemma}

\begin{proof}[Proof of Lemma \ref{3 lem: keepdeviation}]
We assume that $H(y)\neq H(y+n)-n$, it suffices to prove $H(y)\prec H(y+n)-n$. Let $C_0$ is  given by Lemma \ref{3 lem: control the deviation}. By Proposition \ref{2.1 prop lifting conjugate} and Proposition \ref{2.1 prop nmH},  for any $\e>0$, there exists $\delta>0$ and $N\in\NN$ such that 
\begin{enumerate}
	\item For any $x,y\in\RR^d$ with $d(x,y)<\delta$, one has $d\big(H(x),H(y)\big)<\frac{\e}{2}$.
	\item For any $x\in\RR^d, m\geq N$ and $k_m\in A^m\ZZ^d$, one has $d\big( H(x+k_m), H(x)+k_m  \big)<\frac{\e}{2}$.
\end{enumerate}

Let $\e>0$ satisfy $d\big( H(y), H(y+n)-n  \big)>4\e$. Then we can fix $\delta>0$ and $N\in\NN$ as above. Meanwhile, by the minimality of $\mathcalf^s$, for any $y\in\RR^d$, there exist $z_m\in\tildeF^s(x_0)$ with $d_{\tildeF^s}(z_m,x_0)>C_0$ and $k_m\in A^N\ZZ^d$ such that $d(z_m+k_m , y)<\delta$. Then, we have 
$$d\big(H(z_m+k_m), H(y)\big)<\frac{\e}{2}\quad  {\rm and }\quad d\big(H(z_m)+k_m, H(z_m+k_m)\big)<\frac{\e}{2}.$$ 
Hence, 
\begin{align}
	d\big(H(z_m)+k_m, H(y)\big)<\e.\label{eq claim3.11-1}
\end{align}
Meanwhile, since $d(z_m+k_m+n, y+n)<\delta$, we also have $$d\big(H(z_m+n)+k_m, H(y+n)\big)<\e.$$
Hence,
\begin{align}
	d\big(H(z_m+n)-n+k_m, H(y+n)-n\big)<\e.\label{eq claim3.11-2}
\end{align}
By Lemma \ref{3 lem: control the deviation}, one has 
$H(z_m)\prec H(z_m+n)-n$,
it follows that 
\begin{align}
	H(z_m)+k_m\prec H(z_m+n)-n+k_m.\label{eq claim3.11-3}
\end{align}
Since $d\big( H(y), H(y+n)-n  \big)>4\e$, combining with $\eqref{eq claim3.11-1}$, $\eqref{eq claim3.11-2}$ and $\eqref{eq claim3.11-3}$, one has  $$H(y)\prec H(y+n)-n.$$
\end{proof}

  Apply Lemma \ref{3 lem: keepdeviation} for $y=x_0+ln$, for all $l\in\NN$, then we have 
  $$   H(x_0+ln) \preceq H(x_0+ln+n)-n,$$
  it follows that 
  $$  H(x_0+ln)-ln \preceq H(x_0+ln+n)-(l+1)n.$$
  Then by the assumption \eqref{eq. 3. assumption}, we have that for all $l\in\NN$,
  \begin{align}
  	  d_{\tildeG^s} \big( H(x_0+ln)-ln \;,\; H(x_0)\big)\geq \alpha.\label{eq theorem3.1-1}
  \end{align}
  Now let $l_k=|{\rm det}(A)|^k$, for any $k\in\NN$. Then one has $l_kn\in A^k\ZZ^d$. By Proposition \ref{2.1 prop nmH},  one has  $$d_{\tildeG^s} \big( H(x_0+l_kn)-l_kn \;,\; H(x_0)\big)\to 0, \quad{\rm as}\;\; k\to+\infty.$$
  It contradicts with $\eqref{eq theorem3.1-1}$.
 
  Consequently, all $x\in\RR^d$ and all $n\in\ZZ^d$ must be $H(x+n)-n=H(x)$. 
  It means that $H$ is commutative with deck transformations and hence can descend to $h:\TT^d\to\TT^d$ such that $h\circ \pi=\pi\circ H$. Then $h$ is a topological conjugacy on $\TT^d$ between $f$ and $g$. Moreover,  $H$ is differentiable along each leaf of stable foliation $\tildeF^s$, so is $h$ along $\mathcalf^s$.
\end{proof}

   \section{Accessibility of non-invertible Anosov maps}\label{sec accessible}
   Let $f\in\mathcal{A}^1(M)$. Recall the set $Acc(x,f)$, the $u$-accessible class of $x\in M$, collects points which can reach $x$ by finite unstable manifolds (see Definition \ref{2 def u-accessible class}).  We say $f$ is $u$-accessible, if $Acc(x,f)=M$ for some $x$.
   In this section, we prove Proposition \ref{1 prop dichotomy of non-invertible Anosov map} which says that for $f\in\mathcal{A}^1_1(\TT^d)$, if $f$ is not special, then it is $u$-accessible.  Moreover, in the  $u$-accessible case, we show that for any two points on $\TT^d$, there exist at most four local unstable manifolds with  uniformly bounded size connecting them (see Proposition \ref{4 prop: dichotomy of non-invertible Anosov map}). We refer readers to \cite{HertzUres2019} for  a class Anosov diffeomorphisms with one-dimensional center bundles, where two points can be connected by at most three stable and unstable manifolds.
   
   We  first consider $f\in\mathcal{A}^1_1(M)$.  Recall that $M$ is a $C^{\infty}$-smooth close Riemmanian manifold. Let $\tilde{M}$ be the universal cover of $M$ and $\pi:\tilde{M}\to M$ is the natural projection.  The following lemma focuses on the local property of $u$-accessible classes.

   \begin{lemma}\label{4 lemma: s-one accessible  open}
   	Let $f\in\mathcal{A}^1_1(M)$. If $f$ is not special, then there exists $x_0\in M$ such that $x_0\in {\rm Int}\;Acc(x_0,f)$.  Moreover, we have the following two properties:
   	\begin{enumerate}
   		\item    There exists an open neighborhood $V_0$ of $x_0$ and constant $r>0$, such that for any $y,z\in V_0$, there exist $\tilde{y},\tilde{z}\in M_f$ with $(\tilde{y})_0=y$ and $(\tilde{z})_0=z$ such that $$\mathcalf^u(y,\tilde{y},r)\cap \mathcalf^u(z,\tilde{z},r) \neq \emptyset.$$
   		\item	There exist two points $x^*,x_*\in \tilde{M}$ with $\pi(x^*)=\pi(x_*)=x_0$, an open set $U_0\subset M$ and a constant $r>0$ such that  for any $z\in U_0$ and $\tilde{z} \in M_f$ with $(\tilde{z})_0=z$,  $$\mathcalf^u(x_0,\bar{x}^*,r)\cap \mathcalf^u(z,\tilde{z},r) \neq \emptyset \quad {\rm or} \quad \mathcalf^u(x_0,\bar{x}_*,r)\cap \mathcalf^u(z,\tilde{z},r) \neq \emptyset, $$
   		where  $\bar{x}^*= \pi\big( {\rm Orb}_F(x^*)\big)$ and $\bar{x}_*= \pi\big( {\rm Orb}_F(x_*)\big)$. 
   	\end{enumerate}
   	
   \end{lemma}

   \begin{proof}
   	Since $f$ is not special and dim$E^s=1$, there exist a point $x_0\in M$ and two orbit $\bar{x}^*, \bar{x}_*\in M_f$ such that $E^u_f(x_0,\bar{x}^*)+E^u_f(x_0,\bar{x}_*)=T_xM$. As the same reason of \eqref{eq. 2.2. R orbit to T}, one has
   	\begin{align}
   		\bigcup_{\tilde{x}\in M_f , (\tilde{x})_0=x} E^s_f(x,\tilde{x})= \overline{\bigcup_{y\in \tilde{M}, \pi(y)=x} D\pi \big(\tilde{E}^s_f(y)\big)}.\label{eq. 2.23.1}
   	\end{align}  
   	Then we can assume that $\bar{x}^*$ and $\bar{x}_*$ are both projections of two $F$-orbits for points $x^*\in \tilde{M}$ and $x_*\in \tilde{M}$ respectively. Namely, $\bar{x}^*= \pi\big( {\rm Orb}_F(x^*)\big)$ and $\bar{x}_*= \pi\big( {\rm Orb}_F(x_*)\big)$. 
   	
   	Let $\angle(\cdot,\cdot)$ be the angle between two subspaces in $\RR^d$ defined as 
   	$$\angle(A_1,A_2):=\max \Big\{ \max_{v\in A_1, \|v\|=1} \big\{ \min_{w\in A_2}\|v-w\|\big\}\;,\;   \max_{v\in A_2, \|v\|=1} \big\{\min_{w\in A_1}\|v-w\| \big\}  \Big\}.$$
   	Assume that $$\angle\big(E^u_f(x_0,\bar{x}^*), E^u_f(x_0,\bar{x}_*)\big)=\alpha>0.$$
   	
   	By the continuity of unstable directions with respect to $f$-orbits ( see Proposition \ref{2.1 prop: unstable leaf continuity}), for any $\e>0$ there exists $r'>0$ such that 
   	for every $z\in\mathcalf^s(x_0,r')$, there exists $\bar{z}^*\in M_f$ with $(\bar{z}^*)_0=z$ satisfying
   	$$\angle\big(E^u_f(z,\bar{z}^*), E^u_f(x_0,\bar{x}_*)\big)<\alpha/4.$$
   	This implies that $E^u_f(z,\bar{z}^*)$ intersects with $E^u_f(x_0,\bar{x}_*)$ as two subspaces in $\RR^d$ since dim$E^s_f=1$. Note that the unstable manifolds $\mathcalf^u(x_0,\bar{x}^*,r)$ and $\mathcalf^u(x_0,\bar{x}_*,r)$ are $C^1$-smooth and tangent to $E^u_f(x_0,\bar{x}^*)$ and $E^u_f(x_0,\bar{x}_*)$ respectively. Hence, fix $r>0$ and take $\e>0$ small enough, there exists $r'>0$ such that for every $z\in\mathcalf^s(x_0,r')$, there exists $\bar{z}^*\in M_f$ with $(\bar{z}^*)_0=z$ satisfying 	$$ \mathcalf^u(z,\bar{z}^*,r)\cap   \mathcalf^u(x_0, \bar{x}_*,r) \neq \emptyset.$$
   	In particular, we can choose $z^*\in \tildeF^s(x^*,r')$ such that $\bar{z}^*=\pi\big( {\rm Orb}_F(z^*)\big)$.

   	It is clear that $x_0$ is an interior point of the set $$U(x_0):=\bigcup_{ z\in\mathcalf^s(x_0,r')} \mathcalf^u(z,\bar{z}^*,r),$$ and $U(x_0)\subset Acc(x_0,f)$. Hence, $x_0\in {\rm Int}\; Acc(x_0,f)$.

   	For the first item, again by Proposition \ref{2.1 prop: unstable leaf continuity}, there exists a neighborhood $V_0$ of $x_0$ such that for every point $y\in V_0$, there exist two orbit $\bar{y}^*, \bar{y}_* \in M_f$ such that $$\angle\big(E^u_f(x_0,\bar{x}^*), E^u_f(y,\bar{y}^*)\big)<\alpha/4 \quad {\rm and} \quad \angle\big(E^u_f(x_0,\bar{x}_*), E^u_f(y,\bar{y}_*)\big)<\alpha/4.$$
   	It follows that for any $y,z\in V_0$ there exist $\bar{y}^*, \bar{z}_*\in M_f$ such that 
   	\begin{align}
   		\angle\big(E^u_f(y,\bar{y}^*), E^u_f(z,\bar{z}_*)\big)>\alpha/2. \label{eq.4.1.1}
   	\end{align}
   	Up to shrinking $V_0$, \eqref{eq.4.1.1} implies $\mathcalf^u(y,\bar{y}^*,r)\cap \mathcalf^u(z,\bar{z}_*,r)\neq \emptyset$.

   	For the second item, let $B(x_0,\delta)\subset U(x_0)$ be a $\delta$-ball of $x_0$. Then there exists $\delta'>0$ such that $B(x_0,\delta)$ is divided into four connected components by $U_1:=\mathcalf^u(x_0, \bar{x}^*,\delta')$ and  $U_2:=\mathcalf^u(x_0, \bar{x}_*,\delta') $. Indeed,  we give an orientation to the local stable leaf $\mathcalf^s(x_0,r')$.  Then by the Local Product Structure, $U_1$ cuts $B(x_0,\delta)$ into two components $U_1^+$ and $U_1^-$, where $U_1^+$ contains the positive-orientation part of $\mathcalf^s(x_0,r')\cap B(x_0,\delta)$ and $U_1^-$ is the other one. Similarly, we get $U_2^+$ and $U_2^-$. Hence one has that $$U^{\tau,\sigma}=U_1^{\tau}\cap U_2^{\sigma},\quad \tau,\sigma\in\{+,-\}, $$
   	are the four components. Note that $\mathcalf^s(x_0,r')\cap B(x_0,\delta)\setminus \{x_0\}$ is located in $U^{+,+}$ and $U^{-,-}$.
   	
   	Now, let $U_0\subset U^{+,-}$ be an open set. Again by the Local Product Structure and shorten $r$, for every point $z\in U_0$ and every  $\tilde{z}\in M_f$, 
   	the local unstable manifold $\mathcalf^u(z,\tilde{z},r)\subset B(x_0,\delta)$ intersects $\mathcalf^s(x_0,r')\cap B(x_0,\delta)$, hence it should intersect one of $U_1$ and $U_2$.
   \end{proof}

   \begin{remark}
   	As mentioned in the Subsection \ref{subsec 2.1}, if $f$ is transitive and not special, then there exists a residual set $\mathcal{R}$ on $M$ in which each point has  infinite u-directions \cite{MicenaTahzibi2016}. Combining with Lemma \ref{4 lemma: s-one accessible  open}, one has that every point $x\in \mathcal{R}$ is an interior point of $Acc(x,f)$. Recall that when $M=\TT^d$, the Anosov map $f$  is always transitive.
   \end{remark}

   The next lemma implies that every unstable manifold (relying on orbits) of $f$  is dense in $\TT^d$, when we restirct $M$ to be $\TT^d$ (also see  \cite[Lemma 8.6.1]{AokiHiraide1994}).  And this is the key point to get the $u$-accessibility of $f$ from Lemma \ref{4 lemma: s-one accessible  open}.
   
   \begin{lemma}\label{4 lemma: u-dense with uniform size}
   	Let  $f\in\mathcal{A}^1(\TT^d)$. Let $F:\RR^d\to \RR^d$ be a lifting of $f$ by $\pi:\RR^d\to \TT^d$. Then for any $\delta>0$, there exists $R>0$ such that for any $x\in\RR^d$ and $y\in \RR^d$,  $$\tildeF^u(y,R) \cap \big( B(x,\delta) +\ZZ^d \big) \neq \emptyset.$$
   	In particular, $\pi\big(\tildeF^u(y)\big)$ is dense in $\TT^d$, for all $y\in\RR^d$. 
   \end{lemma}
   
   \begin{proof}
   	Let $H:\RR^d\to \RR^d$ be the conjugacy between $F$ and its linearization $A$. By the uniform continuity of $H^{-1}$, for any $\delta >0$ there exists $\e>0$ such that $H^{-1} \Big( B\big(   H(x), \e \big)  \Big) \subseteq B(x,\delta)$, for every $x\in\RR^d$.  By Proposition \ref{2.1 prop nmH}, fix $\e>0$, there exists $m_0\in\NN$ such that for any $z\in\RR^d$ and $n\in A^{m_0}\ZZ^d$, one has
   	\begin{align}
   		d\big( H(z+n), H(z)+n\big)\leq \frac{\e}{2}.  \label{udense1}
   	\end{align}
   	On the other hand, since the  unstable foliation of $A$ is minimal on $\TT^d$, then fix $m_0\in\NN$, there exists $C=C(m_0)>0$ such that there exist $w\in\tildeL^u\big( H(y), C \big)$ and $n\in A^{m_0}\ZZ^d$ satisfying 
   	\begin{align}
   		(w+n)\in B\big(H(x),\frac{\e}{2} \big), \label{udense2}
   	\end{align}
   	for any $x,y\in\RR^d$.   Let $z=H^{-1}(w)\in \tildeF^u(y)$,  by \eqref{udense1} and \eqref{udense2}, one has $H(z+n)\in B\big(H(x),\e\big)$. It follows that $(z+n)\in B(x,\delta)$.
   	
   	Since $d\big(H(y),w \big)\leq C(m_0)$ and $\|H-{\rm Id}_{\RR^d}\|_{C^0}<C_0$, we have
   	\begin{align*}
   		d(y,z)\leq C+2C_0.
   	\end{align*}
   	Then the proof for this lemma is completed by the following claim. 
   	
   	\begin{claim}\label{4  claim: uniform control for unstable distance}
   		For any $C_1>0$ there exists $R>0$ such that for any $y\in\RR^d$ and $z\in\tildeF^u(y)$, if $d(y,z)<C_1$, then $d_{\tildeF^u}(y,z)<R$.
   	\end{claim}
   	
   	\begin{proof}[Proof of Claim \ref{4 claim: uniform control for unstable distance}]
   		We mention that for the case of dim$E^u_f=1$, the claim follows from the quasi-isometric property (see Proposition \ref{2.1 prop quasi-isometric}).
   		
   		Let $\tilde{\mathcal{F}}^u_{f,C_1}(y)$ be the local leaf containing all points in $B(y,C_1)\cap \tildeF^u(y)$. We claim that for any $\e_1>0$, there exists $m_1\in\NN$ such that for any $y\in\RR^d$ and $n\in A^{m_1}\ZZ^d$,  one has 
   		\begin{align}
   			d_{C^1}\big( \tilde{\mathcal{F}}^u_{f,C_1}(y+n)-n\;,\; \tilde{\mathcal{F}}^u_{f,C_1}(y)    \big)<\e_1, \label{eq. 2.17.2}
   		\end{align}
   		where $d_{C^1}(\cdot,\cdot)$ is the $C^1$-topology metric for treating the local unstable leaf as a $C^1$-immersion. 
   		Indeed, this claim is just a corollary of the continuity of unstable manifolds (see Proposition \ref{2.1 prop: unstable leaf continuity}), since $\pi\left({\rm Orb}_F(y+n)\right)$ is arbitrarily closed to $\pi\left({\rm Orb}_F(y)\right)$ when $m_1$ is big enough.
   		
   		Now we fix $\e_1>0$ and $m_1\in\NN$. It is clear that there exists $R'>0$ such that for any $y\in A^{m_1}\big([0,1]^d\big)$ and $z\in \tilde{\mathcal{F}}^u_{f,C_1}(y)$,  one has $d_{\tildeF^u}(y,z)<R'$. Note that for any $y'\in \RR^d$ there exist $y\in A^{m_1}\big([0,1]^d\big)$ and $n\in A^{m_1}\ZZ^d$ such that $y'=y+n$. Let $R=(1+2\e_1)R'$. By \eqref{eq. 2.17.2},  for any $y\in\RR^d$ and $z\in \tilde{\mathcal{F}}^u_{f,C_1}(y)$, we get $d_{\tildeF^u}(y,z)<R$.
   	\end{proof}
   \end{proof}

   Now we can get Proposition \ref{1 prop dichotomy of non-invertible Anosov map}. In particular, for using  easily  in Section \ref{sec 4}, we show that there exist two unstable manifolds such that  any  unstable manifold  intersects one of them as following proposition. This also implies that any two points can be connected by at most four unstable manifolds.
   
   \begin{proposition}\label{4 prop: dichotomy of non-invertible Anosov map}
   	Let $f\in\mathcal{A}^1_1(\TT^d)$.  Then, 
   	\begin{enumerate}
   		\item either $f$ is special,
   		\item or $f$ is $u$-accessible. Moreover, there exist $x_0\in \TT^d$ and $R>0$ such that for any $y\in\TT^d$ and $\tilde{y} \in \TT^d_f$ with $(\tilde{y})_0=y$, there exists  $\tilde{x} \in \TT^d_f$ with $(\tilde{x})_0=x_0$ satisfying
   		$$ \mathcalf^u(y,\tilde{y},R)\cap \mathcalf^u(x_0,\tilde{x},R)\neq \emptyset.$$
   		In particular, $\tilde{x}$ can be chosen in $\{\bar{x}^*,\bar{x}_*\}$ given by Lemma \ref{4 lemma: s-one accessible  open}. 
   	\end{enumerate}
   \end{proposition}

   \begin{proof}
   	Assume $f$ is not special, then we can take $x_0$ and $U_0$ given by Lemma \ref{4 lemma: s-one accessible  open}. For any $y\in\TT^d$,
   	and $\tilde{y} \in \TT^d_f$ with $(\tilde{y})_0=y$, we can
   	choose $y_n=\pi^{-1}(y)\in\RR^d$ be  liftings of $y$ such that $\pi\big( {\rm Orb}_F(y_n)\big) \to \tilde{y}$. Applying Lemma \ref{4 lemma: u-dense with uniform size} to $y_n$, there exists $R>0$ such that $\pi\big(\tildeF^u(y_n, R)\big)\cap U_0\neq\emptyset$. It follows that $\mathcalf^u(y,\tilde{y}, R)\cap U_0\neq\emptyset$. This complete the proof of theorem by the second item of Lemma \ref{4 lemma: s-one accessible  open}. 
   \end{proof}

\section{The topological conjugacy is smooth along  stable leaves} \label{sec 4}
In this section, we prove Theorem \ref{2 thm conjugacy implies s-rigidity} and Corollary \ref{2 cor conjugacy implies smooth}. For convenience, we restate as follow. Note that  we already have the dichotomy for $\mathcal{A}^1_1(\TT^d)$ (see Proposition \ref{1 prop dichotomy of non-invertible Anosov map} and Proposition \ref{4 prop: dichotomy of non-invertible Anosov map}). 

\begin{theorem}\label{4 thm conjugacy implies s-rigidity}
	Let $f,g\in\mathcal{N}^r_1(\TT^d)$ $(r>1)$ be conjugate via a homeomorphism $h:\TT^d\to \TT^d$. If $f$ satisfies  one of the following two conditions,
	\begin{itemize}
		\item $f$ is $u$-accessible,
		\item $f$ is special and irreducible.
	\end{itemize} 
Then $\lambda_f^s(p)=\lambda_g^s\big(h(p)\big)$,  for every $p\in{\rm Per}(f)$.  Moreover, $h$ is $C^r$-smooth along the stable foliation.
\end{theorem}

   \begin{proof}
   	We first prove that the stable periodic data of $f$ and $g$ coincide.  By the analysis in Subsection \ref{subsec 2.2}, we need only prove it  for the case that $f$ and $g$ are both $u$-accessible. Assume by contradiction that there exists $p\in{\rm Per}(f)$ such that  $\lambda^s(p,f)\neq \lambda^s\big(h(p),g\big)$. Without loss of generality, let  $\lambda^s(p,f)> \lambda^s(h(p),g)$. If $\lambda^s(p,f)< \lambda^s(h(p),g)$, we can do the same proof for $h^{-1}$. 
   
   \begin{lemma}\label{5 lemma: differential at p}
   	The restriction $h:\mathcalf^s(p) \to \mathcalg^s\big(h(p)\big)$ is differentiable at $p$, moreover $\|Dh|_{E^s_f(p)}\|=0$.
   \end{lemma}
   \begin{proof}[Proof of Lemma \ref{5 lemma: differential at p}]
   	We can assume that $p$ is a fixed point of $f$, otherwise we go through the proof replacing $f$ by $f^{n_0}$,  where $n_0$ is a period of $p$. For convenience, we denote by $f_0^{-1}(x)$, the $f$-preimage of $x\in \mathcalf^s(p)$ on $\mathcalf^s(p)$ and similar to $g_0^{-1}(x)$ for $x\in \mathcalg^s(h(p))$. 
   	
   	For short, let $\tau_f:= {\rm exp}\big(\lambda_f^s(p) \big)$ and $\tau_g:= {\rm exp}\big(\lambda_g^s(h(p)) \big)$ and $\e_0:=\tau_f-\tau_g>0$. By the adpated metric (\cite[Claim 2.18]{AGGS2022}), there exists $\delta_0>0$ such that for every $x\in\mathcalf^s(p,\delta_0)$, one has 
   	$$\|Df|_{E^s_f(f_0^{-1}(x))}\| \in [\tau_f-\frac{\e_0}{3},\tau_f+\frac{\e_0}{3}] \quad {\rm and}\quad \|Dg|_{E^s_g(g_0^{-1}\circ h(x))}\| \in [\tau_g-\frac{\e_0}{3},\tau_g+\frac{\e_0}{3}].$$
   	
   	Fix $\delta_1>0$ such that $\delta_2:=\delta_1\cdot(\tau_f-\frac{\e_0}{3})^{-1}\leq \delta_0$. Since $h$ is uniformly continuous and the local stable leaves are $C^{r}$ submanifolds, there exist $\delta_1'>0$ and $\delta_2'>0$ such that for any two points $x_0,y_0\in\mathcalf^s(p,\delta_0)$ with $d_{\mathcalf^s}(x_0,y_0)\in [\delta_1,\delta_2]$, one has $d_{\mathcalg^s}\big(h(x_0),h(y_0)\big)\in [\delta_1',\delta_2']$.  It follows that there exists $C>0$ such that $d_{\mathcalg^s}\big(h(x_0),h(y_0)\big)<C\cdot d_{\mathcalf^s}(x_0,y_0)$.
   	
   	For any $\delta<\delta_1$ and any  point $x\in\mathcalf^s(p,\delta)$, there exists $k\in\NN$ such that $$d_{\mathcalf^s}\big(f_0^{-k}(x), f_0^{-k}(p)\big)\in [\delta_1,\delta_2].$$
    Note that $k\to +\infty$ as $\delta\to 0$. Then one has,
   	\begin{align*}
   		(\tau_g+\frac{\e_0}{3})^{-k} d_{\mathcalg^s}\big(h(x),h(p)\big)
   		&\leq d_{\mathcalg^s}\big(g_0^{-k}\circ h(x),g_0^{-k}\circ h(p)\big),\\
   		&\leq C\cdot d_{\mathcalf^s}\big(f_0^{-k}(x),f_0^{-k}(p)\big),\\
   		&\leq C\cdot (\tau_f-\frac{\e_0}{3})^{-k} d_{\mathcalf^s}(x,p).
   	\end{align*}
   It follows that 
   $$\frac{d_{\mathcalg^s}\big(h(x),h(p)\big)}{d_{\mathcalf^s}(x,p)} \leq C\cdot \Big(\frac{\tau_g+\frac{\e_0}{3}}{\tau_f-\frac{\e_0}{3}} \Big)^k \to 0,$$
   as $\delta \to 0$. Hence, we get $\|Dh|_{E^s_f(p)}\|=0$.
   \end{proof}
   
  Since dim$E_{f/g}^s=1$, the holonomy maps along the unstable manifolds are smooth (Proposition \ref{2.1 prop C1 foliation}). By the second item of  Proposition \ref{4 prop: dichotomy of non-invertible Anosov map}, we can send the differentiable point $p$ to the whole $\TT^d$ so that $h$ is differentiable along the stable leaf at every point. It is convenience to complete this on the universal cover $\RR^d$. Let $H:\RR^d \to \RR^d$ be the lifting of $h$, note that
\begin{align}
	H(x+n)=H(x)+n,  \quad {\rm for\;\; every}\;\; x\in\RR^d\;\; {\rm and} \;\;n\in \ZZ^d. \label{eq. H Zd-periodic}
\end{align}
   
   \begin{lemma}\label{5 lemma: u-holonomy send differentiable}
   Assume that  $H:\tildeF^s(y)\to \tildeG^s\big(H(y)\big)$ is differentiable at some point $y\in \RR^d$. Fix $R>0$. Then for any $z\in\tildeF^u(y,R)$, any $n\in\ZZ^d$ and $w\in\tildeF^u(z+n,R)$, the map $H:\tildeF^s(w)\to \tildeG^s\big(H(w)\big)$ is differentiable at the point $w$.
   \end{lemma}
  \begin{proof}[Proof of Lemma \ref{5 lemma: u-holonomy send differentiable}]
  	For $\sigma=f$ or $g$, let Hol$^{\sigma}_{x,x'}: \tilde{\mathcal{F}}_{\sigma}^s(x)\to \tilde{\mathcal{F}}_{\sigma}^s(x')$ be the holonomy map along unstable foliation $\tilde{\mathcal{F}}_{\sigma}^u$, defined as 
  	$${\rm Hol}^{\sigma}_{x,x'}(y)=\tilde{\mathcal{F}}_{\sigma}^u(y)\cap \tilde{\mathcal{F}}_{\sigma}^s(x'), \quad \forall x,x'\in\RR^d\;\; {\rm and}\;\; y\in\tilde{\mathcal{F}}_{\sigma}^s(x).$$
  	Since $H$ maps $\tildeF^{s/u}$ to $\tildeG^{s/u}$, we have $$H\circ {\rm Hol}^f_{x,x'}(y) ={\rm Hol}^g_{H(x),H(x')}\big(H(y)\big),$$ for all $x,x'\in\RR^d$ and $y\in\tildeF^s(x)$. Recall that $T_n:\RR^d \to \RR^d$ with  $ T_n(x)=x+n$ for all $x\in\RR^d$ is the deck transformation for any $n\in\ZZ^d$.    Since $H$ is commutative with the deck transformation $T_n$ \eqref{eq. H Zd-periodic},  for any $\delta>0$, we can redefine the map $H:\tildeF^s(w,\delta) \to \tildeG^s\big(H(w)\big)$ as follow,
  	$$H(x)= {\rm Hol}^g_{H(z)+n,H(w)} \circ T_n \circ {\rm Hol}^g_{H(y),H(z)}\circ H \circ {\rm Hol}^f_{z,y}\circ T_{-n} \circ  {\rm Hol}^f_{w, z+n} (x),$$ for every $x\in \tildeF^s(w,\delta)$,  where $z\in\tildeF^u(y,R)$ and $w\in\tildeF^u(z+n,R)$ for some $n\in\ZZ^d$ and $R>0$.
  	
  	Note that within size $R$, the holonomy maps ${\rm Hol}^f_{w, z+n}$ and ${\rm Hol}^f_{z,y}$ are both differentiable by Proposition \ref{2.1 prop C1 foliation}.
  	Meanwhile, since $\|H-{\rm Id}_{\RR^d}\|_{C^0}<C_0$, one has $$d\big(H(y), H(z)\big) <R+2C_0 \quad {\rm and} \quad d\big( H(z+n),H(w)\big)<R+2C_0.$$ It  follows that there exists $R'>0$ such that   $$H(z)\in\tildeG^u\big(H(y),R'\big) \quad {\rm and} \quad H(z)+n=H(z+n)\in\tildeG^u\big(H(w),R'\big),$$ by Claim \ref{4 claim: uniform control for unstable distance}. 
 Hence the holonomy maps  $ {\rm Hol}^g_{H(z)+n,H(w)}$ and  ${\rm Hol}^g_{H(y),H(z)}$ are also differentiable. It is clear that the deck transformation $T_n$ is differentiable along the stable foliation. As a result, we get the differentiability of  $H$  at $w$ along $\tildeF^s(w)$.
  \end{proof}

 Let $p_0\in\pi^{-1}(p)$ be a lifting of $p$ and let $x_0$ and $R$ be given in Proposition \ref{4 prop: dichotomy of non-invertible Anosov map} and $x^*,x_*$ be the liftings of $x$ given in Lemma \ref{4 lemma: s-one accessible  open}. Note that $H$ is differentiable at $x^*$ along $\tildeF^s(x^*)$ if and only if that for $x_*$. Applying Proposition  \ref{4 prop: dichotomy of non-invertible Anosov map}  for $p$ and $\tilde{p}:=\pi\big({\rm Orb}_F (p_0)\big)$, there exists $\tilde{x}=\bar{x}^* \;{\rm or}\; \bar{x}_*$ such that 
 $$ \mathcalf^u(p,\tilde{p},R)\cap \mathcalf^u(x_0,\tilde{x},R)\neq \emptyset,$$
 and without loss of generality, we can take $\tilde{x}=\bar{x}_*=\pi\big({\rm Orb}_F (x_*)\big)$. Namely,  there exists $z\in\tildeF^u(p_0,R)$ and $n\in\ZZ^d$ such that $z+n\in \tildeF^u(x_*,R)$. Hence by Lemma \ref{5 lemma: u-holonomy send differentiable}, one has that  $H$ is differentiable at $x_*$ along $\tildeF^s(x_*)$. Moreover, by Lemma \ref{5 lemma: differential at p}, we have $\|Dh|_{E^s_f(x_0)}\|=0$.
 
 For any $x\in\TT^d$, by the differentiability of $h$ at $x_0$ along $\mathcalf^s(x_0)$ and repeating the above method, we get  $$\|Dh|_{E^s_f(x)}\|=0, \quad \forall x\in\TT^d.$$
 It contradicts with the fact that $h$ is a homeomorphism along the stable foliation.
 \vspace{3mm}
  
  Now we show the regularity of the conjugacy $h$ restricted on each leaf of $\mathcalf^s$.
  
  In the  $u$-accessible case, we already have $\lambda^s(p,f)=\lambda^s(h(p),g)$ for all $p\in{\rm Per}(f)$.
  Then by Proposition \ref{3 prop: bi-Lipschitz}, we have that  $h$ is bi-Lipschitz along each stable leaf. It follows that there exists $z_0\in\TT^d$ such that $h$ is bi-differentiable at $z_0$ along $\mathcalf^s(z_0)$. Combining Proposition \ref{4 prop: dichotomy of non-invertible Anosov map}  and Lemma \ref{5 lemma: u-holonomy send differentiable}, we get that $h$ is smooth along  the stable foliation. In the special case, by Theorem \ref{1 theorem AGGS}, we have that $h$ is also  smooth along the stable foliation.

  For both two case, we show that the regularity of $h|_{\mathcalf^s}$ is same as one of $f$.
  Note that each stable leaf is a $C^r$-smooth submanifold. Let 
  \begin{align}
  	 \tilde{\rho}_g\big(h(x),h(y)\big):=\frac{\|Dh|_{E^s_f(x)}\|}{\|Dh|_{E^s_f(y)}\|}\cdot \rho_f(x,y), \quad \forall y\in\mathcalf^s(x).\label{eq.4.1.0}
  \end{align}
 Since $\|Dh|_{E^s_f(f(x))}\|=\|Dg|_{E^s_g(h(x))}\|\cdot \|Dh|_{E^s_f(x)}\|\cdot\|Df|_{E^s_f(x)} \|^{-1}$, by  the  fourth item of Proposition \ref{3 prop: density function}, one has $\tilde{\rho}_g=\rho_g$. Since $\rho_f$ and $\rho_g$ are both $C^{r-1}$-smooth when restricted on  local leaves with uniform size (see   \cite[Lemma 4.3]{Llave1992} and \cite[Lemma 2.3 and Lemma 2.4]{Gogolev2017}). Fix $x\in\TT^d$ and $R>0$, one has $Dh|_{E^s_f(y)}$ is $C^{r-1}$-smooth for $y\in\mathcalf^s(x,R)$. Hence, $h$ is $C^r$-smooth along each leaf of the stable foliation.
 \end{proof}

\section{Jacobian Rigidity}\label{sec 5}
In this section, we prove  Thorem  \ref{1 thm jacobian rigidty}. For convenience, we restate it as follow. Recall that for any $r\geq 1$, let   $r_*=\Big\{ \begin{array}{lr}
	r-\e, \; \ r\in\NN  \\ \;\;\; r,\ \ \ \ \ \ \  r\notin  \NN \ {\rm or}\ r=+\infty
\end{array}$, where $\e>0$ can be arbitrarily small.  

 \begin{theorem}\label{Jacobian rigidity}
	Let $f,g \in \mathcal{N}^r_1(\TT^2)$ $(r>1)$ are topologically conjugate $h\circ f=g\circ h$, where  $h\in{\rm Home}_0(\TT^2)$.  Assume that for every periodic point $p=f^n(p)$, we have 
	$$
	{\rm Jac}\big(f^n\big)(p)={\rm Jac}\big(g^n\big)(h(p)),
	$$
	then $h$ is a $C^{r_*}$-smooth diffeomorphism. 
\end{theorem}

\begin{proof}

	By Theorem \ref{4 thm conjugacy implies s-rigidity}, we already have that $h$ is $C^r$-smooth along each stable leaf. From the following Journ$\acute{\rm e}$ Theorem, we can get the $C^{r_*}$-smoothness of $h$ by  proving that it is $C^r$-smooth along each unstable leaf.  
	\begin{lemma}[\cite{Journe1988}]\label{4 lemma journe}
		Let $M_i \; (i=1,2)$ be a smooth manifold and $\mathcal{F}_i^s, \mathcal{F}_i^u$ be continuous foliations on $M_i$ with uniformly $C^r$-smooth leaves and  intersect transversely with each other. Assume that $h:M_1\to M_2$ is a homeomorphism and maps $\mathcal{F}_1^{\sigma}$ to $\mathcal{F}^{\sigma}_2 (\sigma=s,u)$. If $h$ is uniformly $C^{r}$-smooth restricted on each leaf of both $\mathcal{F}_1^s$ and $\mathcal{F}_1^u$, then $h$ is $C^{r_*}$-smooth.
	\end{lemma}
	
	Denote the unstable Lyapunov exponent of $f$ at periodic point $p$ by $\lambda_f^u(p)$. By Theorem \ref{4 thm conjugacy implies s-rigidity} again, we  have $\lambda_f^s(p)=\lambda_g^s(h(p))$ for every $p\in{\rm Per}(f)$.  It follows from Jacobian rigidity that $\lambda_f^u(p)=\lambda_g^u\big(h(p)\big)$ for every $p\in{\rm Per}(f)$. Now we prove that $h$ is differentiable restricted on $\mathcalf^u(x,\tilde{x})$ for any $x\in\TT^2$ and  some $\tilde{x}$. More precisely, we have the following proposition which combining with Lemma \ref{4 lemma journe} implies Theorem \ref{Jacobian rigidity}.
	
	\begin{proposition}\label{4.2 prop H is C^r along u}
	Let $f,g \in \mathcal{N}^r_1(\TT^2)$ $(r>1)$ be conjugate via  $h\in{\rm Home}_0(\TT^2)$. Let $H$ be  a lifting of $h$ by the natural projection $\pi:\RR^2\to \TT^2$.   If $\lambda_f^u(p)=\lambda_g^u\big(h(p)\big)$ for every $p\in{\rm Per}(f)$. Then for any $x\in\TT^2$, there is  a  neighborhood $V(x)$ of $x$ such that  there exists a lifting $V_0(x)\subset \RR^2$ of $V(x)$ by $\pi$ satisfying that for any $y\in V_0(x)$, the restriction $H:\tildeF^u(y,\delta) \to \tildeG^u\big(H(y)\big)$ is $C^r$-smooth for some $\delta>0$.
	\end{proposition}

	\begin{proof}
		Let $F,G$ and $H$ be  liftings of $f,g$ and $h$ by the natural projection $\pi$ respectively. Define the density  function $\rho^u_F(\cdot,\cdot)$ and the affine structure $d^u_F(\cdot,\cdot)$ of $\tildeF^u$ by $$\rho^u_{F}(x,y):=\prod_{i=1}^{+\infty} \frac{\| DF|_{E^u_F(F^{-i}(y))} \|}{\| DF|_{E^u_F(F^{-i}(x))} \|}\quad {\rm and}\quad d^u_F(x,y)=\int_{x}^{y} \rho^u_F(z,x) d{\rm Leb}_{\tildeF^u}(z)$$
	where $x\in\RR^d$ and $y\in\tildeF^u(x)$. And we also define $\rho^u_G(\cdot,\cdot)$ and  $d^u_G(\cdot,\cdot)$ for $\tildeG^u$.  It is clear that these  functions satisfy all analogical properties in Proposition \ref{3 prop: density function} and Proposition \ref{3 prop: affine metric} but descending to $\TT^2$. And by th Livschitz Theorem for the inverse limit space, there exists a uniformly continuous and bounded function $P^u:\RR^2\to \RR$ such that $$\frac{\|DF|_{E^u_F(x)}\|}{\|DG|_{E^u_G(H(x))}\|}=\frac{P^u\big(F(x)\big)}{P^u(x)}, \quad \forall x\in\RR^2.$$ In particular, one has that $h$ is bi-Lipschitz along every unstable leaf from  Proposition \ref{3 prop: bi-Lipschitz}. Here we note that $\tildeF^u$ and $\tildeG^u$ are both quasi-isometric (see Proposition \ref{2.1 prop quasi-isometric}) so that the proof of Proposition \ref{3 prop: bi-Lipschitz} holds for this case. Then we have that for every local unstable leaf there exists a  full-Lebesgue set consisting of points differentiable along unstable leaf. 
	
	It suffices to prove that $H$ is differentiable along each leaf of ${\tildeF^u}$, since we can apply the same method in proof of Theorem \ref{4 thm conjugacy implies s-rigidity} (see \eqref{eq.4.1.0}) to get the $C^r$-regularity. By Proposition \ref{2.1 prop special and conjugate}, $f$ is special if and only if $g$ is special. Then we can complete the proof  by splitting in two situations: the special case and the $u$-accessible case. 
	
	Note that when $f$ is special, the differentiability of $H$ along $\tildeF^u$ can be obtained in the same way as Proposition \ref{3 prop h0 differentiable} and Corollary \ref{3 cor H is differentiable}, since the unstable foliation is well defined on $\TT^d$ in this case. We refer to \cite{GG2008} for the similar case that the differentiability along the weak unstable foliation of Anosov diffeomorphisms and to \cite{Llave1992} for the case of  unstable foliation of  Anosov diffeomorphisms on $\TT^2$. We also refer to \cite[Theorem 1.10]{Micena2022} for  the case of special Anosov endomorphisms on $\TT^2$ whose periodic data on unstable bundle is constant.

	Now we deal with the non-special case. Assume that $f$ is not special. We first  claim that there exist $x_0\in\RR^2$ and  $n\in\ZZ^2$  such that $DT_n\big( E^u_F(x_0)\big)\neq E^u_F(x_0+n)$ and $H$ is differentiable at $x_0$ restricted on $\tildeF^u(x_0)$. 
	Indeed, since $f$ is not special, there exist $z\in\TT^d$  has at least two  unstable directions. By the continuity of the unstable bundle with respect to orbits (see Proposition \ref{2.1 prop: unstable leaf continuity}), there exists a neighborhood $U$ of $z$ such that each point in $U$ has at least two unstable directions. Let $U_0$ be one of the lifting component of $U$. Since every local unstable leaf in $U_0$ has a full-Lebesgue set with points differentiable along unstable leaf, there exists $x_0\in U_0$ such that $H|_{\tildeF^u}$ is differentiable at $x_0$ and $\pi(x_0)\in U$ has at least two unstable direction. Since the projection of the $F$-orbit space on $\RR^2$ is dense in $\TT^s_f$ (see \eqref{eq. 2.23.1}) and the unstable direction is continuous with respect to orbits, there exists $n\in\ZZ^2$ such that $DT_n\big( E^u_F(x_0)\big)\neq E^u_F(x_0+n)$.
 	
 	Then we claim that there exists a neighborhood $V_0$ of $x_0$ in which $H|_{\tildeF^u}$ is differentiable  at every point. This can be done by using both  stable and unstable $C^1$-holonomy maps as follow. 
 	
 	Since $DT_n\big( E^u_F(x_0)\big)\neq E^u_F(x_0+n)$, the manifolds $\tildeF^u(x_0,\delta)$ transversely intersects with $\tildeF^u(x_0+n,\delta)-n$ at point $x_0$ for some small size $\delta>0$. Up to shrinking $\delta$, we can show that $H|_{\tildeF^u}$ is differentiable at every point of $\tildeF^u(x_0+n,\delta)-n$.  Indeed  by the continuity of local unstable leaf, there exists a neighborhood $V_1$ such that for any $y\in V_1$ and $z\in\tildeF^u(x_0+n,\delta)-n$, one has $\tildeF^u(z,\delta)$ transversely intersects with $\tildeF^u(y+n,\delta)-n$. For short, let $\tildeF^{u,n}(y):=\tildeF^u(y+n,\delta)-n$. It follows that we can define  holonomy maps along  $\tildeF^{u,n}$ between $\tildeF^u(x_0,\delta)$ and  $\tildeF^u(z,\delta)$ by $${\rm Hol}^{u,n}_F(y):=\tildeF^{u,n}(y)\cap \tildeF^u(z,\delta),$$ for any $y\in \tildeF^u(x_0,\delta)$. Since $H$ is commutative with the deck transformation $T_n$ and preserves the unstable foliation, $H$ maps  $\tildeF^{u,n}(y)$  to $\tildeG^{u,n}\big(H(y)\big):=\tildeG^u\big(H(y)+n,\e\big)-n$ for some $\e>0$. 
 	
 	\begin{figure}[htbp]
 		\centering
 		\includegraphics[width=14.2cm]{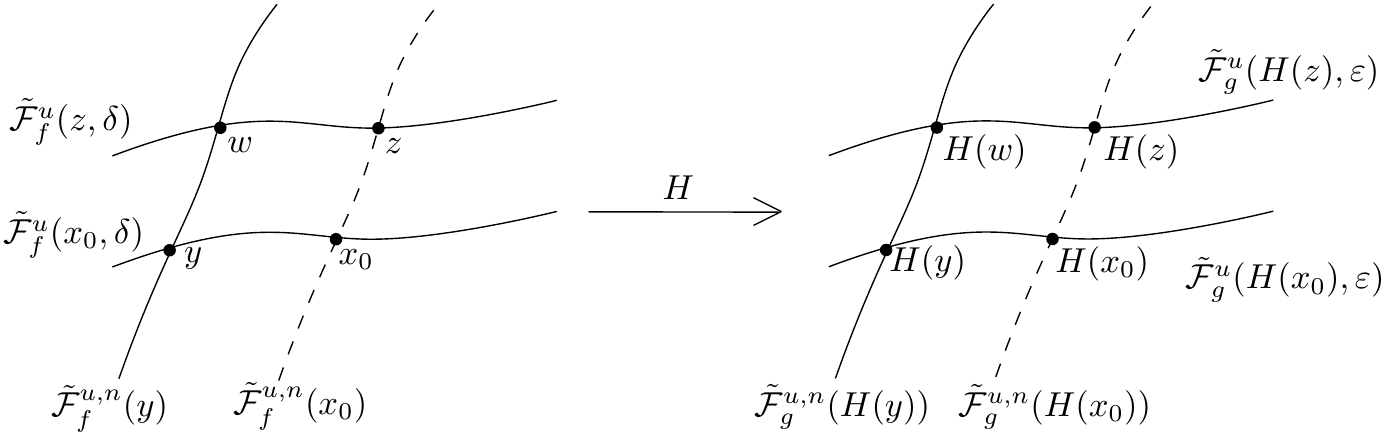}
 		\begin{center}
 			Figure 4. H is commutative with  holonomy maps induced by $\tildeF^{u,n}$ and $\tildeG^{u,n}$.
 		\end{center}
 	\end{figure}
 	
 	Note that we can assume that the manifolds $\tildeG^u(w_1,\e)$ transversely intersects with $\tildeG^{u,n}(w_2)$ for any two points $w_1,w_2\in H(V_1)$. Indeed, the homeomorphism $H$ maps the pair of transverse foliations $\tildeF^u$ and $\tildeF^{u,n}$ in $V_1$ to a pair of topologically transverse foliation $\tildeG^u$ and $\tildeG^{u,n}$ in $H(V_1)$. We can find a point $w_0\in H(V_1)$ such that $\tildeG^u(w_0,\e)$ transversely intersects with $\tildeG^{u,n}(w_0)$, or  these two foliations with $C^1$-smooth leaves are tangent everywhere  in $H(V_1)$. Taking a small neighborhood $V_2\subset H(V_1)$ of $w_0$ such that the transverse property is kept in $V_2$ and  let $H^{-1}(V_2), H^{-1}(w_0)$ replace $V_1, x_0$ respectively. Now we can define the holonomy maps ${\rm Hol}^{u,n}_G:\tildeG^u\big(H(x),\e\big)\to \tildeG^u\big(H(z),\e\big)$ given by $\tildeG^{u,n}$. Since $H$ maps $\tildeF^u$ and $\tildeF^{u,n}$ to $\tildeG^u$ and $\tildeG^{u,n}$ respectively, one has $$H(w)= {\rm Hol}^{u,n}_G \circ H \circ\big({\rm Hol}^{u,n}_F\big)^{-1}(w) : \tildeF^u(z,\delta)\to \tildeG^u\big(H(z)\big),$$ 
 	for any $w\in \tildeF^u(z,\delta)$ (see Figure 4). Then we get that $H|_{\tildeF^u}$ is differentiable at any point $z\in \tildeF^{u,n}(x_0)$, since both ${\rm Hol}^{u,n}_F$ and ${\rm Hol}^{u,n}_G$ are $C^1$-smooth by Proposition \ref{2.1 prop C1 foliation},

 	By the same way, since the  honolomy map along $\tildeF^s$ is $C^1$-smooth and $H$ maps $\tildeF^{s/u}$ to $\tildeG^{s/u}$, if $H|_{\tildeF^u}$ is differentiable at $x\in\RR^2$, then it is also differentiable at every point in $\tildeF^s(x,R)$ for any $R>0$. Hence, for any $z\in \tildeF^{u,n}(x)$ and $w\in\tildeF^s(z,R)$, one has $H|_{\tildeF^u}$ is differentiable. Since  $\tildeF^s$ is commutative with the deck transformation $T_n$, one has that $\tildeF^{u,n}$ and $\tildeF^s$ also admit the Local Product Structure and by this we get the neighborhood $V_0$ of the point $x_0$ in the above claim. 

   Moreover, for any $x\in V_0$ and $R>0$, $H|_{\tildeF^u}$ is differentiable at any point $y\in\tildeF^s(x,R)$. Since the foliation $\mathcalf^s$ is minimal by Proposition \ref{2.2 prop leaf conjugate}, we get Proposition \ref{4.2 prop H is C^r along u} for non-special case by projecting these differentiable points on $\RR^2$ to $\TT^2$. 	
	\end{proof}
	\end{proof}
	
	\begin{remark}
		We still do not know that if the Jacobian rigidity implies the smoothness of conjugacy for Anosov diffeomorphisms on $\TT^2$. 
		More precisely, let $f,g:\TT^2\to \TT^2$ be two $C^{r}$ non-conservative  Anosov diffeomorphisms conjugate via $h$ and ${\rm Jac}\big(f^n\big)(p)={\rm Jac}\big(g^n\big)(h(p))$ for every $p\in {\rm  Per}(f)$ with period $n$, if $h$ is  smooth or not?  
	\end{remark}
	
	\section{Appendix: The differentiability of $\bar{h}_0$}
	
	  We will prove Proposition \ref{3 prop h0 differentiable} in this Appendix. Although as mentioned before the proof has a similar  frame with one in \cite{GG2008}, we would like to point out the Lemma \ref{3 lemma full measure preimage dense} and Lemma \ref{3 lem differentiable transitive point} in the following proof have different taste with the orginal proof. For short, denote by $m^s$ the Lebesgue measure Leb$_{\mathcalf^s}$ on each leaf of $\mathcalf^s$.

	\begin{proof}[Proof of Proposition \ref{3 prop h0 differentiable}]
		We prove it by two steps. First of all, we show that there is a point $\tilde{z}\in\TT^d$ such that $\tilde{z}$ is a transitive point of $\sigma_f^{-1}:\TT^d_f\to\TT^d_f$ and $D\bar{h}_0|_{E^s_f(\tilde{z})}$ exists. Then by the transitive property which is in fact a point $\pi_0(\tilde{z})$ with dense preimage, we will get the differentiability of any other points $\tilde{y}\in\TT^d_f$. 
		
		For the first step, we construct a measure $\mu$ on $\TT^d$ whose conditional measure on local stable foliation $\mathcalf^s$ is absolutely continuous to the Lebesgue measure $m^s$, and for $\mu$-a.e. $z\in\TT^d$ the preimage of $z$ is dense in $\TT^d$. We mention that the construction of $\mu$ is quite similar to one in \cite{GG2008}. However to prove the $\sigma^{-1}_f$-transitivity need more discussion.
		
		Let $x_0\in\TT^d$ be a fixed point of $f$. Denote by $f^{-1}_0:\mathcalf^s(x_0)\to \mathcalf^s(x_0)$, the inverse of $f$ restricted on  $\mathcalf^s(x_0)$. Fix $\delta_0>0$, denote $V_0=\mathcalf^s(x_0,\delta_0)$ an open neighborhood of $x_0$ on $\mathcalf^s(x_0)$. Let $\nu^0\in \mathcal{M}(\TT^d)$, the set of Borel probability  measure on $\TT^d$, be supported on $V_0$ and define as 
		$$\nu^0(A)=\frac{\int_{A\cap V_0} \rho_f(x_0,z)dm^s(z)}{\int_{ V_0} \rho_f(x_0,z)dm^s(z)},
		\qquad \text{where} \qquad
		\rho_f(x_0,z)=\prod_{i=0}^{+\infty} \frac{\| Df|_{E^s_f(f^i(z))} \|}{\| Df|_{E^s_f(f^i(x_0))} \|},
		$$ 
		for any measurable set $A\subset \TT^d$. For any $n\in\NN$, let $V_n=f_0^{-n}(V_0)$ and $\nu^n=(f_0^{-n})_{*}\nu^0$. Then one has $\nu^n$ is supported  on $V_n$ and 
		$$\nu^n(A)=\frac{\int_{A\cap V_n} \rho_f(x_0,z)dm^s(z)}{\int_{ V_n} \rho_f(x_0,z)dm^s(z)},$$ for any measurable set $A\subset \TT^d$, by the affine sturcture (the third item of Proposition \ref{3 prop: density function}). 
		
		Let $\mu^n= \frac{1}{n}\sum_{i=0}^{n-1}\nu^i$. By the compactness of $\mathcal{M}(\TT^d)$, there is a subsequence of $\{\mu^{n_k} \}$ convergent in $\mathcal{M}(\TT^d)$. For short, we denote by $\mu=\lim_{n\to+\infty}\mu^n$. Note that $\mu$ is $f$-invariant.  Indeed, since $\nu^{i}$ is supported on $V_{i}$ for all $i\geq 0$ and the preimages of each local stable leaf are disjoint local stable leaves, we have $\nu^i=(f_0)_*\nu^{i+1}=f_*\nu^{i+1}$. Hence for every continuous function $\phi:\TT^d\to \RR$, one has 
		\begin{align*}
				\Big| \mu^n(\phi)-f_{*}\mu^n(\phi)\Big|&=\Big|\frac{1}{n}\sum_{i=0}^{n-1}\nu^i(\phi)- \frac{1}{n}\sum_{i=0}^{n-1}f_{*}\nu^i(\phi)\Big|\\
				&=\frac{|\nu^{n-1}(\phi)-f_*\nu^0(\phi)|}{n} \\
				&\leq \frac{2}{n}\|\phi\|_{C^0}.
		\end{align*}
	 It follows that $\mu$ is $f$-invariant.
	
		\begin{lemma}\label{3 lemma AC measure}
			The conditional  measure of $\mu$ on each local leaf of $\mathcalf^s$ is absolutely continuous with respect  to the Lebesgue measure $m^s$.
		\end{lemma}
		\begin{proof}[Proof of Lemma \ref{3 lemma AC measure}]
			The proof is actually as same as the proof of  \cite[Lemma 5]{GG2008}. Hence we just skecth our proof here. For small $R>0$, consider any small $\mathcalf^s$ foliation box $D=\bigcup_{x\in T} \mathcalf^s(x,R)\subset \TT^d$, where $T$ is a $(n-1)$-dimensional transversal of $\mathcalf^s$ in $D$. Let $p:D\to T$ be the projection along each local leaf of $\mathcalf^s$ in $D$. Denote by $\mu_T:=p_*(\mu)$, the transverse measure on $T$.  Let $\mu_x$ be the conditional measure of $\mu$ on leaf $\mathcalf^s(x,R)$ for $x\in D$. Also we can define $\nu^n_T,\nu^n_x$ and $\mu^n_T, \mu^n_x$.
			\begin{enumerate}
				\item Considering the intersection of $V_n$ with $D$, we can only deal with the case that the end points of $V_n$ lie outside $D$. Let $\{a_1,...,a_m\}= V_n\cap T$. Then we have 
				$$\nu^n_T=\sum_{i=1}^{m}\Big( \int_{\mathcalf^s(a_i,R)} \rho_f(x_0, x) d m^s_{a_i}(x)\Big)\cdot \delta(a_i),$$
				and
				$$d\nu_x^n(y)=\sum_{i=1}^{m}\Big( \int_{\mathcalf^s(x,R)} \rho_f(x, y) d m^s_{x}(y)\Big)^{-1} \rho_f(x,y) d m^s_x(y).$$
			And the transverse measure $$\mu_T=\lim_{n\to+\infty}\mu^n_T=\lim_{n\to+\infty}\frac{1}{n}\sum_{i=0}^{n-1}\nu^i_T.$$
			\item By the characterize of conditional measure and the first item, let $\phi:D\to \RR$ be any continuous function, we calculate directly,
			\begin{align*}
			&\int_{T}d\mu_T(x)\int_{\mathcalf^s(x,R)}\phi(x,y)d\mu_x(y)=	\int_{D}\phi d\mu\\
			&= \lim_{n\to +\infty}	\int_{D}\phi d\mu^n=\lim_{n\to+\infty}\frac{1}{n}\sum_{i=0}^{n-1}\int_{D}\phi d\nu^n\\
				&=\lim_{n\to+\infty}\frac{1}{n}\sum_{i=0}^{n-1}\int_{T}d\nu^n_T(x)\int_{\mathcalf^s(x,R)}\phi(x,y)d\nu^n_x(y)\\
				&=\int_{T}d\mu_T(x) \Big( \int_{\mathcalf^s(x,R)} \rho_f(x, y) d m^s_{x}(y)\Big)^{-1}  \int_{\mathcalf^s(x,R)}\phi(x,y) \rho_f(x,y)dm_x(y)
			\end{align*}
		Hence for $\mu_T$-a.e. $x\in T$, the conditional measure $\mu_x$ is absolutely continuous with respect to $m^s$ on $\mathcalf^s(x,R)$ with the density $ \Big( \int_{\mathcalf^s(x,R)} \rho_f(x, y) d m^s_{x}(y)\Big)^{-1}\cdot \rho_f(x,y)$
		\item Recall that $\mathcalf^s$ is minimal  on $\TT^d$ since $f$ is irreducible and leaf conjugates to its linearization. Hence $\mu$ is supported on the whole $\TT^d$.
			\end{enumerate}

		\end{proof}

		By a standard method (see \cite[Proposition I.3.1]{QXZ2009}),  there is a unique $\sigma_f$-invariant measure $\tilde{\mu}\in \mathcal{M}(\TT^d_f)$ such that  $\mu=(\pi_0)_{*}(\tilde{\mu})$. We mention that $$\tilde{\mu}\big([A_0,A_1,\cdots,A_n]\big)=\mu\big(\bigcap_{i=0}^{n}f^{-i}(A_i) \big),$$ for any cylinder $[A_0,A_1,\cdots,A_n]\subset \TT^d_f$.
		
		\begin{lemma}\label{3 lemma full measure preimage dense}
			There is a set $\Lambda\subset\TT^d_f$ with $\tilde{\mu}(\Lambda)=1$ such that for any $\tilde{z}\in\Lambda$, $\tilde{z}$ is $\sigma_f^{-1}$-transitive point, i.e., the negative orbit $\{\sigma_f^{-n}(\tilde{z})\}_{n\in\NN}$ is dense in $\TT^d_f$. 
		\end{lemma}
		\begin{proof}[Proof of Lemma \ref{3 lemma full measure preimage dense}]
			 It suffices to show that for any open ball $B\subset \TT^d_f$ with radius $\frac{1}{m}\ (m\in \NN)$ which is actually a cylinder $[B_{-k},\cdots, B_0,\cdots,  B_k]$, where $B_i$ is an open ball in $\TT^d$,  one has that $\tilde{\mu}$-a.e. points in $\TT^d_f$ go through $B$ infinity times by $\sigma_f^{-1}$. Then covering $\TT^d_f$ by $\frac{1}{m}$-ball and let $m$ go to infinity, one can get the full-$\tilde{\mu}$ transitivity.
			
			Taking an open ball $B=[B_{-k},\cdots, B_0,\cdots,  B_k] \subset \TT^d_f$.  For convenience, we see  the set $B_0=\pi_0(B)$, an open ball on $\TT^d$, as a foliation box of $\mathcalf^s$. Since $\mathcalf^s$ is minimal, there exists $R>0$ such that $\TT^d\subseteq \bigcup_{x\in B_0}\mathcalf^s(x,R)$.  Let $\psi:\TT^d_f\to \RR$ be  a continuous function such that $\psi=1$ on $B$ and supported on a small neighborhood of $B$. By Birkhoof Ergodic Theorem,
			$$\psi^+=\lim_{n\to+\infty} \frac{1}{n}\sum_{i=0}^{n-1} \psi\circ \sigma_f^{i}\ \ \ \in {\rm L}^1(\TT^d_f),$$
			and $\int\psi^+ d\tilde{\mu}= \int \psi d\tilde{\mu}$. Let $A=\{\tilde{x}\in\TT^d_f \ |\  \psi^+(\tilde{x})=0\}$. Note that $A$ is $\sigma_f$-invariant.
			By Dominated Convergence Theorem and applying Birkhoof Ergodic Theorem to $\psi\cdot \mathbb{1}_A$ where $\mathbb{1}_A$ is the characteristic function of $A$, one has that 
			\begin{align*}
				\int_A \psi d\tilde{\mu}&=\int \psi\cdot \mathbb{1}_Ad\tilde{\mu}=\int \lim_{n\to+\infty} \frac{1}{n}\sum_{i=0}^{n-1} (\psi\cdot \mathbb{1}_A)\circ \sigma_f^{i} d\tilde{\mu}\\
				&=\lim_{n\to+\infty} \frac{1}{n}\sum_{i=0}^{n-1} \int  \big(\psi \circ \sigma_f^{i}\big) \cdot \big(\mathbb{1}_A\circ \sigma_f^{i}\big) d\tilde{\mu}\\
				&=\lim_{n\to+\infty} \frac{1}{n}\sum_{i=0}^{n-1} \int_A  \psi \circ \sigma_f^{i}d\tilde{\mu}
				= \int_A\psi^+d\tilde{\mu}=0.
			\end{align*}
			It follows that $\tilde{\mu}(A\cap B)=0$ and hence $\psi^+(\tilde{x})>0$ for $\tilde{\mu}-$a.e. $\tilde{x}\in B$. It means that for $\mu$-a.e. $x_0\in B_0$, there exists $\bar{x}\in B$ with $\pi_0(\bar{x})=x_0$ such that  $\psi^+(\bar{x})>0$, by the construction of $B$ and the fact that $\mu$ is supported on $\TT^d$.
			Note that for any  $\tilde{x}\in\TT^d_f$ with $\pi_0(\tilde{x})=x_0$, one has $\psi^+(\tilde{x})=\psi^+(\bar{x})>0$. That implies the function $\psi^+$ is in fact independent with the negative orbit. It follows that for any $y_0\in \mathcalf^s(x_0,R)$ and any $\tilde{y}\in\TT^d_f$ with $\pi_0(\tilde{y})=y_0$, one has $\psi^+(\tilde{y})=\psi^+(\tilde{x})>0$. Applying the Hopf argument, for $\mu$-a.e. $z\in\TT^d$ and any $\tilde{z}\in\TT^d_f$ with $\pi_0(\tilde{z})=z$,  we have $\psi^+(\tilde{z})>0$ since the conditional measure of $\mu$ on each local stable leaf is absolutely continuous to the Lebesgue measure $m^s$. 
			Denote the set of such $z\in\TT^d$ by $C_0$, again $\mu(C_0)=1$. Let $C=\{\tilde{z}\in\TT^d_f\ |\ \pi_0(\tilde{z})\in C_0 \}$. Then one has $\tilde{\mu}(C)=\lim_{n\to +\infty}\mu\big(\pi_{-n}(C)\big)= \lim_{n\to +\infty}\mu\big(f^{-n}(C_0)\big)=1$. Again for any $\tilde{z}\in C$, $\psi^+(\tilde{z})>0$.
		
			By Birkhoof Ergodic Theorem again, there is a set $\Gamma\subset\TT^d_f$ with $\tilde{\mu}(\Gamma)=1$ such that 
			$$\psi^-(\tilde{x}):=\lim_{n\to+\infty} \frac{1}{n}\sum_{i=0}^{n-1} \psi\circ \sigma_f^{-i}(\tilde{x}) = \psi^+(\tilde{x}), \quad \forall \tilde{x}\in \Gamma.$$
			Then  $\tilde{\mu}(\Gamma\cap C)=1$ and for and $\tilde{z}\in \Gamma\cap C$, one has $\psi^-(\tilde{z})=\psi^+(\tilde{z})>0$.  It means that $\tilde{\mu}$-a.e. points go through $B$ infinity times by $\sigma_f^{-1}$. 
		\end{proof}

	  \begin{lemma}\label{3 lem differentiable transitive point}
	  	There is a point $\tilde{z}\in \TT^d_f$ which is a transitive point of $\sigma_f^{-1}$ and $D\bar{h}_0|_{E^s_f(\tilde{z})}$ exists.  
	  \end{lemma}
		
		\begin{proof}[Proof of Lemma \ref{3 lem differentiable transitive point}]
		On the one hand let us consider the set $\mathcal{R}=\pi \big( {\rm Orb}_F(\RR^d) \big)\subset \TT^d_f$, the projection of all orbits of $F$ on the universal cover $\RR^d$. Recall that the closure of $\mathcal{R}$ equals to the whole $\TT^d_f$. Moreover, the cylinder of $\mathcal{R}$  from $(-n)$-th to $0$-th position is $\mathcal{R}_n=\big[\TT^d,\cdots,\TT^d \big]_n$ since $\pi_{-i}(\mathcal{R})=\TT^d$ for all $i\in\NN$. Then we have $\tilde{\mu}(\mathcal{R})=\lim_{n\to +\infty}\tilde{\mu}(\mathcal{R}_n)=\lim_{n\to +\infty}\mu (\TT^d)=1$.  It follows that $\tilde{\mu}(\mathcal{R}\cap \Lambda)=1$, where $\Lambda$ is given by Lemma \ref{3 lemma full measure preimage dense}. Hence $\mu\big( \pi_0( \mathcal{R}\cap \Lambda )\big)=1$.
		Let $A_0$ be a small stable foliation box on $\TT^d$. Then by the absolute continuity of the measure $\mu$ along $\mathcalf^s$, one has that there exists a point $w\in A_0$ and a set $W^*\subset \mathcalf^s(w,R)$ with $m^s(W^*)=1$ such that for all $x\in W^*$ there exists $\tilde{x}\in \mathcal{R}\cap \Lambda$ with $\pi_0(\tilde{x})=x$, where $\mathcalf^s(w,R)$  is the component  of local stable leaf in $A_0$ containing $w$.
		
		On the other hand, let us fix the above leaf $\mathcalf^s(w,R)$ and fix the lifting $w_0\in\RR^d$ of $w$. By Proposition \ref{3 prop: bi-Lipschitz} or Corollary \ref{3 cor h0 almost differentiable}, for any $k\in\ZZ^d$ the differentiable points of $H$ restricted on $\tildeF^s(w_0+k,R)$ consists a full-$m^s$ set denoted by $W_k$. Then we have that $W=\bigcap_{k\in\ZZ^d}\{W_k-k\}\subset \tildeF^s(w_0,R)$  and $m^s(W)=1$ since $\tildeF^s(x+k)=\tildeF^s(x)+k$. It means that there is a full-$m_s$ set $W_*=\pi(W)\subset \mathcalf^s(w,R)$ such that for any $x\in W_*$ and any orbit $\tilde{x}\in \mathcal{R}$ of $x$, the derivative $D\bar{h}_0|_{E^s_f(\tilde{x})}$ exists.  Then we have $m^s(W^*\cap W_*)=1$. Hence there exists $z\in W^*\cap W_*$. This means that there exists $\tilde{z}\in \mathcal{R}\cap \Lambda$ and $D\bar{h}_0|_{E^s_f(\tilde{z})}$ exists.  Note that by Lemma \ref{3 lemma full measure preimage dense}, $\tilde{z}$ is a $\sigma_f^{-1}$-transitive point.
	\end{proof}
	
		\vspace{2mm}
		Let $\tilde{z}\in \TT^d_f$ be in Lemma \ref{3 lem differentiable transitive point}. Take any $\tilde{y}\in\TT^d_f\setminus\{\tilde{z}\}$, we prove that  $\bar{h}_0$ is differentiable on $E^s_f(\tilde{y})$ and $$\|D\bar{h}_0|_{E^s_f(\tilde{y})}\|=\frac{\tilde{P}(\tilde{z})}{\tilde{P}(\tilde{y})}\cdot\|D\bar{h}_0|_{E^s_f(\tilde{z})}\|,\quad \forall \tilde{y}\in\TT^d_f.$$
		The proof of this part is similar to one in \cite{GG2008} except that  we deal with it on the inverse limit space. For completeness, we prove it as following.
		
		Let $\tilde{y}=(y_i)\in\TT^d_f$ and $\tilde{y}'=(y_i')\in \mathcalf^s(\tilde{y},R)\setminus\{\tilde{y}\}$ for some $R>0$. Assume the orientation $y_0\prec y'_0$ on $\mathcalf^s$. Take $D\subset \TT^d_f$, a neighborhood of $\tilde{y}$. Let $D'\subset\TT^d_f$ be a neighborhood of $\tilde{y}'$  such that for any $\tilde{x}'=(x_i')\in D'$, there exists $\tilde{x}=(x_i)\in D$ with $x_0\prec x'_0$ satisfying $d^s_f(x_0,x'_0)=d^s_f(y_0,y'_0)$, where $d^s_f(\cdot,\cdot)$ is the affine structure on $\mathcalf^s$ given by Proposition \ref{3 prop: affine metric}.  
		
		For any $\e>0$, by the uniform continuity of $\bar{h}$ and $\tilde{P}$, we can take $D$ small enough such that 
		\begin{itemize}
			\item For any points $\tilde{x}=(x_i)\in D, \tilde{x}'=(x_i')\in D'$ with $x_0\prec x'_0$ and $d^s_f(x_0,x_0')=d^s_f(y_0,y_0')$, one has that $\Big|d^s_g\big(\bar{h}_0(\tilde{x}), \bar{h}_0(\tilde{x}')\big) \ - \  d^s_g\big(\bar{h}_0(\tilde{y}), \bar{h}_0(\tilde{y}')\big)\Big|<\e$.
			\item For any $\tilde{x}=(x_i)\in D$, $|\tilde{P}(\tilde{x})-\tilde{P}(\tilde{y})|<\e$.
		\end{itemize}
		
		Since $\tilde{z}=(z_i)$ is a $\sigma_f^{-1}$-transitive point, for any $N>0$ there exists $n>N$ such that $\sigma_f^{-n}(\tilde{z})\in D$. Take $\tilde{z}'=(z_i')\in \mathcalf^s(\tilde{z},R)$ such that $z_{-n}\prec z_{-n}'$ and $d^s_f(z_{-n},z_{-n}')=d^s_f(y_0,y_0')$. Note that $\sigma_f^{-n}(\tilde{z}')\in D'$.
		For given $\e>0$, let $N$ big enough such that for $n>N$, the point $\tilde{z}'$ satisfies
		$$\Big|\frac{d^s_g\big(\bar{h}_0(\tilde{z}), \bar{h}_0(\tilde{z}')\big)}{d^s_f(z_0,z_0')} \ -\  \|D\bar{h}_0|_{E^s_f(\tilde{z})}\|\Big|<\e.$$
		Now we fix this $n$. By the affine structure that is the second item of Proposition \ref{3 prop: affine metric} and Livschitz Theorem (see Lemma \ref{3 lem: s rigidity induce to RRd}), one has that 
		\begin{align*}
			\frac{d^s_g\big(\bar{h}_0(\tilde{z}), \bar{h}_0(\tilde{z}')\big)}{d^s_f(z_0,z_0')} &= \prod_{i=1}^{n}\frac{\|Dg|_{E^s_g(\bar{h}_{-i}(\tilde{z}))}\|}{\|Df|_{E^s_f(z_{-i})}\|} \cdot \frac{d^s_g\big(\bar{h}_{-n}(\tilde{z}), \bar{h}_{-n}(\tilde{z}')\big)}{d^s_f(z_{-n},z_{-n}')}, \\
			&=\frac{\tilde{P}\big(\sigma_f^{-n}(\tilde{z})\big)}{\tilde{P}(\tilde{z})} \cdot \frac{d^s_g\big(\bar{h}_{-n}(\tilde{z}), \bar{h}_{-n}(\tilde{z}')\big)}{d^s_f(z_{-n},z_{-n}')},
		\end{align*}
		where $\bar{h}_{-n}=\pi_{-n}\circ \bar{h}$ and $\pi_{-n}:(\TT^d)^{\ZZ}\to \TT^d$ is the projection to the $(-n)$-th position. 
		Note that $$\bar{h}_{-n}=\pi_0\circ \sigma_g^{-n}\circ \bar{h}=\pi_0\circ \bar{h}\circ \sigma_f^{-n}=\bar{h}_0\circ \sigma_f^{-n}.$$
		Since $\sigma_f^{-n}(\tilde{z})\in D, \sigma_f^{-n}(\tilde{z}')\in D'$ and
		$d_f^s(z_{-n},z_{-n}')=d^s_f(y_0,y_0')$, we have that 
		$$\Big|d^s_g\big(\bar{h}_{-n}(\tilde{z}), \bar{h}_{-n}(\tilde{z}')\big) \ - \  d^s_g\big(\bar{h}_0(\tilde{y}), \bar{h}_0(\tilde{y}')\big)\Big|<\e\quad {\rm and}\quad \Big|\tilde{P}(\sigma_f^{-n}(\tilde{z}))-\tilde{P}(\tilde{y})\Big|<\e.$$
	 Recall that there exists $C>1$ and $C^{-1}<\tilde{P}(\tilde{x})<C$ for any $\tilde{x}\in\TT^d_f$. It follows that 
		\begin{align*}
			&\Big|    \frac{\tilde{P}(\tilde{y})}{\tilde{P}(\tilde{z})} \cdot \frac{d^s_g\big(\bar{h}_{0}(\tilde{y}), \bar{h}_{0}(\tilde{y}')\big)}{d^s_f(y_{0},y_{0}')}\ -\  \frac{\tilde{P}\big(\sigma_f^{-n}(\tilde{z})\big)}{\tilde{P}(\tilde{z})} \cdot \frac{d^s_g\big(\bar{h}_{-n}(\tilde{z}), \bar{h}_{-n}(\tilde{z}')\big)}{d^s_f(z_{-n},z_{-n}')}      \Big|\\
			&=\Big|  \frac{ \tilde{P}(\tilde{y})\cdot d^s_g\big(\bar{h}_{0}(\tilde{y}), \bar{h}_{0}(\tilde{y}')\big)-  \tilde{P}\big(\sigma_f^{-n}(\tilde{z})\big)\cdot d^s_g\big(\bar{h}_{-n}(\tilde{z}), \bar{h}_{-n}(\tilde{z}')\big)}{\tilde{P}(\tilde{z})\cdot d^s_f(y_{0},y_{0}')}\      \Big|\\
			&<\frac{d^s_g\big(\bar{h}_{0}(\tilde{y}), \bar{h}_{0}(\tilde{y}')\big) + C}{d^s_f(y_0,y_0')}\cdot C\cdot \e.
		\end{align*}
	Since $\tilde{y}'\in\TT^d_f$ is given in advance, one has that 
	$$\Big|    \frac{\tilde{P}(\tilde{y})}{\tilde{P}(\tilde{z})} \cdot \frac{d^s_g\big(\bar{h}_{0}(\tilde{y}), \bar{h}_{0}(\tilde{y}')\big)}{d^s_f(y_{0},y_{0}')}\ -\   \|D\bar{h}_0|_{E^s_f(\tilde{z})}\|\Big| \to 0, $$
	as $\e\to 0$. Let $\tilde{y}'\to \tilde{y}$, we have that $$\|D\bar{h}_0|_{E^s_f(\tilde{y})}\|=\frac{\tilde{P}(\tilde{z})}{\tilde{P}(\tilde{y})}\cdot\|D\bar{h}_0|_{E^s_f(\tilde{z})}\|,\quad \forall \tilde{y}\in\TT^d_f.$$
	\end{proof}

	\bibliographystyle{plain}

\bibliography{ref}

   	\flushleft{\bf Ruihao Gu} \\
   School of Mathematical Sciences, Peking University, Beijing, 100871,  China\\
   \textit{E-mail:} \texttt{rhgu@pku.edu.cn}\\
   
   \flushleft{\bf Yi Shi} \\
   School of Mathematical Sciences, Peking University, Beijing, 100871,  China\\
   \textit{E-mail:} \texttt{shiyi@math.pku.edu.cn}\\

\end{document}